\documentclass[11pt,a4paper]{article}
\usepackage[a4paper,left=2.5cm,right=2.5cm,top=2.5cm,bottom=2.5cm]{geometry} 
 
\usepackage{amsmath,amssymb,amsthm,amsfonts,subcaption,caption}
\usepackage[colorlinks=True,
            linkcolor=blue,
            anchorcolor=green,
            citecolor=blue
            ]{hyperref} 
\usepackage[noabbrev,capitalise,nameinlink]{cleveref}
\usepackage{float}
\usepackage{subfloat}
\usepackage{graphicx}
\usepackage{epstopdf}
\usepackage{multirow} 
\usepackage{enumitem,url}
\usepackage{multirow}
\usepackage{booktabs}
 
\usepackage{subcaption}
\usepackage{amssymb}
\usepackage{bm}
\usepackage{hyperref}
\usepackage{mathtools}

\usepackage{tikz}
 \usetikzlibrary{arrows.meta, positioning, calc, decorations.markings}

\usepackage[ruled,vlined]{algorithm2e}
\usepackage{etoolbox}

 \newtheorem{theorem}{Theorem}[section]

\newtheorem{lemma}{Lemma}[section]
\newtheorem{proposition}{Proposition}[section]
\newtheorem{corollary}{Corollary}[section]
\newtheorem{definition}{Definition}[section]
\newtheorem{remark}{Remark}[section]
\newtheorem{assumption}{Assumption}[section]

\newtheorem{example}{Example}[section]

\usepackage{hyperref}
\newcommand{\footremember}[2]{%
    \footnote{#2}
    \newcounter{#1}
    \setcounter{#1}{\value{footnote}}%
}

\graphicspath{ {./figures/} }

\def\R{\mathbb{R}}
\def\H{\mathbf{H}}
\def\Q{\mathbf{Q}}
\def\G{\mathbf{G}}
\def\A{\mathbf{A}}
\def\B{\mathbf{B}}
\def\C{\mathbf{C}}
\def\D{\mathbf{D}}
\def\I{\mathbf{I}}
\def\W{\mathbf{W}}
\def\U{\mathbf{U}}
\def\Y{\mathbf{Y}}
\def\Z{\mathbf{Z}}
\def\V{\mathbf{V}}
\def\X{\mathbb{X}}
\def\T{\mathbb{T}}
\def\F{\mathbb{F}}
\def\E{\mathbb{E}}

\def\N{\mathbb{N}^*}

\def\bx{{\bf x}}
\def\bq{{\bf q}}
\def\by{{\bf y}}
\def\bz{{\bf z}}
\def\bu{{\bf u}}
\def\bd{{\bf d}}
\def\bv{{\bf v}}
\def\bw{{\bf w}}
\def\bb{{\bf b }}
\def\ba{{\bf a }}

\def\bmu{{\bm \mu }}
\def\blambda{{\bm \lambda }}
\def\bnu{{\bm \nu }}
\def\supp{{\rm supp }}

\def\bg{{\boldsymbol g }}

\graphicspath{ {./Figures/} }

\author{
	Shuai Li\footremember{1}{School of Mathematics and Statistics, Beijing Jiaotong University, Beijing, 100044, China. 
	\\ \indent\indent Emails: 24110488@bjtu.edu.cn, shlzhou@bjtu.edu.cn, zyluo@bjtu.edu.cn.} 
	\and Shenglong Zhou\footnotemark[1] 
	 \footremember{2}{Corresponding author.} 
	\and Ziyan Luo\footnotemark[1] 
}

\title{\vspace{-1.25cm}
Sparse Quadratically Constrained Quadratic Programming via Semismooth Newton Method 
\vspace{-0.0cm}}

\date{}

\begin{document}
\flushbottom
 
\maketitle
 
\vspace{-1.5cm}

\begin{abstract}  
\noindent \textbf{Abstract:}  Quadratically constrained quadratic programming (QCQP) has long been recognized as a computationally challenging problem, particularly in large-scale or high-dimensional settings where solving it directly becomes intractable. The complexity further escalates when a sparsity constraint is involved, giving rise to the problem of sparse QCQP (SQCQP), which makes conventional solution methods even less effective. Existing approaches for solving SQCQP typically rely on mixed-integer programming formulations, relaxation techniques, or greedy heuristics but often suffer from computational inefficiency and limited accuracy.  In this work, we introduce a novel paradigm by designing an efficient algorithm that directly addresses SQCQP. To be more specific, we introduce P-stationarity to establish first- and second-order optimality conditions of the original problem, leading to a system of nonlinear equations whose generalized Jacobian is proven to be nonsingular under mild assumptions. Most importantly, these equations facilitate the development of a semismooth Newton-type method that exhibits significantly low computational complexity due to the sparsity constraint and achieves a locally quadratic convergence rate. Finally, extensive numerical experiments validate the accuracy and computational efficiency of the algorithm compared to several established solvers.


\vspace{0.5cm}
 
\end{abstract}

\numberwithin{equation}{section}

\maketitle

\section{Introduction}
\label{Introduction}
We aim to solve the following sparse quadratically constrained quadratic programming (SQCQP),
\begin{equation}
	\label{eq: SCQP}
	\begin{array}{lllll}	
		\underset{\bx\in \mathbb{R} ^n}{\min}  &\dfrac{1}{2}\bx^{\top} \Q_{0}\bx+\bq_{0}^{\top}\bx +c_0\\[1.5ex]
		\mathrm{s.t.}  &\dfrac{1}{2}\bx^{\top} \Q_{i}\bx+\bq_{i}^{\top}\bx + c_i\leq  0,~i=1,2,\ldots,k,\\ [1.5ex]
		  &\A\bx-\bb\leq  0,~  \| \bx \|_0 \le s, ~\bx \in \X,
\end{array}\tag{SQCQP}
\end{equation}
where $\mathbf{Q}_{i}$ is an $n$-order symmetric matrix, ${\bq_i\in\R^n}$, and ${c_i\in\R}$, ${i=0,1,\ldots,k}$, $\A\in \mathbb{R}^{m\times n}$ and ${\bb\in\R^m}$, $\|\bx\|_0$ is the $\ell_0$-norm of $\bx$, counting the number of its nonzero elements,  $s\ll n$ is an integer, and $\X$ is a simple closed convex set. Throughout the paper, we assume ${\X:=X_1\times X_2\times \cdots \times X_n}$, where each ${X_i\ni0}$ is a closed interval in $\mathbb{R}$.  The primary challenge of the SQCQP problem stems from sparsity constraint $\| \bx \|_0 \le s$ since the $\ell_0$-norm is non-convex and discontinuous, making the problem generally NP-hard. Importantly, this complexity remains even in the convex setting, where each $\Q_i$ is positive semi-definite. Another challenge is the involvement of quadratic constraints, especially in large-scale or high-dimensional settings. Nevertheless, the SQCQP problem finds extensive applications, including portfolio optimization with cardinality constraints on asset selection \cite{bai2016splitting, cui2013convex}, sparse generalized eigenvalue problems in high-dimensional data analysis and machine learning \cite{beck2016sparse, lindenbaum2021l0, sriperumbudur2011majorization}, and sparse array beamforming in signal processing \cite{hamza2019hybrid, huang2023sparse}. 
\subsection{Related work}
 To overcome the computational challenges posed by the $\ell_0$-norm, a variety of research efforts have focused on reformulating \eqref{eq: SCQP} as a mixed-integer programming (MIP) by introducing binary variables. This reformulation enables the application of various solution methods, including branch-and-bound techniques, evolutionary algorithms, and other heuristic-based approaches. A comprehensive overview of existing methods can be found in \cite{bertsimas2009algorithm, chang2000heuristics, di2012concave, streichert2004evolutionary} and the references therein. However, these methods often exhibit limited computational efficiency, particularly when applied to large-scale instances. Beyond the MIP-based reformulation, other existing approaches for solving the SQCQP problem can be broadly classified into two categories.

 {\it a) Relaxation methods.} One strategy for addressing the challenges stemmed from the $\ell_0$-norm is to relax it by a continuous or convex surrogate function \cite{ahn2017difference, chartrand2007exact, dong2019structural, fan2001variable}. An alternative line of research transfers the original problem into nonlinear programming (NLP) with complementarity constraints \cite{vcervinka2016constraint, feng2013complementarity, kanzow2021augmented, kanzow2022sparse, steffensen2024relaxation, steffensen2010new}. While this reformulation formally aligns with the standard NLP framework, it introduces significant theoretical and numerical challenges. Specifically, the resulting problems are inherently non-convex and degenerate, often violating classical constraint qualifications required for NLP \cite{burdakov2016mathematical}. Consequently, specialized techniques must be developed to ensure both theoretical rigor and computational tractability.

 {\it b) Greedy methods.} In contrast, greedy approaches are capable of directly handling the $\ell_0$-norm, with their development rooted in the field of compressed sensing. For problems involving only a single sparsity constraint set, a wide range of algorithms have developed, including first-order methods using gradient information  \cite{IHT, GP, OMP, CoSaMP, CGIHT}, as well as Newton-type algorithms that leverage second-order information \cite{GraSP, GraHTP, NTGP, NHTP,  wang2021extended, GPNP, sun2023gradient}.  However, compared to the case with a single sparsity constraint, research on problems involving additional constraints remains relatively limited. Existing studies mainly focus on scenarios with relatively simple constraint sets, such as the simplex set \cite{kyrillidis2013sparse}, non-negative constraints \cite{pan2017convergent}, box constraints \cite{cesarone2013new}, affine sets \cite{bauschke2014restricted, hesse2014alternating}, and symmetric sets \cite{beck2016minimization, lu2015optimization, li2020matrix}, with algorithms typically based on gradient projection techniques. For problems involving more complex nonlinear equality or inequality constraints, only a few works have been reported. For example, the authors in \cite{lu2013sparse} introduced a block coordinate descent algorithm using a penalty decomposition technique to handle such cases. Moreover, for problems involving semi-continuous variables and additional closed convex set constraints,  \cite{bai2016splitting} proposed an augmented Lagrangian algorithm again based on the penalty decomposition.

It is worth mentioning that most existing algorithms fall under the category of first-order methods. Although these methods are structurally simple and easy to implement, they often fail to achieve high solution accuracy and computational efficiency compared to Newton-type methods. In a recent study \cite{zhao2022},  an efficient Lagrange-Newton algorithm (LNA) was proposed for sparse NLP with additional equality constraints. The method can be viewed as a generalization of NHTP \cite{NHTP} using Newton's method to solve the stationary equations. It has been demonstrated that both algorithms achieve fast locally quadratic convergence rate under certain conditions. Despite the scarcity of Newton-type algorithms in constrained sparse optimization, the impressive computational efficiency of LNA and the well-established semi-smooth Newton methods for solving the Karush-Kuhn-Tucker (KKT) system in traditional NLP \cite{de1996semismooth, facchinei1997new, fischer1992special, qi1997semismooth} suggest the potential to develop an efficient second-order method for \eqref{eq: SCQP}, which motivates the research of this paper.
\subsection{Contribution}
The primary contribution of this paper lies in both the theoretical analysis and the development of a Newton-type algorithm for solving problem \eqref{eq: SCQP}. To the best of our knowledge, this is the first work to propose a direct numerical algorithm for this problem rather than relying on its reformulations. Our contributions can be summarized as follows.

  \textit{$A.$ Optimality analysis via P-stationarity.} We introduce the concepts of KKT points and P-stationary points to establish the first- and second-order optimality conditions for problem \eqref{eq: SCQP}, clarifying their connections to local minimizers. By leveraging the properties of the sparse projection operator and a nonlinear complementarity problem (NCP) function, we equivalently reformulate the P-stationarity condition as a system of stationary equations. This reformulation provides a crucial theoretical foundation for the development of a Newton-type method. However, besides the solution variables, these equations also include an unknown discrete index set that needs to be determined simultaneously. Consequently, solving these equations differs from the  standard nonlinear equation by employing the semismooth Newton-type methods. 

 \textit{$B.$ A semismooth Newton-type method with locally quadratic convergence.} It is important to note that, in addition to the solution variables, the stationary equations also involve an unknown discrete index set that must be determined simultaneously. As a result, solving these equations differs from solving standard nonlinear equations using a Newton-type method. Nevertheless, we successfully address this challenge by developing a semismooth Newton-type method, {SNSQP}.
The novel design ensures that each iteration produces an $s$-sparse solution $\bx^\ell$, namely, $\|\bx^\ell\|_0\leq s$. Consequently, although a system of linear equations must be solved at each step, the algorithm maintains a remarkably low computational complexity, enabling efficient large-scale computation. Furthermore,   we establish the nonsingularity of the generalized Jacobian of the stationary equations under mild assumptions, which guarantees a locally quadratic convergence rate of the proposed algorithm.

 \textit{$C.$ High numerical performance.} We conduct extensive numerical experiments on various applications using both synthetic and real-world datasets. Comparative evaluations against state-of-the-art algorithms and commercial solvers, such as CPLEX and GUROBI, demonstrate that {SNSQP} achieves superior computational efficiency and solution accuracy, highlighting its strong potential for a variety of large-scale applications.

\subsection{Organization}
The paper is organized as follows. The next subsection introduces key notation used in the paper. In Section \ref{Optimality analysis} , we establish the optimality conditions of   \eqref{eq: SCQP} using KKT points and P-stationary points, and then reformulate P-stationarity as a system of stationary equations. In Section \ref{Generalized Jacobian non-singularity},  the nonsingularity of the generalized Jacobian of the stationary equations is analyzed. In Section \ref{A semi-smooth Newton algorithm}, we propose a semi-smooth Newton method  and establish its locally quadratic convergence. Extensive numerical experiments and concluding remarks are given in the last two sections.

\begin{table}[!t]
	{\begin{tabular}{ll}
			\toprule
			\textbf{Notation} & \textbf{Description}\\
			\midrule
			$[n]$   &$:=\{1,2,\ldots,n\}$.\\
			$|t|$   & The absolute value of scalar $t$.\\
			
			$|T|$  & The cardinality of set $T\subseteq[n]$.\\
			$\overline{T}$  & The complementary set of $T$, namely, $\overline{T}=[n]\setminus T$.\\
			$x_{( i )}$  & The $i$th largest (in absolute value) entry of vector $\bx$.\\
			$\| \bx\| $  & The Euclidean norm of vector $\bx$.\\
			$\|\bx\|_\infty$ & The $\ell_\infty$-norm of vector $\bx$.\\
			$\supp(\bx)$  & $:=\{ i \in [n]: x_i \neq 0\}$, the support set of vector $\bx$.\\
			$\langle \bw,\bx  \rangle$ & The inner product of two vectors  $\bw $ and $\bx$, i.e., $\langle \bw,\bx  \rangle=\bw^\top\bx=\sum w_i x_i$.\\
$(\bw;\bx)$ & The vector formed by stacking  $\bw $ and $\bx$, i.e., $(\bw;\bx)= (\bw^\top~\bx^\top)^\top$.\\			
			$\mathbb{J} _s(\bx)$  & $:  = \{ J \subseteq [n]: |J|=s,J\supseteq \supp(\bx) \}.$\\
			$\bx_T$  & The subvector of $\bx$ containing elements indexed by $T$.\\
			$\A_{IJ}$ & The submatrix of $\A\in\R^{m\times n}$ whose rows and columns are indexed by $I$ and $J$.\\
			 & In particular, $\A_{:T}:=\A_{[m]T}$ and $\A_{T:}:=\A_{T[n]}$.\\
			$\X_{J}$ & $: = \times_{i\in J}X_i$, e.g., $\X_{\{1,3,4\}} =X_1\times X_3 \times X_4.$ \\
			
			$\mathbf{e}_i$ & The $i$-th column of identity matrix $\textbf{I}$.\\
			
                ${\rm bd}(\Omega)$ & The boundary of set $\Omega$.\\
                ${\rm int}(\Omega)$ & The interior of set $\Omega$.\\
			\bottomrule
	\end{tabular}}
     \caption{A list of notation.}
    \label{table: notation}
\end{table}
\subsection{Notation}
We end this section by introducing some notation to be used throughout the paper, with most of them summarized in Table \ref{table: notation}. In addition, we denote
\begin{equation}
	\label{eq: def-f-omega}
	\begin{aligned}
f_i( \bx )&:=\frac{1}{2}\bx^{\top}  \mathbf{Q}_{i}\bx+\mathbf{q}_{i}^{\top}\bx+c_i,~i=0,1,2,\ldots,k,\\
\F &:=\{ \bx\in \mathbb{R} ^n:~f_i\left( \bx \right)\leq 0,~i=1,2,\ldots,k,~\A\bx-\bb\leq  0 \},\\
\mathbb{S}&:= \{\bx \in \mathbb{R}^n:~\| \bx \|_0 \le s \},
 \end{aligned}	
\end{equation}
where $\mathbb{S}$ is known as the sparse set. Let ${\Pi _\Omega }(\bx)$ be the projection of $\bx$ onto set $\Omega$, namely,
$${\Pi _\Omega }(\bx) = {\rm argmin}_{\bz\in\Omega}~\|\bz-\bx\|^2.$$
Therefore,  ${\Pi _\mathbb{S}}(\bx)$ keeps the first $s$ largest (in absolute value) entries of $\bx$ and sets the remaining to be zeros.  Let $T_{\Omega}(\bx)$,  $N_{\Omega}(\bx)$, and $\widehat{N}_{\Omega}(\bx)$ be the Clarke tangent cone,  Clarke/limiting normal cone, and Fr{\'e}chet normal cone of $\bx$ at $\Omega$, respectively.  One can refer to \cite{rockafellar2009variational} for their definitions. It is easy to calculate that  
\begin{equation}
\label{eq: normal and tangent cone of X}
T_{[a,b]}(x) =   \begin{cases}
[0,+\infty),  & \text { if }~ x=a, \\
\mathbb{R}, & \text { if } ~ x\in(a,b), \\
(-\infty, 0], & \text { if }~ x =b,\end{cases}\qquad
N_{[a,b]}(x)  \in \begin{cases}(-\infty, 0], & \text { if }~ x=a, \\
\{0\}, & \text { if }~ x\in(a,b), \\
{[0,+\infty),} & \text { if }~ x=b.\end{cases} 
\end{equation}
Moreover, by  \cite[Theorem 2.1]{pan2015solutions}, we have
\begin{equation}
\label{eq: normal cone}
\widehat{N}_{\mathbb{S}}(\bx)= \begin{cases}
\mathbb{R}_{\overline{\Gamma}}^n:=\left\{\bx \in \mathbb{R}^n: \bx_{\Gamma}=0\right\}, & \text { if }~\|\bx\|_0=s, \\ 
\{0\}, & \text { if }~\|\bx\|_0<s.\end{cases}
\end{equation}
Let $f: \mathbb{R}^n \to \mathbb{R}^m$ be a locally Lipschitz continuous function. Then $f$ is differentiable almost everywhere by Rademacher's Theorem. By denoting $D_f$ as the set of points where $f$ is differentiable, the Clarke generalized Jacobian \cite{clarke1990optimization} of $f$ at $\mathbf{x} \in \mathbb{R}^n$ is defined as
$$
\begin{array}{l}
\partial f(\mathbf{x}) = \operatorname{co}\left\{ \lim _{\mathbf{x}^\ell \in D_f, \mathbf{x}^\ell \to \mathbf{x}} \nabla f(\mathbf{x}^\ell) \right\},
\end{array}$$
where $\operatorname{co}(\Omega)$ represents the convex hull of $\Omega$. Let $\mathcal{U} \subseteq \mathbb{R}^n$ be an open set and $f: \mathcal{U} \to \mathbb{R}^m$ be a locally Lipschitz continuous function. We say that $f$ is semismooth at $\mathbf{x} \in \mathcal{U}$ if it is directionally differentiable at $\mathbf{x}$, and for every $\Delta \mathbf{x} \to 0$ and $\mathbf{H} \in \partial f(\mathbf{x} + \Delta \mathbf{x})$ it satisfies
$$
f(\mathbf{x} + \Delta \mathbf{x}) - f(\mathbf{x}) - \mathbf{H} \Delta \mathbf{x} = o(\|\Delta \mathbf{x}\|).
$$
Furthermore, if the above equation is replaced by
$$
f(\mathbf{x}+\Delta \mathbf{x})-f(\mathbf{x})-\mathbf{H} \triangle \mathbf{x}=O\left(\|\Delta \mathbf{x}\|^2\right),
$$
then $f$ is said to be strongly semismooth at $\mathbf{x}$. 
Finally,  Fischer-Burmeister (FB) function \cite{fischer1992special} $\phi:\mathbb{R} \times \mathbb{R} \rightarrow \mathbb{R}$ is defined by
$$
\phi(a, b):=\sqrt{a^2+b^2}-a-b .
$$
It is well-known that the FB function is a  nonlinear complementarity problem (NCP) function that satisfies  
$\phi \left( a,b \right) =0$ if and only if  $a\geq 0,~ b\geq  0,$ and $ab=0.$

\section{Optimality Analysis}

\label{Optimality analysis}
In this section, we analyze the optimality conditions of  (\ref{eq: SCQP}) by introducing P-stationarity, which is then  equivalently formulated as stationary equations using the sparse projection operator and the FB function. This serves as a fundamental theoretical basis for algorithm design.
\subsection{First-order optimality conditions}
The Lagrangian function of  \eqref{eq: SCQP} is,
\begin{equation}
L\left( \bx,\bmu,\blambda \right):=f_0\left( \bx \right)+\sum_{i=1}^k{\mu _if_i\left( \bx \right)}+\langle \blambda, \A\bx-\bb \rangle,
\end{equation}
 where $\bmu \in \mathbb{R}^k$ and $\blambda \in \mathbb{R}^m $ are Lagrange multipliers. Given point $(\bx^*,\bmu^*,\blambda^*)  \in  \mathbb{R}^{n}\times\mathbb{R}^k \times \mathbb{R}^m$,  hereafter, we always denote
 \begin{align*}
\bg^* :=\nabla _x L\left( \bx^*,\bmu ^*,\blambda ^* \right), ~~ 
\H^* :=\nabla_{xx}^2 L\left( \bx^*,\bmu ^*,\blambda ^* \right),~~
\Gamma_*:=\supp(\bx^*),~~\X_*:=\X_{{\Gamma_*}}.
 \end{align*}
\begin{definition}[{\bf KKT Points}]
We call $\bx^*$ a KKT point of \eqref{eq: SCQP} if there is $(\bnu^*,\bmu^*,\blambda^*)  \in  \mathbb{R}^{n}\times\mathbb{R}^k \times \mathbb{R}^m $ such that,
\begin{equation}
	\label{eq: BKKT}
 \begin{cases}
		-\bg^* - \bnu^* \in \widehat{N}_{\mathbb{S}}\left( \bx^* \right),\\
	\bnu^*_{\Gamma_*}\in N_{\X_*}( \bx_{{\Gamma_*}}^{*} ),~ \bnu^*_{\overline{\Gamma}_*}=0, ~\bx_{\Gamma_*}^* \in \X_{\Gamma_*},\\
			f_i\left( \bx^* \right) \leq  0, ~\mu_i ^*\geq 0, ~\mu _{i}^{*}f_i\left( \bx^* \right) =0, ~i\in[k],\\		
		\A\bx^*-\bb\leq  0, ~\blambda ^*\geq 0,~    \langle {\blambda ^*}, \A\bx^*-\bb\rangle =0.			
	\end{cases}  
\end{equation}
\end{definition}
To derive the relationship between a KKT point and a local minimizer of \eqref{eq: SCQP}, we need the following restricted linear independent constraint qualification (LICQ) condition.
\begin{assumption}
	[{\bf Restricted LICQ}]
	\label{assum_RLICQ}
	Let $\bx^*\in \F  \cap \mathbb{S} \cap \X $ and 
	\begin{align}
		\mathcal{A}_1\left(\bx^*\right)&:=\left\{i \in[k]:f_i\left(\bx^*\right)=0\right\},\\
		\mathcal{A}_2\left(\bx^*\right)&:=\left\{i \in[m]: \langle\ba_i, \bx^*\rangle=\bb_i\right\}, \\
		\mathcal{A}_3\left(\bx^*; T\right)&:=\left\{i \in {T} :x^*_i \in b d\left(X_i\right)\right\}.
	\end{align}
Assume the following groups of vectors are linearly independent for any $T \in \mathbb{J}_s(\bx^*)$,
	\begin{align}\label{linearly-independent-Gamma}
	\Big\{(\nabla f_i(\bx^*))_{{\Gamma_*}}:i\in \mathcal{A} _1\left( \bx^* \right)\Big\}  \cup \Big\{ \left( \ba_{i} \right) _{{\Gamma_*}}: i\in \mathcal{A} _2\left( \bx^* \right) \Big\} \cup \Big\{ \left( \mathbf{e}_i  \right) _{{\Gamma_*}}:  i\in \mathcal{A} _3\left( \bx^* ;T\right)\Big\},\end{align} 
	where $ \ba_{i}$ is a column vector formed by the $i$th row of $\A$.
\end{assumption}
\begin{theorem}	\label{thm_first order BKKT}
	Let $\bx^*$ be a local minimizer of (\ref{eq: SCQP}) and Assumption \ref{assum_RLICQ} hold  at $\bx^*$. Then there is a unique $(\bnu^*,\bmu^*,\blambda^*)  \in  \mathbb{R}^{n}\times\mathbb{R}^k \times \mathbb{R}^m $ such that $\bx^*$ is a KKT point of (\ref{eq: SCQP}).
\end{theorem}
\begin{proof} Let ${\bx^*}$ be a local minimizer of (\ref{eq: SCQP}). Then there is a neighbourhood $\mathbb{N}^*$ of $\bx^*$ satisfying
	\begin{equation*}
 f_0(\bx^*)\leq  f_0(\bx),~~ \forall~ \bx\in\F \cap \mathbb{S} \cap \X \cap \mathbb{N}^*.
	\end{equation*}
Let $\mathbb{S}_T:=\{\bx\in\R^n:\bx_{\overline{T}} = 0,~\bx_{T}\in \X_{T}\}$ for any given ${T \in \mathbb{J}_s(\bx^*)}$. Then one can verify that $\mathbb{S}_T\subseteq (\mathbb{S} \cap \X)$, resulting in $f_0(\bx^*)\leq  f_0(\bx)$ for any $ \bx\in  \F \cap \mathbb{S}_T \cap \mathbb{N}^*$. This indicates that ${\bx^*}$  it is also a local minimizer of the problem,
	\begin{equation}
		\label{eq2.1}
		\begin{aligned}
			 \min ~  f_0(\bx), ~~ {\rm s.t.}  ~   \bx\in\F,~ \bx_{\overline{T}} = 0,~\bx_{T}\in \X_{T},
		\end{aligned}
	\end{equation}
for any given ${T \in \mathbb{J}_s(\bx^*)}$.  By $\Gamma_*\subseteq T$,   Assumption \ref{assum_RLICQ} indicates that
	\begin{align*} 
	\Big\{(\nabla f_i(\bx^*))_{T}:i\in \mathcal{A} _1\left( \bx^* \right)\Big\}  \cup \Big\{ \left( \ba_{i} \right) _{T}: i\in \mathcal{A} _2\left( \bx^* \right) \Big\} \cup \Big\{ \left( \mathbf{e}_i  \right) _{T}:  i\in \mathcal{A} _3\left( \bx^* ;T\right)\Big\}\end{align*} 
are linearly independent, namely,  the LICQ holds at $\bx^*$ for (\ref{eq2.1}). Then by \cite[Theorem 1]{wachsmuth2013licq}, there exist a unique $(\bnu,\bmu,\blambda,\bm{\gamma})  \in \mathbb{R}^{n}\times \mathbb{R}^k \times \mathbb{R}^m \times \mathbb{R}^{n-s}$, such that 
    \begin{equation}
    	\label{eq2.2}
     \begin{cases}
    		\nabla _xL\left( \bx^*,\bmu,\blambda \right) + \bnu +\sum_{i\in \overline{T}}{{\gamma}_{i}\mathbf{e}_i}=0,\\
    		\bnu_T\in N_{\X_{T}}\left( \bx_{T}^{*} \right) ,~\bnu_{\overline{T}}=0,~ \bx_T^* \in \X_T,\\
    		f_i\left( \bx^* \right) \leq  0, ~\mu_i\geq 0, ~\mu _{i}f_i\left( \bx^* \right) =0, ~i\in[k],\\ 
    		\A\bx^*-\bb\leq  0, ~\blambda\geq 0,~\langle {\blambda},   \A\bx^*-\bb \rangle=0.
    	\end{cases} 
    \end{equation}
If $\|\bx^*\|_0=s$, then $\mathbb{J}_s(\bx^*) = \{ \Gamma_* \}$ and thus $T=\Gamma_*$. By letting $(\bnu^*,\bmu^*,\blambda^*) = (\bnu,\bmu,\blambda)$, we obtain $-\bg^* -  \bnu^* =\sum_{i\in \overline{\Gamma}_*}{ \gamma_{i}\mathbf{e}_i} \in \mathbb{R}^n_{\overline{\Gamma}_*}=\widehat{N}_{\mathbb{S}}(\bx^*)$ from \eqref{eq: normal cone} and $\bnu^*_{\overline{\Gamma}_*}=0$.

If $\|\bx^*\|_0<s$, then for any  $T \in \mathbb{J}_s(\bx^*)$,  we rewrite the first row in (\ref{eq2.2}) as follows,
    \begin{equation}
    	\label{eq2.3}
    	\begin{cases}
    		\left[ \nabla _xL\left( \bx^*,\bmu,\blambda \right)  + \bnu\right] _{\Gamma_*}=0,\\
    		\left[ \nabla _xL\left( \bx^*,\bmu,\blambda \right) +\bnu \right] _{T\backslash\Gamma_*}=0,\\
    		\left[ \nabla _xL\left( \bx^*,\bmu,\blambda \right) +\bnu+\sum_{i\in \overline{T}} \gamma_{i}\mathbf{e}_i  \right] _{\overline{T}}=0.\\
    	\end{cases}
    \end{equation}
    Assumption \ref{assum_RLICQ} indicates the uniqueness of $(\bmu,\blambda,\bnu)$ from the first equation of (\ref{eq2.2}). Therefore, we let $(\bnu^*,\bmu^*,\blambda^*) = (\bnu,\bmu,\blambda)$. Since $\cup_{T\in \mathbb{J} \left( x^* \right)}(T\backslash\Gamma_*)=\overline{\Gamma}_*$, it follows from the second row of (\ref{eq2.3}) that $\left( \bg^* +  \bnu \right) _{\overline{\Gamma}_*}=0$.  This together with  the first row of (\ref{eq2.3}) shows that $\bg^*+  \bnu^*  = 0\in\widehat{N}_{\mathbb{S}}(\bx^*)$ from \eqref{eq: normal cone}.  Furthermore,    $\bnu^*_{\overline{T}}=0$  for any $T \in \mathbb{J}_s(\bx^*)$ suffices to $\bnu^*_{\overline{\Gamma}_*}=0$.
    
Therefore,  both cases lead to the first condition in \eqref{eq: BKKT} and $\bnu^*_{\overline{\Gamma}_*}=0$. Then the other conditions in \eqref{eq: BKKT} can be ensured by \eqref{eq2.2} with  $ (\bnu,\bmu,\blambda)=(\bnu^*,\bmu^*,\blambda^*)$.  
\end{proof}
\begin{example} Given $c\geq0$, consider the following problem,
	 \begin{equation}
	 	\label{eq: counterexample}
	 	\begin{array}{lllll}	
	 		\underset{\bx\in \mathbb{R} ^3}{\min}  & (x_1-1)^2+(x_2+1)^2+(x_3-1)^2\\
	 		\mathrm{s.t.}  & x^2_1+(x_2-1)^2-3\leq0,\\
	 		& x_2^2+(x_3-c)^2-(1-c)^2\leq0,\\
	 		& x_1+ x_2 + x_3-2\leq0,\\
	 		&\| \bx \|_0 \le 2, ~-2\leq x_i\leq2,~ i=1,2,3.
	 	\end{array}
	 \end{equation}
	 The first inequality constraint indicates $x_2\geq 1-\sqrt{3}>-1$, which together with $\| \bx \|_0 \le 2$ indicates that $\bx^*=(1,0,1)^{\top}$ is a global minimizer of problem (\ref{eq: counterexample}) and hence $\Gamma_*=\{1,3\}$, $\mathcal{A}_1(\bx^*)=\{2\}$, and $\mathcal{A}_2(\bx^*)=\{1\}$. Then to ensure the restricted LICQ, we need to ensure the linear independence of vectors $(0~2(1-c))^\top$ and $(1~1)^\top$. Clearly, they are linearly independent if $c\neq1$ and dependent if $c=1$. However, irrespective of whether the restricted LICQ holds, one can always find that ${\mu^*=0}$, ${\lambda^*=0}$ and ${\nu^*=0}$ satisfying  (\ref{eq: BKKT}). Therefore, $\bx^*$ is KKT point. This example demonstrates that the restricted LICQ is a sufficient, but not necessary, condition for Theorem \ref{thm_first order BKKT}.
\end{example}
\begin{remark}
It is worth mentioning that the restricted LICQ plays a central role in the subsequent theoretical analysis; therefore, we illustrate the feasibility of this assumption in practice through two special cases of (\ref{eq: SCQP}). When $\X=\R^n$ and $k=0$ (i.e., no quadratic constraints), the restricted LICQ reduces to $\A_{\mathcal{A}_2\left(\bx^*\right)\Gamma_*}$ having a full row rank, requiring $|\mathcal{A}_2\left(\bx^*\right)|\leq |\Gamma_*|\leq s$. When only the box constraint $\bx \in \X$ is present, Assumption \ref{assum_RLICQ} holds at any feasible point if $0 \in \operatorname{int}(\X)$.
\end{remark}

According to Theorem \ref{thm_first order BKKT}, a KKT point is closely related to a local minimizer of \eqref{eq: SCQP}. However, computing it directly is intractable. To address this, we introduce an alternative point that can be efficiently obtained through a well-designed numerical algorithm. This point is defined based on the projection onto sparse set $\mathbb{S}$, which is why we refer to it as a  $P$-stationary point, where  $P$ represents the projection.
\begin{definition}
	\label{df1}
	We call $\bx^* \in \mathbb{R}^n$  a $P$-stationary point of   (\ref{eq: SCQP}) associated with a constant $\tau>0$ if there is  $(\bnu^*,\bmu^*,\blambda^*)  \in  \mathbb{R}^{n}\times\mathbb{R}^k \times \mathbb{R}^m $ such that,
	\begin{equation}
		\label{eq2.5}
		  \begin{cases}
			\bx^*~=\Pi _{\mathbb{S}}\left(\bx^*-\tau (\bg^* +   \bnu^* ) \right),\\
\bx_{\Gamma_*}^*=\Pi_{\X_*}( \bx^*_{\Gamma_*}+\bnu^*_{\Gamma_*}),~\bnu^*_{\overline{\Gamma}_*}=0,\\					f_i\left( \bx^* \right) \leq  0, ~\mu_i ^*\geq 0, ~\mu _{i}^{*}f_i\left( \bx^* \right) =0, ~i\in[k],\\		
		\A\bx^*-\bb\leq  0, ~\blambda ^*\geq 0,~\langle {\blambda ^*}, \A\bx^*-\bb\rangle=0.
		\end{cases} 
	\end{equation}
\end{definition}
 Recalling the definition of the projection operator, the conditions in the first two rows of \eqref{eq2.5} are equivalent to
\begin{equation} \label{eq-x*-x-gamma*}
	\begin{cases}
		\bx^*~=\underset{\mathbf{u}\in \mathbb{S}}{\mathrm{arg}\min}~\frac{1}{2}\left\| \mathbf{u}-\left( \bx^*-\tau (\bg^* +   \bnu^* ) \right) \right\| ^2,\\
		\bx_{\Gamma_*}^*=\underset{\mathbf{v}\in \X_*}{\mathrm{arg}\min}~\frac{1}{2}\left\| \mathbf{v}-\bx^*_{\Gamma_*}-\bnu^*_{\Gamma_*}\right\| ^2, ~\bnu^*_{\overline{\Gamma}_*}=0.
	\end{cases} 
\end{equation}
According to the generalized Fermat’s rule \cite[Theorem 10.1] {rockafellar2009variational}, the optimality conditions of the above two problems are exactly the ones presented in the first two rows of \eqref{eq: BKKT}. Therefore, a P-stationary point must be a KKT point. Now using the FB function, the two complementary conditions in (\ref{eq2.5}) can be written as
\begin{equation} \label{def-varphi-phi}
{\varphi}\left( \bx,\bmu \right) :=\left[ \begin{array}{c}
	\phi \left( -f_1\left( \bx \right),\mu _1 \right)\\
	\vdots\\
		\phi \left( -f_k\left( \bx \right),\mu _k \right)\\
\end{array} \right]=0 ,\qquad {\psi}\left( \bx,\blambda \right) :=\left[ \begin{array}{c}
		\phi \left( b_1 -\langle \ba_{1}, \bx \rangle ,\lambda _1 \right)\\
	\vdots\\
		\phi \left( b_m- \langle \ba_{m}, \bx \rangle  ,\lambda _m \right)\\
\end{array} \right]=0.
\end{equation}
According to \cite[Lemma 2.2] {beck2013sparsity} of ${\Pi _\mathbb{S}}(\cdot)$, the first condition in \eqref{eq2.5} is equivalent to $\bx^* \in \mathbb{S}$ and 
\begin{equation}
	\label{eq: proj-lemma}
	|g^*_i+v^*_i|\begin{cases}
	=0,& \text{if}~i\in\Gamma_*,\\
		\leq  x^*_{(s)}/\tau, & \text{if}~i\notin\Gamma_*.
	\end{cases}
\end{equation}
The above condition immediately results in $\bg_{\Gamma_*}^*+\bnu^*_{\Gamma_*}=0$. If  $||\bx^*||_0=s$, then $ \tau\|\bg^*_{\overline{{\Gamma}}_*}\|_{\infty}<x^*_{(s)}$. If $||\bx^*||_0<s$, then $x^*_{(s)} = 0$, which together with $\bnu^*_{\overline{\Gamma}_*}=0$ from  \eqref{eq2.5} leads to $ \bg^*_{\overline{\Gamma}_*}=0$.  Based on these facts, combining \eqref{eq2.5}, \eqref{eq-x*-x-gamma*}, and \eqref{def-varphi-phi}, one can observe that $\bx^*$ is $P$-stationary point if and only if $\bx^* \in \mathbb{S}$ and  
 \begin{equation}
		\label{eq2.5-1}
		  \begin{cases}
		  \begin{cases}
			\bg_{\Gamma_*}^*+\bnu^*_{\Gamma_*}=0, ~~\tau\|\bg^*_{\overline{{\Gamma}}_*}\|_{\infty}<x^*_{(s)}, & \text{if}~\|\bx^*\|_0=s,\\
			\bg^*_{\Gamma_*}+\bnu^*_{\Gamma_*}=0,~~\bg^*_{\overline{\Gamma}_*}=0,& \text{if}~\|\bx^*\|_0<s,
			\end{cases}\\ 
			 ~~~\bx^*_{\Gamma_*}=\Pi _{{\X_*}}( \bx^*_{\Gamma_*}+\bnu^*_{\Gamma_*}),~\bnu^*_{\overline{\Gamma}_*}=0,\\
			 ~~~{\varphi}\left( \bx^*,\bmu^* \right)=0, ~~{\psi}\left( \bx^*,\blambda^* \right)=0,
			 \end{cases} 
	\end{equation}
 which further results in the following optimality conditions for (\ref{eq: SCQP}).
\begin{theorem}[{\bf First-order necessary condition}]	\label{thm_first order necessary stationary}
	Let $\bx^*$ be a local minimizer of (\ref{eq: SCQP}) and  Assumption \ref{assum_RLICQ} hold at $\bx^*$. Then there is a unique  $(\bnu^*,\bmu^*,\blambda^*)  \in  \mathbb{R}^{n}\times\mathbb{R}^k \times \mathbb{R}^m $ such that $\bx^*$ is a $P$-stationary point of (\ref{eq: SCQP}) for any $\tau \in \left( 0,\overline{\tau} \right)$, where
	\begin{equation}
		\label{eq2.7}
		\bar{\tau}:= \begin{cases}\frac{x^*_{(s)}}{ \|\bg^*_{\overline{\Gamma}_*} \|_{\infty}}, & \text { if }\left\|\bx^*\right\|_0=s, \\ 
		+\infty, & \text { if }\left\|\bx^*\right\|_0<s.
		\end{cases}
	\end{equation}
\end{theorem}
\begin{proof}
	 Theorem \ref{thm_first order BKKT} states that there is a unique $(\bnu^*,\bmu^*,\blambda^*)$ such that $\bx^*$ is a KKT point of (\ref{eq: SCQP}), which by the first condition and $\bnu^*_{\overline{\Gamma}_*}=0$ in \eqref{eq: BKKT} indicates
	\begin{equation}
		\label{eq2.8}
		\begin{cases}
		\bg_{\Gamma_*}^*+\bnu^*_{\Gamma_*}=0, & \text{if}~\|\bx^*\|_0=s,\\
			\bg^*_{\Gamma_*}+\bnu^*_{\Gamma_*}=0,~~\bg^*_{\overline{\Gamma}_*}=0,& \text{if}~\|\bx^*\|_0<s,
		\end{cases}
	\end{equation}	
	The value of $\bar{\tau}$ means that for any $\tau \in \left( 0,\bar{\tau} \right)$, we have $\tau\left|g_i^*\right|< x^*_{(s)}$ for all $i \in\overline{\Gamma}_*$ when $\left\|\bx^*\right\|_0=s$. By \eqref{eq2.5-1},   we can conclude that $x^*$ is a $P$-stationary point of (\ref{eq: SCQP}) for any $\tau \in \left( 0,\bar{\tau} \right)$.
\end{proof}
The above theorem states that a local minimizer $\bx^*$ is a P-stationary point for any $\tau \in \left( 0,\overline{\tau} \right)$, where $\overline{\tau}$ can be arbitrary positive scalar when $\|\bx^*\|_0<s$ while relies on $\bx^*$ when $\|\bx^*\|_0=s$.  Therefore, it is unknown and difficult to estimate under the case of $\|\bx^*\|_0=s$ in general. However, in some simple cases, $\overline{\tau}$ could be replaced by its positive lower bound. For example, consider the case of $k=0$ and $m=0$, let $\ell$ satisfy $0<\ell \leq x^*_{(s)}$ and $u=\max_i\{|\ell_i|, |u_i|\}$ with $\ell_i$ and $u_i$ being the end points of $X_i$ (namely, $X_i=[\ell_i, u_i]$). Then 
\begin{equation*}
\begin{aligned} 
 \|\bg^*_{\overline{\Gamma}_*} \|_{\infty} & = \left\|(\Q_{0})_{\overline{\Gamma}_*\Gamma_*}\bx^*_{\Gamma_*}+(\bq_{0})_{\overline{\Gamma}_*}  \right\|_{\infty} \leq u\|\Q_{0}\|_1 +\|\bq_{0}\|_{\infty},
\end{aligned}\end{equation*}
where $\|\Q\|_1 $  is the 1-norm of matrix $\Q$, which results in the following lower bound of $\overline{\tau}$,
\begin{equation}\label{lower-bound-tau}
 \underline{\tau}:= \frac{\ell}{u\|\Q_{0}\|_1 +\|\bq_{0}\|_{\infty}}.
 \end{equation}
This $\underline{\tau}$ can be used in Theorem \ref{thm_first order necessary stationary} to replace $\overline{\tau}$.  In other words,  in the case of $k=0$ and $m=0$, if  lower positive bound $\ell$ and upper positive bound $u$ of the non-zero entries of $\bx^*$ can be estimated in advance, then  $\underline{\tau}$ in \eqref{lower-bound-tau} can be used as an explicit substitute for $\overline{\tau}$. It is worth noting that the authors in \cite{cui2013convex} have studied sparse portfolio selection problems with constraints:  $0\leq x_i\leq b$ and $x_i \geq a_i$ if $x_i\neq0$, where $a_i\in(0, 1)$. Hence,  $\ell$ and $u$ are available as $\ell=\min_i a_i$ and $u=\max\{b,1\}$. 

It is worth mentioning that unknown constant $\overline{\tau}$ serves only to establish the existence of an interval $\tau \in \left( 0,\overline{\tau} \right)$ for which the theoretical results hold. However, it is unnecessary to estimate $\overline{\tau}$ in practice, as it does not provide any guidance for choosing $\tau$ in numerical experiments.

\begin{theorem}[{\bf First-order sufficient condition}]\label{thm_first order sufficient}
	Suppose that $\Q_0,\Q_1\ldots,\Q_k$ are positive semi-definite.  Then a $P$-stationary point of (\ref{eq: SCQP}) is a local minimizer if $\left\| \bx^* \right\| _0=s$ and a global minimizer if $\left\| \bx^* \right\| _0<s$.
\end{theorem}

\begin{proof} Let $\bx^*$ is a $P$-stationary point of (\ref{eq: SCQP}). Then $\bx^*$ is a KKT point and hence there is $(\bnu^*,\bmu^*,\blambda^*)$ such that \eqref{eq2.5-1} and \eqref{eq: BKKT} hold. Note that ${\bnu^*_{\Gamma_*}\in N_{{\X_*}}( \bx^{*}_{\Gamma_*} )}$, it follows $\langle \bnu^*_{\Gamma_*}, \bx_{\Gamma_*} -\bx^*_{\Gamma_*} \rangle\leq  0$ for any $ \bx\in \F  \cap \mathbb{S} \cap \X$. If  $||\bx^*||_0<s$, then $\bg^*_{\Gamma_*} = - \bnu^*_{\Gamma_*} $ and $\bg^*_{\overline{\Gamma}_*}=0$  by \eqref{eq2.5-1}, yielding that 
\begin{align}\label{g-x-x*}
\langle \bg^*, \bx  -\bx^* \rangle =  \langle \bg^*_{\Gamma_*}, \bx_{\Gamma_*} -\bx^*_{\Gamma_*} \rangle  =\langle -\bnu^*_{\Gamma_*}, \bx_{\Gamma_*} -\bx^*_{\Gamma_*} \rangle\geq  0,
\end{align}
Additionally, since $\left\{\Q_0,\Q_1\ldots,\Q_k\right\}$ are positive semi-definite, $L\left( \cdot,\bmu^*,\blambda^* \right)$ is convex, thereby 
		\begin{align}\label{f0x-f0x*}
		f_0(\bx) \geq L\left( \bx,\bmu ^*,\blambda ^* \right) &\geq  L\left( \bx^*,\bmu ^*,\blambda ^* \right) +\langle\bg^*, \bx-\bx^*\rangle \geq f_0(\bx^*),
		\end{align}
for any $\bx\in \F  \cap \mathbb{S} \cap \X$, which means that $\bx^*$ is a global minimizer of (\ref{eq: SCQP}).

If  $||\bx^*||_0=s$, then there is a sufficiently small neighborhood $\mathbb{N}^*$ of $\bx^*$  such that $\supp(\bx)={\Gamma_*}$ for any $\bx \in \mathbb{N}^* \cap (\F \cap \mathbb{S} \cap \X)$.  In fact, one can set the radius of the neighborhood $\mathbb{N}^*$ smaller than $r:={\min}_{i\in {\Gamma}_*}|x^*_i|$. For any $\bx \in \mathbb{N}^*\cap (\F \cap \mathbb{S} \cap \X)$, if there is $t\in\Gamma_*$ such that $t\notin\supp(\bx)$, then $r > \|\bx-\bx^*\|\geq |x^*_t|\geq r$, a contradiction, which implies ${\Gamma}_*\subseteq \supp(\bx)$. This together with $\bx \in \mathbb{S} $ and $|{\Gamma_*}|=s$ yields $\supp(\bx)={\Gamma_*}$.  This indicates condition \eqref{g-x-x*} is also true,   which by the convexity of $L\left( \cdot,\bmu^*,\blambda^* \right)$ implies that condition \eqref{f0x-f0x*} holds for any $ \bx \in \mathbb{N}^* \cap (\F \cap \mathbb{S} \cap \X)$. Hence,  $\bx^*$ is a local minimizer of (\ref{eq: SCQP}).
\end{proof} 
 \begin{figure}[!t]
	\centering
	\begin{tikzpicture}[
		node distance = 1.2cm and 2cm,
		arrow/.style = {-Stealth, thick},
		doublearrow/.style = {Stealth-Stealth, thick}
		]
		\node (pstat) {P-stationarity};
		\node[right=of pstat] (kkt) {KKT};
		\node[right=of kkt] (local) {Local minimizers};
		
		\node[below=2cm of kkt] (bkkt) {B-KKT};
		\node[right=1cm of bkkt] (ckkt) {C-KKT}; 
		\node[right=1cm of ckkt] (sstat) {S-stationarity};  
		\node[right=1cm of sstat] (mstat) {M-stationarity};  
		
		\draw[-Stealth, thick] ([yshift=-1.5pt]pstat.east) -- ([yshift=-1.5pt]kkt.west);
		\draw[-Stealth, thick] ([yshift=1.5pt]kkt.west) -- 
		node[above=0.5mm] {\scriptsize $\tau \in (0,\overline{\tau})$} ([yshift=1.5pt]pstat.east);
		
		\draw[-Stealth, thick] ([yshift=-1.5pt]kkt.east) -- 
		node[below=0.5mm] {\scriptsize $\Q_i \succeq 0$} ([yshift=-1.5pt]local.west);
		\draw[-Stealth, thick] ([yshift=1.5pt]local.west) -- 
		node[above=0.5mm] {\scriptsize Assumption~\ref{assum_RLICQ}} ([yshift=1.5pt]kkt.east);
		
		\draw[-Stealth, thick] ([xshift=-2.5pt]kkt.south) -- ([xshift=-2.5pt]bkkt.north);
		\draw[-Stealth, thick] ([xshift=2.5pt]bkkt.north) -- 
		node[right=0.5mm] {\scriptsize $\X=\mathbb{R}^n$} ([xshift=2.5pt]kkt.south);
		
		\draw[arrow] (bkkt) -- (ckkt);
		
		\draw[-Stealth, thick] ([yshift=-1.5pt]ckkt.east) -- ([yshift=-1.5pt]sstat.west);
		\draw[-Stealth, thick] ([yshift=1.5pt]sstat.west) -- ([yshift=1.5pt]ckkt.east);
		
		\draw[arrow] (sstat) -- (mstat);
		
	\end{tikzpicture}
	\caption{Relationships among different points for problem \eqref{eq: SCQP}.}
	\label{fig: relationship}
\end{figure}
Based on Theorems \ref{thm_first order BKKT}, \ref{thm_first order necessary stationary}, and  \ref{thm_first order sufficient}, we summarize the relationships among P-stationarity, KKT point, local minimizers of problem \eqref{eq: SCQP}, as shown in the first line in Figure \ref{fig: relationship}, where $\Q_i\succeq 0$ means $\Q_i$ is a symmetric positive semidefinite matrix. In the figure, we also present the relationships with four other type of optimality points. They are B-KKT and C-KKT points in \cite[Definition 3.1]{pan2017optimality}, and S-stationary and M-stationary points in \cite[Definition 4.1]{vcervinka2016constraint}. It should be noted that the distinction between B-KKT points and KKT points in this paper lies primarily in the range of the dual multiplier $\bnu^*$. A B-KKT point requires ${\bnu^*\in N_{\X}( \bx^{*} )}$, while a KKT point requires ${\bnu^*_{\Gamma_*}\in N_{\X_*}(\bx_{{\Gamma_*}}^{*} )}$ and ${\bnu^*_{\overline{\Gamma}_*}=0}$. Therefore, every KKT point is necessarily a B-KKT point. When ${\X = \mathbb{R}^n}$, the two concepts are equivalent for \eqref{eq: SCQP}. The remaining relationships shown in the second row of the figure are derived in \cite{zhao2022}. Overall, one can observe that a point being a P-stationary point represents the better condition than the others.

\subsection{Stationary equations}
Hereafter, we denote some notation for simplicity. Given $\bx \in \mathbb{S}$ and a constant $\tau>0$, let
\begin{align*}
\bu  &:= \bx - \tau \left({\nabla _x}L(\bx,\bmu,\blambda)+ \bnu\right),\\
\Y&:=(\bx,\bnu,\bmu,\blambda)\in    \mathbb{R}^n \times \mathbb{R}^{n}\times\mathbb{R}^k \times \mathbb{R}^m ,
\end{align*}
which allows us to define a useful set  by
	\begin{equation}
		\label{eq2.10}
		\mathbb{T}_{\tau}(\Y) := \Big\{  T \subseteq [n]:~ |T|=s,~|  u_i  | \geq | u_j |,\ \forall i \in T,\ \forall j \in \overline{T} \Big\},
	\end{equation}
where $u_i$ is the $i$th entry of $\bu$. One can observe that any $T\in \mathbb{T}_{\tau}(\Y)$ consists of $s$ indices of the first $s$ largest (in absolute value) entries of $\bu$. 
	 Finally, given $\Y$ and an index set $T$,  we define a system of equations as follows,
\begin{equation}
		\label{eq2.011}
		F\left( \Y;T \right) :=\left[ \begin{array}{c}
			\left( \nabla _x L\left( \bx,\bmu,\blambda \right) + \bnu\right) _T  \\
			\bx_{\overline{T}}\\
			\bx_{T}-\Pi _{{\X_{T}}}\left( \bx_{T}+\bnu_{T} \right)\\
			 \bnu_{\overline{T}}\\
			 {\varphi}\left( \bx,\bmu \right)\\
			{\psi}\left( \bx,\blambda \right)\\						
		\end{array} \right].
	\end{equation}
This equation enables the further exploration of condition \eqref{eq2.5}, as outlined below.
\begin{theorem}[{\bf Stationary Equations}]
	\label{thm3} Point  $\bx^*$ is a $P$-stationary point of (\ref{eq: SCQP}) with $\tau>0$ if and only if there is $(\bnu^*,\bmu^*,\blambda^*)  \in  \mathbb{R}^{n}\times\mathbb{R}^k \times \mathbb{R}^m $ such that 
	\begin{equation} \label{eq2.11}
		F\left( \Y^*;T\right) =0, ~~\forall~T \in \mathbb{T}_{\tau}(\Y^*).
	\end{equation}
	Furthermore, 
\begin{align}\label{T=J}
\mathbb{T}_{\tau}(\Y^*) \equiv \mathbb{J}_s({\bx^*}).
\end{align}
\end{theorem}
\begin{proof}   By comparing \eqref{eq2.5-1} and \eqref{eq2.11}, we need to show that 
\begin{align}\label{equ-F-F'-1}
\bx^*=\Pi_{\mathbb{S}}\left(\bx^* - \tau (\bg^*+ \bnu^*)\right)~~\Longleftrightarrow~~ \left(\bg^*+ \bnu^*\right)_T=0,~ \bx^*_{\overline{T}}=0,~ \forall~T \in \mathbb{T}_{\tau}(\Y^*), 
\end{align}
and  the following relation,
\begin{align} 
	&\left[ \begin{array}{c}
			\bx^*_{\Gamma_*}-\Pi _{{\X_*}}\left( \bx_{\Gamma_*}^*+\bnu_{\Gamma_*}^* \right) \nonumber \\
				 \bnu^*_{\overline{\Gamma}_*}
		\end{array} \right]=0,~~\forall~T \in \mathbb{T}_{\tau}(\Y^*) \\
\label{equ-F-F'-2}	\Longleftrightarrow&
			\left[ ~\begin{array}{c}
			\bx^*_{T}-\Pi _{{\X_{T}}}\left( \bx_{T}^*+\bnu_{T}^* \right)\\
			 \bnu^*_{\overline{T}}
		\end{array}~~ \right]=0,~~\forall~T \in \mathbb{T}_{\tau}(\Y^*).
\end{align}
Equivalence \eqref{equ-F-F'-1} follows from \cite[Lemma 4]{NHTP} immediately. Thus, condition \eqref{T=J} follows from \cite[Theorem 3]{zhao2022} and \eqref{equ-F-F'-1}.  Based on this relation and $\bx^*_{\overline{T}}=0$, we can conclude that 
\begin{align} \label{Gamma-in-T}
\Gamma_*\subseteq T,\qquad~\forall~T \in \mathbb{T}_{\tau}(\Y^*).  \end{align}
We prove \eqref{equ-F-F'-2} by considering two cases. When $\|\bx^*\|_0=s$, it has $T=\Gamma_*$ due to \eqref{Gamma-in-T} and $|T|=s$, which immediately shows \eqref{equ-F-F'-2}.  When $\|\bx^*\|_0<s$, it follows from \eqref{Gamma-in-T}  that  $\overline{\Gamma}_*\supseteq  \overline{T}$  for any $T\in \mathbb{T}_{\tau}(\Y^*)$, which means that $ \bnu^*_{\overline{\Gamma}_*} =0 $ suffices to  $\bnu^*_{\overline{T}}=0$ and $$\bx^*_{T\setminus\Gamma_*}-\Pi _{{\X_{T\setminus\Gamma_*}}}\left( \bx_{T\setminus\Gamma_*}^*+\bnu_{T\setminus\Gamma_*}^* \right)=0-\Pi _{{\X_{T\setminus\Gamma_*}}}\left(0+0 \right)=0,$$
showing inclusion $`\Longrightarrow'$. From \eqref{equ-F-F'-2} and \eqref{T=J}, ${\bnu^*_{\overline{T}} =0}$ for any $T \in \mathbb{T}_{\tau}(\Y^*)=\mathbb{J}_s(\bx^*),$ which contributes to  
$0=\cup_{T\in \mathbb{J}_s(\bx^*) }\bnu^*_{\overline{T}} =\bnu^*_{\overline{\Gamma}_*}$. Consequently, $`\Longleftarrow'$ follows due to $\Gamma_*\subseteq T$.
\end{proof}

The motivation for deriving the above theorem stems from numerical algorithm design. Although an unknown index set   $T$ is involved, the core of condition \eqref{eq2.11} is a system of equations, which makes a Newton-type method feasible.
\subsection{Second-order optimality conditions}
To end this section, we establish second-order optimality conditions for (\ref{eq: SCQP}). For convenience, given $\Y^*:=(\bx^*,\bnu^*,\bmu^*,\blambda^*),$ we define the following index sets,
\begin{equation}
	\label{index set 3}
	\begin{aligned}
		 &\eta^1:=\left\{i\in[k] : {\mu}^*_i=0, f_i({\bx}^* )<0\right\}, \qquad\eta^2:=\left\{i\in[m]: {\lambda}^*_i=0, \langle\ba_{i},{\bx}^*\rangle<b_i \right\},\\
		 &\theta^1:=\left\{i\in[k] : {\mu}^*_i=0, f_i({\bx}^* )=0\right\}, \qquad\theta^2:=\left\{i\in[m]: {\lambda}^*_i=0, \langle\ba_{i},{\bx}^*\rangle=b_i\right\},\\
		 &\beta^1:=\left\{i\in[k] : {\mu}^*_i>0, f_i({\bx}^* )=0\right\},\qquad \beta^2:=\left\{i\in[m] : {\lambda}^*_i>0, \langle\ba_{i},{\bx}^*\rangle=b_i\right\}. 
	\end{aligned}
\end{equation}
In addition, for a given index set $T\in \mathbb{T}_\tau(\Y^*)$, let
\begin{equation}
	\label{index set 3-1}
	\begin{aligned}
		 &\eta ^3(T):=\left\{ i\in T: {\nu}^*_i=0,{x}^*_{i} \in {\rm int} (X_{i}) \right\},\qquad&&\eta ^3:= \eta ^3(\Gamma_*),\\
		 &\theta ^3(T):=\left\{ i\in T: {\nu}^*_i=0, {x}^*_{i}\in {\rm bd} (X_{i}) \right\},\qquad&&\theta ^3:= \theta ^3(\Gamma_*),\\
		 &\beta^3(T):=\left\{ i\in T : {\nu}^*_i\ne 0,{x}^*_{i} \in {\rm bd} (X_{i}) \right\}, \qquad&&\beta^3:= \beta ^3(\Gamma_*).\\
	\end{aligned}
\end{equation} 
For a general constrained optimization problem,
	\begin{equation*}
		\label{eq: GCOP}
			\underset{\bx\in \mathbb{R} ^n}{\min} ~f(\bx), ~~\mathrm{s.t.} ~F(\bx)\in D,~\bx\in X,
	\end{equation*}
	where $f: \mathbb{R}^n\rightarrow\mathbb{R}$ and $F: \mathbb{R}^n\rightarrow\mathbb{R}^m$ are twice continuously differentiable, $X \subset \mathbb{R}^n$ and $D \subset \mathbb{R}^m$ are polyhedral, the critical cone \cite[Example 13.25]{rockafellar2009variational} at a feasible point $\bx^*$ is defined by
	\begin{equation*}
	\label{eq: critical cone general}	
	K(\bx^*) = \Big\{ \bw \in T_X(\bx^*) ~\Big|~\nabla F(\bx^*)\bw \in T_D(F(\bx^*)), ~ w_i \nabla_i f(\bx^*) =0,~ i\in[n] \Big\}.
	\end{equation*}
	Based on the intermediate problem \eqref{eq2.1} with $T=\Gamma_*$, namely,
	\begin{equation} \label{eq: sub SQCP}
		\begin{aligned}
			\min ~  f_0(\bx), ~~ {\rm s.t.}  ~   \bx\in\F,~ \bx_{\overline{\Gamma}_*} = 0,~\bx_{\Gamma_*}\in \X_*,
		\end{aligned}
	\end{equation} 
one can derive  the critical cone  for (\ref{eq: SCQP}) as follows,
\begin{equation}
	\label{eq: critical cone}
\mathbb{C}\left( \Y^* \right) :=	\left\{
\bd\in \mathbb{R} ^n:\begin{array}{rr}	
	\langle \nabla f_i({\bx}^* ), \bd\rangle\leq  0,~i\in \theta^1,& ~~\langle  \nabla f_i({\bx}^* ), \bd\rangle=0,~i\in \beta^1
        \\
        \langle \ba_i, \bd \rangle \leq 0,~i\in \theta^2,& ~~\langle \ba_i, \bd \rangle=0,~i\in \beta^2
	\\
	d_{i} \in T_{X_{i}}( {x}^*_{i}),~ i\in \theta^3,& ~~ d_{i} =0,~ i\in \beta^3
	\end{array} \right\}.
\end{equation}

\begin{theorem}	
	[{\bf Second-order necessary condition}]
	\label{thm_second order necessary}
	Let $\bx^*$ be a local minimizer of (\ref{eq: SCQP}) and Assumption \ref{assum_RLICQ} hold  at $\bx^*$. Then there is unique $(\bnu^*,\bmu^*,\blambda^*)  \in  \mathbb{R}^{n}\times\mathbb{R}^k \times \mathbb{R}^m $ such that 
	$$
	\langle \bd, \H^*\bd \rangle \geq  0,~~ \forall~\bd\in \mathbb{C}\left( \Y^* \right) \cap \mathbb{R}_{\Gamma_*}^n.
	$$
\end{theorem}
\begin{proof} As ${\bx^*}$ is a local minimizer of (\ref{eq: SCQP}), it is also a local minimizer of \eqref{eq: sub SQCP}.
	 Then the conclusion follows from Assumption \ref{assum_RLICQ}  and  \cite[Example 13.25]{rockafellar2009variational} immediately. \end{proof}

\begin{theorem}
	[{\bf Second-order sufficient condition}]
	\label{thm_second order sufficient}
	Let $\bx^*$ be a P-stationary point of (\ref{eq: SCQP}), namely, there is $(\bnu^*,\bmu^*,\blambda^*)  \in  \mathbb{R}^{n}\times\mathbb{R}^k \times \mathbb{R}^m $ satisfying (\ref{eq2.5}). Suppose that
	\begin{equation}
		\label{eq2.14}
		 \langle \bd, \H^* \bd \rangle > 0,~~\forall~\bd(\neq 0)\in 
		 \begin{cases}
		 \mathbb{C}\left( \Y^*\right)\cap \mathbb{R}_{\Gamma_*}^n,& \text{if}~
		 \|\bx^*\|_0=s,\\ 
		 \mathbb{C}\left( \Y^*\right),& \text{if}~
		 \|\bx^*\|_0<s.\\
		 \end{cases} 
	\end{equation}
	Then $\bx^*$ is a strictly local minimizer of (\ref{eq: SCQP}).
\end{theorem}
\begin{proof}
    We argue by contradiction. Suppose that $\bx^*$ is not a strictly local minimizer of \eqref{eq: SCQP}. Then there is a sequence $\left\{\bx^{\ell} \right\} \subset(\F  \cap \mathbb{S} \cap \X) \backslash\left\{\bx^*\right\}$ satisfying $\bx^{\ell}  \rightarrow \bx^*$ and
\begin{equation}\label{fxl-smaller-fx*}
		f_0( \bx^{\ell} ) \leq  f_0\left( \bx^* \right),~~ \ell = 1,2,3,\ldots
\end{equation}
Let $\bd^{\ell}:= (\bx^{\ell}  - \bx^*)/{\|\bx^{\ell}  - \bx^*\|}$.  Then $\|\bd^{\ell} \| = 1$ and thus there is a convergent subsequence of $\{\bd^{\ell} \}$ whose limit point $\bd$ satisfies $\|\bd\|=1$.  Without loss of generality, we assume that sequence $\{\bd^{\ell} \}$ itself converges to $\bd$. It is obvious that $\bd_{\overline{\Gamma}_*}=0$ when $ \|\bx^*\|_0=s$ since $\supp(\bx^\ell)=\Gamma_*$ when $\bx^\ell$ is sufficiently close to $\bx^*$. By \eqref{fxl-smaller-fx*}, 
\begin{equation*} 
0 \geq  f_0(\bx^{\ell} )-f_0(\bx^*) =  \langle\nabla f_0\left(x^*\right), \bx^{\ell} -\bx^*\rangle + o(\|\bx^{\ell} -\bx^*\|).
\end{equation*}
Dividing the both sides of the above inequality by $\|\bx^{\ell} -\bx^*\|$ and letting $\ell\rightarrow \infty$, we have
\begin{equation}\label{grad0-d-less-0}
\langle\nabla f_0\left(\bx^*\right),\bd\rangle\leq 0.
\end{equation}
We note that a P-stationary point is also a KKT point, so ${\bnu ^*_{\Gamma_*}\in N_{{\X_*}}( \bx_{\Gamma_*}^{*})}$ by \eqref{eq: BKKT}, leading to $\langle \bnu^*_{\Gamma_*}, \bx^{\ell}_{\Gamma_*} -\bx^*_{\Gamma_*} \rangle \leq  0$ and hence $\langle \bnu^*_{\Gamma_*}, \bd_{\Gamma_*}  \rangle=\lim_{\ell\to \infty}\langle \bnu^*_{\Gamma_*}, \bd^{\ell}_{\Gamma_*} \rangle \leq  0$. This indicates ${\bd_{\Gamma_*} \in T_{\X_*}\left( \bx^{*}_{\Gamma_*} \right)}$ and thus $\langle\bnu^*_{\Gamma_*}, \bd_{\Gamma_*} \rangle\leq 0$ . By the definitions of $(\theta ^1,\beta^1,\theta ^2,\beta^2)$ and  $\left\{\bx^{\ell} \right\} \subset \F  \cap \mathbb{S} \cap \X$, it follows  
\begin{align*}
f_i(\bx^{\ell}) - f_i(\bx^*) = f_i(\bx^{\ell})\leq  0, ~~&\forall~i\in \theta ^1\cup \beta^1,\\
\langle \ba_i,\bx^{\ell}  - \bx^* \rangle =\langle \ba_i,\bx^{\ell}\rangle-b_i-(\langle \ba_i,\bx^* \rangle-b_i) \leq  0, ~~&\forall~i\in \theta ^2\cup \beta^2.
\end{align*}
Based on these, the similar reasoning to show \eqref{grad0-d-less-0} enables us to derive that 
\begin{equation}
    \begin{cases}
    \label{eq: d in C1}
	\left< \nabla f_i\left( \bx^* \right) ,\bd \right> \leq  0, &i\in \theta ^1\cup \beta^1,\\
	\langle \ba_i,\bd \rangle \leq  0, &i\in \theta ^2\cup \beta^2,\\
	d_{i} \in T_{X_{i}}\left( x_{i}^{*} \right) ,& i\in\theta ^3\cup \beta^3.
\end{cases}
\end{equation}
Next,  we further assert that
\begin{equation}
    \begin{cases}
    \label{eq: d in C2}
	\left< \nabla f_i\left( \bx^* \right) ,\bd \right> =  0, &i\in \beta^1,\\
	\langle \ba_i,\bd \rangle =  0, &i\in \beta^2,\\
	d_{i} =0 ,& i\in \beta^3.
\end{cases}
\end{equation}
Suppose there is an $i_0 \in \beta^1$ such that $\left< \nabla f_{i_0}\left( \bx^* \right) ,\bd \right> < 0$. Then we have 
\begin{align*}
0&=\langle \nabla _xL\left( \bx^*,\bmu ^*,\blambda ^* \right) + \bnu^*,\bd\rangle\\[1ex]
&=\left\langle \nabla f_0\left(\bx^*\right), \bd\right \rangle+\sum_{i \in \beta^1} \mu_i^* \left\langle \nabla f_i\left(\bx^*\right), \bd\right \rangle + \sum_{i \in \beta^2} \lambda_i^* \left\langle\ba_i, \bd\right \rangle+\sum_{i \in \beta^3}\nu_i^*d_{i}\\
&\leq \left\langle \nabla f_0\left(\bx^*\right), \bd\right \rangle+ \mu_{i_0}^* \left\langle \nabla f_{i_0}\left(\bx^*\right), \bd\right \rangle\\[1ex]
&< \left\langle \nabla f_0\left(\bx^*\right), \bd\right \rangle\leq0,
\end{align*}
where the first equality is from \eqref{eq2.5-1} by the P-stationarity of $\bx^*$ and $\bd_{\overline{\Gamma}_*}=0$ if $\|\bx^*\|_0=s$, the second equality and the first inequality are from \eqref{index set 3}, and the last inequality is from \eqref{grad0-d-less-0}. This contradiction shows the first equation in \eqref{eq: d in C2}. Then the same reasoning enables the last two equations in \eqref{eq: d in C2}. Combining \eqref{eq: d in C1}, \eqref{eq: d in C2}, and $\bd_{\overline{\Gamma}_*}=0$ when $ \|\bx^*\|_0=s$, we conclude that 
\begin{align}\label{d-in-C}
\bd(\neq0)\in\begin{cases}
		 \mathbb{C}\left( \Y^*\right)\cap \mathbb{R}_{\Gamma_*}^n,& \text{if}~
		 \|\bx^*\|_0=s,\\ 
		 \mathbb{C}\left( \Y^*\right),& \text{if}~
		 \|\bx^*\|_0<s.\\
		 \end{cases} 
\end{align}
 Additionally, it follows from \eqref{index set 3} and \eqref{fxl-smaller-fx*} that
\begin{align}\label{Ll-less-L*}
L(\bx^{\ell} , \bmu^*, \blambda^*)=f_0(\bx^{\ell} )+\sum_{i \in \beta^1} \mu_i^* f_i(\bx^{\ell} )+\sum_{i \in \beta^2} \lambda_i^* (\langle\ba_i, \bx^*\rangle - b_i) \leq  f_0\left(\bx^*\right)=L\left(\bx^*, \bmu^*, \blambda^*\right).
\end{align}
We note that $\supp(\bx^{\ell}) = \Gamma_*$ when $\|\bx^*\|_0=s$ and $\bg^*_{\overline{\Gamma}_*}=0$ by \eqref{eq2.5-1} when $\|\bx^*\|_0<s$. Then the similar reasoning to show \eqref{g-x-x*} can show  $\langle \bg^*, \bx^{\ell} -\bx^* \rangle \geq  0$.
 This together with \eqref{Ll-less-L*} result  in
            \begin{equation*}
\begin{aligned}
0   &\geq  2L(\bx^{\ell} , \bmu^*, \blambda^*)- 2L(\bx^*, \bmu^*, \blambda^*)  \\[1ex]
 &=  \langle \bx^{\ell} -\bx^*, \H^*(\bx^{\ell} -\bx^*)  \rangle+ 2\langle  \bg^*, \bx^{\ell} -\bx^*\rangle  \\[1ex]
&\geq   \langle \bx^{\ell} -\bx^*, \H^*(\bx^{\ell} -\bx^*).
\end{aligned}
\end{equation*}
Dividing the both sides of the above inequality by $\|\bx^{\ell} -\bx^*\|^2$ and letting $\ell\rightarrow \infty$, we have
$$
	 \langle \bd, \H^* \bd \rangle\leq  0, 
$$
for any $\bd$ satisfying \eqref{d-in-C}, which contradicts to \eqref{eq2.14}. The proof is completed.
\end{proof}

\section{Nonsingularity of Generalized Jacobian }
\label{Generalized Jacobian non-singularity}
In this section, we aim to analyze the nonsingularity of the generalized Jacobian matrix $\partial F$ of stationary equations (\ref{eq2.11}).
	Given ${\Y^*=(\bx^*,\bnu^*,\bmu^*,\blambda^*)\in \mathbb{R} ^{p }}$ with $p:=2n+k+m$, we select one index set ${T \in \mathbb{T}_{\tau}\left(\Y^*\right)}$. One can verify that $F\left( \cdot ;T \right)$ is strongly semismooth at $\Y^*$.  In fact, the first, second, and fourth rows of \eqref{eq2.011} are twice continuously differentiable and thus strongly semismooth, while the strongly semismoothness of the last two rows follows from the strongly semismoothness of the Fischer–Burmeister function \cite[Lemma 1] {sun1999ncp}. The projection operator in the third row of \eqref{eq2.011} is a piecewise affine operator, and according to \cite[Proposition 7.4.7] {facchinei2003finite}, it is also strongly semismooth.  Any generalized Jacobian matrix $\W \in \partial F\left( \Y^*; T\right) $ takes the form of
\begin{equation}
	\label{eq3.1}
	\W:=\left[ \begin{matrix}
		\H^* _{TT}&	\H^* _{T\overline{T}}&		\B^*_{T:}&		  \A_{:T}^{\top}  &		\I  & 0  \\
	0& \I &		0&		0&		0&0  \\
	 \I-\C^* &0 &		0&		0&		-\C^* & 0 \\		
		0& 0&		0&		0&		0&  \I \\
		\U_1^*(\B^*_{T:})^{\top}&\U_1^*(\B^*_{\overline{T}:})^{\top}&		\V_1^*&		0&		0& 0  \\
		\U_2^*\A_{:T}&	\U_2^*\A_{:\overline{T}}&	0&		\V_2^*&		0& 0 		\\		
		
	\end{matrix} \right] \in \mathbb{R} ^{p \times p}.
\end{equation} 
The sub-matrices in \eqref{eq3.1} are defined by 
\begin{equation}
\label{def-DCUV}
\begin{aligned}
& \B^* :=\left[ \begin{matrix}
	\nabla f_1(\bx^*)&		\cdots&		\nabla f_k(\bx^*)\\
\end{matrix} \right]\in \mathbb{R} ^{n \times k},\\   
& \C^* := \mathrm{Diag}\left( \mathbf{c}\right) \in \partial \Big(\Pi _{{\X_T}}\left( \bx^*_{T}+ \bnu^*_{T}\right) \Big)\in \mathbb{R} ^{s \times s},\\
&\U_i^*:=\mathrm{Diag}\left( \mathbf{u}^{t}\right)\in \mathbb{R} ^{k \times k},\\[1ex]
& \V_i^*:=\mathrm{Diag}\left( \mathbf{v}^{t}\right)\in \mathbb{R} ^{m \times m}, ~~i = 1,2,
\end{aligned}
\end{equation}
where  $\mathrm{Diag}(\bu)$ is a diagonal matrix with diagonal entries formed by $\bu$, and for $t = 1,2,$
\begin{align}
\label{def-cdu}
\left( u_{i}^{t}, v_{i}^{t} \right) \begin{cases}
	=\left( 0,-1 \right),          & \text{if}~i\in \eta ^t,\\	
	\in \left\{ \left( \alpha ,\beta \right) :\|( \alpha ,\beta)-(1,-1)\|^2\leq  1 \right\},          & \text{if}~i\in \theta ^t,\\
=\left( 1,0 \right),          & \text{if}~i\in \beta^t,\\
\end{cases}\qquad c_i\begin{cases}
	=1,          & \text{if}~i\in \eta ^3(T),\\
	\in [0,1], & \text{if}~i\in \theta ^3(T),   \\
	=0,    &   \text{if}~i\in \beta^3(T).
\end{cases}
\end{align}
After simple elementary operations, we can get the following reduced   matrix
\begin{equation}
	\label{eq3.2}
	\G:=\left[ \begin{matrix}
		\H^* _{TT}&		\B^*_{T:}&		  \A^{\top}_{:T}  &		\I  \\
		 \I-\C^*  &		0&		0&		-\C^* \\
		 \U_1^*(\B^*_{T:})^{\top}&		\V_1^*&		0&		0 \\
		\U_2^*\A_{:T}&		0&		\V_2^*&		0 		\\
	\end{matrix} \right] \in \mathbb{R} ^{ q \times q},
\end{equation} 
where $q:=2s+k+m$, which has the same nonsingularity as $\W$ in \eqref{eq3.1}. 
Based on these notation, we establish the CD-regularity of ${F\left( \Y^* ;T \right)}$ for any given $T$ as follows. We say ${F\left( \Y^* ;T \right)}$ is CD-regular if any ${\W\in\partial F\left( \Y^* ;T \right)}$ is nonsingular.	
\begin{theorem}
	[{\bf CD-regularity}]\label{thm_CD-regularity of F*}
	Let $\bx^*$ be a $P$-stationary point of (\ref{eq: SCQP}) with ${\tau>0}$. Then for any given ${T\in \mathbb{T}_{\tau}(\Y^*)}$, ${F\left( \Y^* ;T \right)}$ is CD-regular if Assumption \ref{assum_RLICQ} holds at $\Y^*$ and
     \begin{equation}
		\label{eq: strong SOC}
		\langle \bd, \H^*  \bd \rangle > 0, \forall 0 \ne \bd\in \mathbb{Q}\left( \Y^*;T \right),
     \end{equation} 
     where $\mathbb{Q}$ is a cone  defined by
\begin{equation*}
	\label{eq: cone of feasible directions}
	\mathbb{Q}\left( \Y^*;T \right) :=	\left\{
\bd\in \mathbb{R} ^n:\begin{array}{r}		
	\langle \nabla f_i({\bx}^* ), \bd\rangle=0,~i\in \beta^1
	\\
     \langle \ba_i, \bd \rangle=0,~i\in \beta^2
	\\
	 d_{i} =0,~ i\in \beta^3(T) 
        \end{array} \right\}.
\end{equation*}
\end{theorem}
\begin{proof} Given any ${T\in \mathbb{T}_{\tau}(\Y^*)}$, to show the nonsingularity of ${\W\in \partial F\left( \Y^*;T\right)}$ is equivalent to show that the following homogeneous system has a unique solution 0,
	$$
	\W\left( \bd^x;~ \bd^\nu;~  \bd^\mu;~  \bd^\lambda \right)=0.
	$$
	This condition is equivalent to
		\begin{align}\label{dx-barT-dnu-barT-0}
		\begin{cases}
		0= \bd^x_{\overline{T}},\\	
		0= \bd^\nu_{\overline{T}},\\
	0=\G\left( \bd^x_{T};~\bd^\nu_T;~ \bd^\mu;~ \bd^\lambda\right),
	\end{cases}
\end{align}
which by \eqref{eq3.2}, we  get
	\begin{equation}
		\label{eq3.4}
		\begin{cases}
			0=\H^*_{TT}\bd^x_{T}+  \bd^\nu_{T}+\B^*_{T:}\bd^{\mu}+\A^{\top}_{:T}\bd^{\lambda} ,\\
			0=(\I-\C^*)\bd^x_{T} -\C^*\bd^\nu_{T},\\
			0=\U _1^* (\B^*_{T:})^{\top}\bd^x_{T}+\V_1^*\bd^{\mu},\\ 
			0=\U _2^* \A_{:T}\bd^x_{T} +\V_2^*\bd^{\lambda}.
		\end{cases}
	\end{equation}
	Since $\Y^*$ satisfies (\ref{eq2.5}) and \eqref{def-cdu}, we have
	$$
	\begin{aligned}
		\eta^t & = \{i  : \left(u_i^{t}, v_i^{t}\right)=(0,-1) \}, &&\eta ^3(T)= \{ i\in T: c_i=1  \}, \\[1ex]
            \theta^t & = \{i  :\left\|\left(u_i^{t}, v_i^{t}\right)-(1,-1)\right\|^2\leq  1 \}, \qquad&&\theta^3(T)= \{ i\in T: c_i\in[0,1] \},\\[1ex]
		\beta^t & = \{i  : \left(u_i^{t}, v_i^{t}\right)=(1,0) \},&& \beta^3(T)= \{ i\in T: c_i=0  \},
	\end{aligned}
	$$
where $t=1,2$.	Based on these indices, we further define
\begin{align*}
\widetilde{\eta}^t&:=\{i\in\theta ^t:  \left(u_i^{t}, v_i^{t}\right)=(0,-1)\}\cup \eta^t,&&\widetilde{\eta}^3 := \{i\in\theta^3(T): c_i = 1\}\cup  \eta^3(T),\\[1ex]
\widetilde{\theta}^t&:= \{i  :\left\|\left(u_i^{t}, v_i^{t}\right)-(1,-1)\right\|^2 < 1 \},&&\widetilde{\theta}^3:=\{i\in\theta^3(T): c_i\in(0,1)\},\\[1ex]
		\widetilde{\beta}^t &: = \{i\in\theta ^: \left(u_i^{t}, v_i^{t}\right)=(1,0) \} \cup \beta^t,&& \widetilde{\beta}^3:= \{i\in\theta^3(T): c_i = 0\}\cup  \beta^3(T).
\end{align*}
One can observe that $u_i^{t}> 0$ and $v_i^{t}<0$ for any $i\in\widetilde{\theta}^t, t=1,2$.  These and the last three equations in \eqref{eq3.4} lead to the facts in Table \ref{tab:facts-eta-theta-bata},
\begin{table}[H]  
\renewcommand{\arraystretch}{1.25}\addtolength{\tabcolsep}{15pt}
\caption{Conditions on $(\widetilde{\eta}^t, \widetilde{\theta}^t, \widetilde{\beta}^t),t=1,2,3.$ \label{tab:facts-eta-theta-bata}}
\centering  
\begin{tabular}{lrrr}\hline
 & $i \in \widetilde{\eta}^t$ & $i \in \widetilde{\theta}^t$ & $i \in \widetilde{\beta}^t$ \\\hline 
		 $t=1$ & $d^{\mu}_i=0$ & $
		 ((\B^*_{T:})^{\top}\bd^x_{T})_i=- (v_i^1/{u_i^1}) d^{\mu}_i
		 $ & $ \langle (\nabla f_i(\bx^*))_T, \bd^x _T\rangle=0$ \\
		 
		$t=2$ & $d^{\lambda}_i=0$ & $
		\langle (\ba_i)_T, \bd^x_T\rangle=- (v_i^2/{u_i^2}) d^{\lambda}_i
		$ &  $\langle (\ba_i)_T,~ \bd^x_T\rangle=0$ \\
$t=3$ 	 & $ i\in T, ~d^{\nu}_i=0$ & $
i\in T,  d_i^x=  c_id^{\nu}_j/(1-c_i)$ &  $d_i^x=0$ \\
		\hline
\end{tabular}
\end{table}

The conditions in the third column of Table \ref{tab:facts-eta-theta-bata}, $\beta^t\subseteq\widetilde{\beta}^t, t=1,2$, and  $\beta^3(T)\subseteq\widetilde{\beta}^3$ yield that
\begin{align}
\label{eq: dx-in-QYT}
\bd^x \in \mathbb{Q}\left( \Y^*;T \right).\end{align}
Multiplying the both sides of the first equation in \eqref{eq3.4} by $\bd^x_T$ derives
\begin{align}
\label{eq: eq in CD regular}
\langle \bd^x_T, \H ^*_{TT} \bd^x_T \rangle +\sum_{i\in T\cap\widetilde{\theta}^3} \frac{c_i(d^{\nu}_j)^2}{1-c_i}-\sum_{i\in \widetilde{\theta}^1}  \frac{v_i^1(d^{\mu}_i)^2}{u_i^1}  -\sum_{i\in\widetilde{\theta}^2}\frac{v_i^2(d^{\lambda}_i)^2}{u_i^2} =0.
\end{align}
 which together with  \eqref{eq: strong SOC} and \eqref{eq: dx-in-QYT}  immediately results in $$\left(\bd^x_T,~\bd^{\nu}_{T\cap\widetilde{\theta}^3},~\bd^{\mu}_{\widetilde{\theta}^1},~\bd^{\lambda}_{\widetilde{\theta}^2}\right)=(0,0,0,0).$$
This recalling the first column of Table \ref{tab:facts-eta-theta-bata} further derives
 $$\left(\bd^{\nu}_{T\cap(\widetilde{\eta}^3\cup\widetilde{\theta}^3)},~\bd^{\mu}_{\widetilde{\eta}^1\cup\widetilde{\theta}^1},~\bd^{\lambda}_{\widetilde{\eta}^2\cup\widetilde{\theta}^2}\right)=(0,0,0).$$
  By inserting these values into the first equation in (\ref{eq3.4}), it holds that
	\begin{equation} \label{Id-Bd-Ad}
	 \I_{T (T\cap\widetilde{\beta}^3)}\bd^{\nu}_{(T\cap\widetilde{\beta}^3)}	+ \B^*_{T\widetilde{\beta}^1}\bd^{\mu}_{\widetilde{\beta}^1}+\A^{\top}_{\widetilde{\beta}^2T}  \bd^{\lambda}_{\widetilde{\beta}^2}   =0.
	\end{equation}
	One can check that  ${\widetilde{\beta}^1\subseteq \mathcal{A}_1(\bx^*)}$, ${\widetilde{\beta}^2\subseteq \mathcal{A}_2(\bx^*)}$,   $T\cap{\widetilde{\beta}^3\subseteq \mathcal{A}_3(\bx^*;T)}$, and $\Gamma_*\subseteq T$ for any $T \in \mathbb{J}_s(\bx^*)$. Hence, Assumption \ref{assum_RLICQ} indicates the following groups of vectors 
	are linearly independent,	 $$ \Big\{(\nabla f_i(\bx^*))_{T}:i\in \widetilde{\beta}^1\Big\} \cup \Big\{\left( \ba_{i} \right) _{T}: i\in \widetilde{\beta}^2\Big\} \cup \Big\{ \left( \mathbf{e}_{i}  \right) _{T}: i\in T\cap\widetilde{\beta}^3\Big\},$$ 
which by \eqref{Id-Bd-Ad} implies $(\bd^{\nu}_{(T\cap\widetilde{\beta}^3)}, \bd^{\mu}_{\widetilde{\beta}^1}, \bd^{\lambda}_{\widetilde{\beta}^2})=(0,0,0)$. Overall, $(\bd^x_T,\bd^\nu_T,\bd^\mu,\bd^\lambda)=0$, which combining the first two conditions in \eqref{dx-barT-dnu-barT-0} yields the conclusion. \end{proof}
In the sequel, we pay attention to point ${\Y=(\bx, \bnu, \bmu, \blambda)}$ with  ${\bx\in\mathbb{S}}$ in a neighbourhood of $\Y^*$. Therefore, we define a set by
$$\E:=\mathbb{S}\times\mathbb{R}^{n}\times\mathbb{R}^k \times \mathbb{R}^m.$$  
Moreover, we note that each $\nabla f_i(\bx)$ is Lipschitz continuous, so ${\nabla_x L(\bx, \bmu, \blambda) + \bnu}$ is locally Lipschitz continuous. Leveraging this fact, we build the following property of any  $\Y$ around $\Y^*$.
\begin{lemma}
	\label{lem2}
	 Let $\bx^*$ be a P-stationary point of (\ref{eq: SCQP}) with ${\tau>0}$. Then there is a neighbourhood $\mathbb{N}^*$ of $\Y^*$ such that for any $\Y \in \E\cap\N$,
	 \begin{equation}
	 	\label{eq3.7}
	 	\mathbb{T}_{\tau}(\Y) \subseteq \mathbb{T}_{\tau}\left(\Y^*\right),\qquad \Gamma_* \subseteq (\supp(\bx) \cap T),~ \forall~ T \in \mathbb{T}_{\tau}(\Y ).
	 \end{equation}
Consequently, $F\left( \Y^*;T\right) =0$ for any $T \in \mathbb{T}_{\tau}(\Y).$
\end{lemma}
\begin{proof} Since the second claim follows from \eqref{eq3.7} and \eqref{eq2.11}, we only prove \eqref{eq3.7}.  Let neighbourhood $\mathbb{N}^*$ be a sufficiently small region. Then \eqref{eq3.7} is true when $\left\|\bx^*\right\|_0=s$ due to $$\mathbb{T}_{\tau}(\Y)=\mathbb{T}_{\tau}\left(\Y^*\right)=\left\{\Gamma_*\right\}=\{\supp(\bx)\}.$$ Therefore, we only prove the case of $\left\|\bx^*\right\|_0<s$. One can easily check that for any $\Y \in \E\cap\N$,  
\begin{align}\label{Js-Js-Tt}
\Gamma_* \subseteq \supp(\bx),\qquad \mathbb{J}_s(\bx) \subseteq \mathbb{J}_s(\bx^*)=\mathbb{T}_{\tau}\left(\Y^*\right),
\end{align}
where the equality is from \eqref{T=J}.  If $\Gamma_*=\emptyset$, the conclusion holds evidently. Now we focus on $\Gamma_*\neq\emptyset$. This means that ${\min_{i\in\Gamma_*}|x_{i}|>0}$ because $\bx$ is close to $\bx^*$. Moreover, it follows from \eqref{eq2.5-1} that ${\bg^*+\bnu^*=0}$. As a result of the  locally Lipschitz continuity, $\tau\|\nabla_x L\left(\bx, \bmu, \blambda\right)+\bnu\|_{\infty}<\min_{i\in\Gamma_*}|x_i|$,   which indicates that
$$\Gamma_*\in T,~~ \forall~T\in \mathbb{T}_{\tau}(\Y).$$
This recalling the definition of $\mathbb{J}_s(\bx^*)$ enables us to conclude that $\mathbb{T}_{\tau}(\Y) \subseteq \mathbb{J}_s(\bx^*)$, which together with \eqref{Js-Js-Tt} shows the desired result. 
\end{proof}
\begin{theorem}
	\label{thm_locally CD-regularity}
	Let $\bx^*$ be a P-stationary point of (\ref{eq: SCQP}) with ${\tau>0}$ and the assumptions in Theorem \ref{thm_CD-regularity of F*} hold at $\Y^*$. Then there is a neighbourhood $\mathbb{N}^*$ of $\Y^*$ such that the following statements are valid for any $\Y \in \E\cap\N$.
	\begin{itemize}[leftmargin=15pt]
	\item[1)] Every $\W \in \partial F(\Y ; T)$  is nonsingular for each given $T \in \mathbb{T}_{\tau}(\Y)$.
	\item[2)] There is a constant $C^*>0$ such that
	\begin{align}\label{thm_locally CD-regularity-equ}
	\left\| \W ^{-1}\right\| \leq C^*,~~ \forall~\W \in \partial F(\Y ; T),~\forall T \in \mathbb{T}_{\tau}(\Y).  
	\end{align}
	where $\left\| \W \right\|$ represents the spectral norm of $\W$. 
\end{itemize}
\end{theorem}
\begin{proof}
	1) By invoking Theorems \ref{thm_CD-regularity of F*}, for each given ${T_* \in \mathbb{T}_{\tau}(\Y^*)}$, any ${\W^* \in \partial F\left( \Y^*;T_*\right)}$ is nonsingular.  Consider any point  $\Y \in \E\cap\N$. Lemma \ref{lem2} means that ${\mathbb{T}_{\tau}(\Y)\subseteq\mathbb{T}_{\tau}(\Y^*)}$. As a result, for each given ${T \in \mathbb{T}_{\tau}(\Y)}$,  any $\W^* \in \partial F\left( \Y^*;T\right) $ is nonsingular.  We first prove that for each given $T \in \mathbb{T}_{\tau}(\Y)$, every ${\W \in \partial F(\Y ; T)}$  is nonsingular and there is a constant $C_T>0$ such that ${\|\W^{-1}\|\leq C_T}$. If this is not true, then given such $T$, there is a sequence ${\Y^{\ell}(\in\mathbb{N}^*)\rightarrow \Y^*}$ and $\W^{\ell} \in \partial F\left( \Y^{\ell};T \right)$ satisfying either all $\W^{\ell}$ are singular or $\|(\W^{\ell})^{-1}\| \rightarrow \infty$. The locally Lipschitz continuity of $F\left(\cdot;T\right)$  means $\partial F\left( \cdot;T\right)$ is bounded in $\mathbb{N}^*$. By passing to a subsequence, we may assume ${\W^{\ell} \rightarrow {\W}^*}$. Then ${\W}^*$ is singular, contradicting to the nonsingularity of ${\W}^*$   
	
	2) We note that there are finitely many  $T \in \mathbb{T}_{\tau}(\Y)$. By setting ${C^*=\max_{T \in \mathbb{T}_{\tau}(\Y)}C_T}$, condition \eqref{thm_locally CD-regularity-equ} follows immediately.
\end{proof}

\section{A Semi-smooth Newton Algorithm}
\label{A semi-smooth Newton algorithm}
In this section, we leverage a semi-smooth Newton-type method to solve stationary equations \eqref{eq2.11} for problem (\ref{eq: SCQP}), before which we define some notation for the ease of reading.  Given a point ${\Y=(\bx, \bnu, \bmu, \blambda)\in    \mathbb{R}^n\times \mathbb{R}^{n} \times\mathbb{R}^k \times \mathbb{R}^m }$ and $T\in\T_{\tau}(\Y)$,  we define $\by$ and a direction $\bd$ by
\begin{equation}
\label{vectorize-Y}
\begin{aligned}
\by&:=\left( \bx_T;~\bx_{\overline{T}};~\bnu_T;~\bnu_{\overline{T}};~\bmu;~\blambda\right)~~~\in\R^p,\\
\bd&:=\left(  \bd^x_T;~\bd^x_{\overline{T}};~\bd^\nu_T;~\bd^\nu_{\overline{T}};~\bd^\mu;~\bd^\lambda  \right) \in\R^p, 
\end{aligned} \end{equation}
namely, $\by$ is a vectorization of $\Y$. In addition,  by letting
\begin{align*}
J&:=n+\overline{T},  \qquad K:=[p]\setminus (\overline{T}\cup J),
\end{align*} 
one can write $\bd=(\bd_{\overline{T}};~\bd_{J};~\bd_{K})$, where
\begin{align*}
 \bd_{\overline{T}}=\bd^x_{\overline{T}}\in\R^{n-s},\qquad
 \bd_{J}=\bd^\nu_{\overline{T}}\in\R^{n-s},\qquad
 \bd_{K}=\left(  \bd^x_T;~\bd^\nu_T;~\bd^\mu;~\bd^\lambda  \right)\in\R^{q}.
\end{align*}
\subsection{Algorithm design}
To employ the Newton method, we need to find a solution to a system of linear equations. For (\ref{eq: SCQP}),  given current $\Y^{\ell}$ and $T_{\ell}\in\T_s(\Y^{\ell})$ at step $\ell$, we aim to find a direction $\bd^\ell$ by solving stationary equations $\W^{\ell}\bd^{\ell}= -F(\Y^{\ell};T_{\ell}) $. In fact, it is easy to see from \eqref{eq3.1} and \eqref{eq3.2} that
		\begin{equation}\label{Wd=FYT}
		\W^{\ell}\bd^{\ell}= -F(\Y^{\ell};T_{\ell}) ~~\Longleftrightarrow~~	\begin{cases}
				\bd_{\overline{T}_\ell}^{\ell}=-\bx_{\overline{T}_{\ell}}^{\ell},\\
				\bd_{J_\ell}^{\ell}=-\bnu_{\overline{T}_{\ell}}^{\ell},\\
				\G^{\ell}\bd^{\ell}_{K_\ell} =\D^{\ell}\bx_{\overline{T}_{\ell}}^{\ell} - \left({F}(\Y^{\ell};T_{\ell})\right)_{K_\ell},		\end{cases}
		\end{equation}
where $\D^{\ell}$ is a sub-matrix of $\W^{\ell}$  defined by
\begin{equation}
	\label{eq3.1-D}
	\D^{\ell}:=\left[ \begin{matrix}
	 \H^\ell_{T_\ell\overline{T}_\ell} \\
0 \\
\U_1^\ell(\B^\ell_{\overline{T}_\ell:})^{\top}\\
\U_2^\ell\A_{:\overline{T}_\ell}\end{matrix} \right] \in \mathbb{R} ^{q\times (n-s)}.
\end{equation}

\begin{algorithm}[!t]
	\caption{{SNSQP}: Semi-smooth Newton algorithm for (\ref{eq: SCQP})} \label{Alg1:SMN} 
	 Initialize $\Y^0$, ${\tau > 0}$, ${\varepsilon> 0}$, ${\rho \in (0,1)}$,   and ${\sigma \in (0,1/2)}$. Select ${T_0 \in {\mathbb{T}}_{\tau}(\Y^0)}$,  and set $\ell\Leftarrow0$.
	
	\While{$\|F(\Y^{\ell};T_{\ell})\| > \varepsilon $}{
	
	 Choose $\W^{\ell} \in \partial F(\Y^{\ell};T_{\ell})$.
	 
	 Compute $ \bd_{\overline{T}_\ell}^{\ell}=-\bx_{\overline{T}_{\ell}}^{\ell}$ and $\bd_{J_\ell}^{\ell}=-\bnu_{\overline{T}_{\ell}}^{\ell}$.
	 
	 Compute $\bd^{\ell}_{K_\ell}$ by solving \eqref{eq: newton step1} if it is solvable and by solving \eqref{eq: newton step2}  otherwise.

     Find the smallest non-negative integer $t_{\ell}$ such that
	\begin{equation}\label{line-search-step}
	\Psi \left( \by^{\ell}+ \bd^{\ell}(\rho^{t_{\ell}});T_{\ell} \right)\leq  \Psi \left( \by^{\ell} ;~T_{\ell} \right)+\sigma \rho ^{t_{\ell}} \left\langle F(\Y^{\ell};~T_{\ell}), \W^{\ell} \bd^{\ell}\right\rangle
	\end{equation}
	Update $\by^{{\ell}+1}=\by^{\ell}+ \bd^{\ell}(\rho^{t_{\ell}})$ and ${T_{\ell+1}} \in {\mathbb{T}}_{\tau}(\Y^{\ell+1})$. Set  $\ell \Leftarrow {\ell} + 1$.
	}
\end{algorithm}

Hereafter, let  ${\H^{\ell}=:\nabla_{xx}^2 L\left( \bx^\ell,\bmu ^\ell,\blambda ^\ell\right)}$.  Matrices $\W^{\ell}$ and $\G^{\ell}$ are calculated similarly to \eqref{eq3.1} and \eqref{eq3.2}, where $(\B^\ell, \C^\ell, \U_1^\ell, \U_2^\ell, \V_1^\ell, \V_2^\ell)$ are computed using \eqref{def-DCUV} by replacing $\Y^*$ by $\Y^\ell$. The equivalence in \eqref{Wd=FYT} enables a significant dimensional reduction, from $(p\times p)$ to $(q\times q)$, making large-scale computation possible. However,  when matrix $\G^{\ell}$ is in a bad condition,  solving the following equations may yield an unfavorable direction $\bd^{\ell}_{K_\ell}$, 
	\begin{equation}\label{eq: newton step1}
	\G^{\ell}\bd^{\ell}_{K_\ell} =\D^{\ell}\bx_{\overline{T}_{\ell}}^{\ell} - \left({F}(\Y^{\ell};T_{\ell})\right)_{K_\ell}.
		\end{equation}
 To overcome such a drawback, we turn to solve the following equations as a compensation, 
	\begin{equation}
		\label{eq: newton step2}
			\left({\G^{\ell}}^{\top}\G^{\ell} + \kappa_{\ell}\I\right)\bd_{K_\ell}^{\ell}= {\G^{\ell}}^{\top}\left(\D^{\ell}\bx_{\overline{T}_{\ell}}^{\ell} - \left({F}(\Y^{\ell};T_{\ell})\right)_{K_\ell}\right),
\end{equation}
where ${\kappa_{\ell}>0}$ is a decreasing scalar along with $\ell$ rising and ${\lim_{\ell\to\infty}\kappa_{\ell}=0}$.  For example, ${\kappa_{\ell}=0.01/\ell}$ is used in the subsequent numerical experiment.  Finally, given a step size ${\alpha>0}$, we define 
\begin{align}
\bd^{\ell}(\alpha):=
\left( 
\bd^{\ell}_{\overline{T}_\ell};~
 \bd^{\ell}_{J_\ell};~
\alpha \bd^{\ell}_{K_\ell} \right)
\end{align}
to ensure point $\bx^{\ell+1}\in\mathbb{S}$. This is because from  \eqref{Wd=FYT}, for any $\alpha>0$, we have
\begin{align}\label{x-ell+1-feasible}
 \by^{\ell+1}=\by^{\ell} + \bd^{\ell}(\alpha) &=
\left(
\bx^{\ell}_{\overline{T}_\ell} + \bd^{\ell}_{\overline{T}_\ell} ;~
\bnu^{\ell}_{J_\ell} + \bd^{\ell}_{J_\ell};~
\by^{\ell+1}_{K_\ell} + \alpha \bd^{\ell}_{K_\ell}
\right) \nonumber\\
& =
\left( 
0 ;~
0 ;~
\by^{\ell+1}_{K_\ell} + \alpha \bd^{\ell}_{K_\ell} 
\right).
\end{align}
We observe that $\by^{\ell+1}_{\overline{T}_{\ell}}$  corresponds to $\bx^{\ell+1}_{\overline{T}_{\ell}}$, which by ${|\overline{T}_{\ell}|=n-s}$ results in $\bx^{\ell+1}\in\mathbb{S}$. Consequently, $\supp(\bx^\ell)\subseteq T_{\ell-1}$, leading to $|T|:=|T_{\ell-1}\cap\overline{T}_{\ell}|\leq s$. This suffices to $\D^{\ell}\bx_{\overline{T}_{\ell}}^{\ell}=\D^{\ell}_{:T}\bx_{T}^{\ell}$ in \eqref{Wd=FYT}, which significantly reduces the computational complexity of evaluating $\D^{\ell}\bx_{\overline{T}_{\ell}}^{\ell}$ from $O(q(n-s))$  to  $O(qs)$, achieving a substantial efficiency gain. Overall, the complexity of solving  \eqref{Wd=FYT} is about $$O(s^3+ks^2+qs).$$
The developed algorithmic framework is presented in Algorithm \ref{Alg1:SMN}, which is called {SNSQP}, an abbreviation for the Semi-smooth Newton for (\ref{eq: SCQP}). To make the algorithm steady, we adopt an Armijo line search to determine the step size at every step based on a merit function,
\begin{equation}
	\label{eq: merit function}
	\Psi \left( \by;T \right) :=\frac{1}{2}\left\| F\left( \Y;T\right) \right\| ^2.
\end{equation}

Finally, it is important to point out that the stationary equations, \eqref{eq2.11}, involve an unknown set $T$. If this set were known in advance, one could employ the semismooth Newton-type method to solve these equations with a fixed $T$ and establish the local convergence rate in a way provided in \cite{qi1997semismooth}. However,  set $T$ may change from one iteration to the next. A different set $T_\ell$ leads to a distinct system of equations, $F(\Y;T_\ell) = 0$. Consequently, at each step, the algorithm computes a Newton direction for a different system of equations, rather than a fixed one. This is where the standard proof for the quadratic rate of convergence does not apply to our case. Nevertheless, we establish this property for Algorithm \ref{Alg1:SMN} under the assumptions in Theorem \ref{thm_CD-regularity of F*}, as outlined below. 
\subsection{Local convergence rate} 
 In the sequel, we use notation $\|\Y\|^2$ to denote $$\|\Y\|^2=\|\bx\|^2+\| \bmu\|^2+\|\blambda\|^2+\| \bnu\|^2.$$
Similar to $\by$ defined in \eqref{vectorize-Y}, let $\bz\in\R^p$ be the vectorization of $\Z\in\E$. 
\begin{corollary}
\label{corollary: descent property of F} Let $\bx^*$ be a P-stationary point of (\ref{eq: SCQP}) with ${\tau>0}$ and the assumptions in Theorem \ref{thm_CD-regularity of F*} hold at $\Y^*$. Then for any given $\epsilon\in(0,1)$, there is a neighbourhood $\N_\epsilon$ of $\Y^*$ such that, for any point $\Y\in\E \cap\N_\epsilon$,  
\begin{equation}
\label{eq: equation of F}
    \left\| F\left(\Z;T \right) \right\| \leq  \epsilon \left\| F\left( \Y;T \right) \right\|, ~~\forall~\W \in \partial F(\Y ; T),~\forall~T \in \mathbb{T}_{\tau}(\Y).
\end{equation}
where $\bz:=\by- \W^{-1}F\left(\Y;T\right)$ is the vectorization of $\Z$. 
\end{corollary}
\begin{proof} Let $\N$ be a sufficiently small neighbourhood  of $\Y^*$ and consider $\Y\in\E\cap\N$.  For each given $T \in \mathbb{T}_{\tau}(\Y)$, every $\W \in \partial F(\Y ; T)$  is nonsingular from Theorem \ref{thm_locally CD-regularity}. Then it follows,
    \begin{equation}
        \label{eq: Locally quadratic convergence}
		\begin{aligned}
		\left\| \Z  - \Y^* \right\| 	=~& \left\| \by -\W^{-1}{F}( \Y;T ) - \by^* \right\| 
			\\
			 =~&\left\| \W^{-1}  \left( {F}( \Y;T ) - {F}\left( \Y^*;T\right) -\W  ( \by - \by^*) \right)\right\| 
			\\
			 \leq~  &\left\| \W^{-1} \right\| \left\| {F}( \Y;T ) - {F}\left( \Y^*;T \right) -\W( \by - \by^* ) \right\| 
			\\
			 =~&O( \| \Y-\Y^* \|^2) ,
		\end{aligned}
    \end{equation}
    where the second equality is from Lemma \ref{lem2} that ${F(\Y^*;T)=0}$ for any ${T \in \mathbb{T}_{\tau}(\Y)}$ and the last equality holds due to the strongly semismoothness of $F(\cdot;T)$ at $\Y^*$ for given $T$. The above condition means that $\Z$ also lies in $\N$, which by semismooth implying B-differentiability, we have
    \begin{equation}
        \label{eq: B-diff}
        \left\| {F}( \Z;T ) - {F}\left( \Y^*;T\right) - {F}'\left( \Y^*;T; \Z - \Y^* \right) \right\| = o( \| \Z-\Y^* \|)
    \end{equation}
where ${F}'\left( \Y^*;T; \Z - \Y^* \right)$ is the directional derivative of the function ${F}\left( \cdot;T\right)$ at point $\Y^*$ in direction $\Z - \Y^*$. Now we can claim that given $\epsilon\in(0,1)$,  there is a smaller neighbourhood $\N_\epsilon\subseteq\N$ such that, for any $\Y\in\N_\epsilon\cap\E$,
      \begin{equation}
        \label{eq: eq3 about z-z*}
        \begin{aligned}
        \left\| {F}( \Z;T ) - {F}'\left( \Y^*;T; \Z - \Y^* \right) \right\| \leq   \| \Z - \Y^* \|,\\
        \left\|  \Z - \Y^*  \right\| \leq  \kappa_* \| \Y - \Y^* \|,
  		\end{aligned}
  \end{equation}
where $\kappa_*:= {\epsilon}/({C^*(c^*+1)+\epsilon})$, the first inequality is from ${F(\Y^*;T)=0}$ and
\eqref{eq: B-diff}, the second one is due to \eqref{eq: Locally quadratic convergence}, and $c^*:=\sup_{\Y\in \N_\epsilon}\sup_{\W\in\partial F(\Y;T)}\|\W\|$.
The above conditions derive that
 \begin{equation*}
		\begin{aligned}
			\| \Y - \Y^* \| & \leq  \| \Z - \Y \| + \|\Z - \Y^* \|
			\\
			& = \| \W^{-1}{F}( \Y;T ) \| + \|\Z - \Y^* \| 
                \\
			& \leq  C^*\| {F}( \Y;T ) \| + \|\Z - \Y^* \|
			\\
			& \leq  C^*\| {F}( \Y;T ) \| + \kappa_* \|\Y - \Y^* \| ,
		\end{aligned}
  \end{equation*}
 where  the first inequality is from \eqref{thm_locally CD-regularity-equ}.  This leads to
    \begin{align*}
       \| \Y - \Y^* \|  \leq  \frac{ \epsilon}{(c^*+1)\kappa_*} \| {F}( \Y;T ) \| ,
    \end{align*}
 which by  \eqref{eq: eq3 about z-z*} further results in
    \begin{equation*} 
		\begin{aligned}
			\| F(\Z;T) \| & \leq  \left\| {F}'\left( \Y^*;T; \Z - \Y^* \right) \right\| +  \| \Z - \Y^* \|\\[1ex]
			& \leq  (c^* + 1)\| \Z- \Y^* \| 
			\\
			& \leq   (c^* + 1)\kappa_*   \| \Y - \Y^* \|\\
			& \leq \epsilon\| F({\Y};T) \|,
		\end{aligned}
  \end{equation*}
showing the desired result.
\end{proof}

\begin{theorem}
	\label{thm_convergen rate}
	Let $\bx^*$ be a P-stationary point of (\ref{eq: SCQP}) with ${\tau>0}$ and the assumptions in Theorem \ref{thm_CD-regularity of F*} hold at $\Y^*$. Then in a neighborhood of $\Y^*$, Algorithm \ref{Alg1:SMN}  admits full Newton steps and converges to $\Y^*$ quadratically.
\end{theorem}

\begin{proof}
	Let $\N$ be a sufficiently small neighbourhood  of $\Y^*$ and $\{\Y^{\ell}\}$ be the sequence generated by Algorithm \ref{Alg1:SMN} in $\N$. According to \eqref{x-ell+1-feasible}, $\bx^{\ell+1}\in\mathbb{S}$ and thus $\Y^{\ell}\in\E \cap\N$. By Lemma \ref{lem2}, 
 we have $T_{\ell} \in \mathbb{T}_{\tau}(\Y^{\ell}) \subseteq \mathbb{T}_{\tau}\left(\Y^*\right)$ . Recalling Theorem \ref{thm_locally CD-regularity}, any $\W^{\ell} \in \partial F(\Y^{\ell};T_{\ell})$ is nonsingular and $\|(\W^{\ell})^{-1}\|$ is bounded. Consequently, equations  \eqref{Wd=FYT} or \eqref{eq: newton step1} are solvable. Now, we prove that the full Newton steps always admit, namely $\rho^{t_\ell}=1$. Let $\bz^{\ell} := \by^{\ell} + \bd^{\ell}(1) =\by^{\ell} + \bd^{\ell}$ be the vectorization of $\Z^\ell\in\E$. Given $\sigma\in(0,1/2)$, it follows from Corollary \ref{corollary: descent property of F} that 
     $$
     \left\| F ( \Z^{\ell};T_{\ell} ) \right\| \leq  \sqrt{1-2\sigma} \left\| F ( \Y^{\ell};T_{\ell}  ) \right\|,
     $$
which immediately delivers
     $$
     \begin{aligned}
             \Psi (\bz^{\ell};T_{\ell})   = \frac{1}{2}  \| F(\Z^{\ell};T_{\ell} ) \|^2   \leq  \frac{1-2\sigma}{2}\| F(\Y^{\ell};T_{\ell}) \|^2 = (1 - 2\sigma)\Psi (\by^{\ell};T_{\ell}).
    \end{aligned}
     $$
This condition indicates that
     $$
        \Psi (\by^{\ell}+\bd^{\ell};T_{\ell}) - \Psi (\by^{\ell};T_{\ell}) \leq  -2\sigma  \Psi (\by^{\ell};T_{\ell}) = -\sigma \langle F(\Y^{\ell};T_{\ell}), \W^{\ell}\bd^{\ell}\rangle,
     $$
which implies that condition \eqref{line-search-step} holds with a unit step length. Finally, the same reasoning to  show \eqref{eq: Locally quadratic convergence} enables us to obtain $
     \| \Y^{{\ell}+1} - \Y^*\| = O( \| \Y^{\ell}-\Y^* \|^2),
     $
     a quadratic convergence rate.
\end{proof}
\begin{remark} Regarding Algorithm \ref{Alg1:SMN}, some comments are given as follows. 
\begin{itemize}[leftmargin=12pt]
\item Despite the locally quadratic convergence rate attained in Theorem \ref{thm_convergen rate}, the requirement  for the sequence to be close to $\Y^*$ implies that Algorithm \ref{Alg1:SMN} is inherently a local method. However, to the best of our knowledge, there are currently no results establishing global convergence for Newton-type methods when directly applied to sparsity-constrained optimization with additional equality or inequality constraints.  The main difficulty in proving global convergence lies in showing the descent property of a suitable merit function. This difficulty arises because support set $T_{\ell}$ changes at every step and is difficult to identify. At present, the best result available is also the local quadratic convergence if the initial point is  sufficiently close to a stationary point \cite{zhao2022}. 

\item At each iteration of Algorithm \ref{Alg1:SMN}, a line search  mechanism (\ref{line-search-step}) is integrated. However, this scheme is restrained on sub-set $\{\bx: \supp(\bx)\subseteq T_{\ell}\}$ rather than entire sparse set $\{\bx: \|\bx\|_0\leq s\}$, and thus it differs from the commonly used globalization mechanisms. The primary purpose of incorporating (\ref{line-search-step}) is to improve the robustness and stability of the algorithm across different applications. A possible direction toward establishing global convergence is to redesign the line search mechanism as
	\begin{equation*}
	\Psi \left( \by^{\ell}+ \bd^{\ell}(\rho^{t_{\ell}});T_{\ell+1}\right)\leq  \Psi \left( \by^{\ell} ;~T_{\ell} \right)+\sigma \rho ^{t_{\ell}} \left\langle F(\Y^{\ell};~T_{\ell}), {\bf W}^{\ell} \bd^{\ell}\right\rangle.
	\end{equation*}
	The only difference between the above scheme and (\ref{line-search-step}) lies in whether the left-hand side involves set $T_{\ell+1}$ or set $T_{\ell}$.  However, the existence of the step size remains considerably challenging if adopting the above strategy. Nevertheless, this is worthy of future research.
\end{itemize}
\end{remark} 
\section{Numerical Experiments}
\label{Numerical experiments}
In this section, we conduct extensive numerical experiments on sparse recovery problems, sparse canonical correlation analysis, and sparse portfolio selection to showcase the performance of {SNSQP}. All experiments are implemented in MATLAB R2021a (available at \url{https://sparseopt.github.io/SQCQP/}), running on a laptop computer of 16GB memory and Inter(R) Core(TM) i7 2.4Ghz CPU. In our numerical experiments, hyperparameters for {SNSQP} are set as follows: $\rho=0.5$, $\sigma=0.5$, $\nu_i^0 = 0, \mu^0_i=0.01$, and $\lambda_i^0=0.01$ for each $i$. The choices of $\tau$ and initial point $\bx^0$ vary depending on the specific problems.

\subsection{Recovery problems}
We first demonstrate the performance of {SNSQP} for solving two recovery problems using synthetic datasets.  The first example does not include the quadratic constraints, described as follows.
\begin{example}    \label{example: recovery1}
Let $f_0\left( \bx \right) =\frac{1}{2}\left\| \D\bx-\bd \right\| ^2$ and $k=0$ (namely, no quadratic constraints), where $\D\in\mathbb{R}^{d\times n}$ and $\bd\in\mathbb{R}^{d}$. Let the linear inequality constraint and $\X$ be $\langle \bx,\bm{1}\rangle  =1$ and $\bx \geq 0$, where ${{\bf 1}:=(1,1,\ldots,1)^\top}$.	
Suppose a ground truth solution $\bx^*\in\mathbb{S}$ is given as follows. Its $s$ nonzero entries are randomly generated from the uniform distribution $U[0,1]$, and then let $\bx^*=\bx^*/\|\bx^*\|_1$. Entries of $\D$ are generated from standard normal distribution $\mathcal{N}(0,1)$,  and then let $\D=\D/\sqrt{d}$. Finally, let
	$\bd=\D\bx^*+\text{nf}\cdot\boldsymbol{\varepsilon}$, where nf is the noise ratio and $\boldsymbol{\varepsilon}$ is the noise. In addition, to control the noise level, we define the signal-to-ratio (SNR) \cite{xiao2022geometric} of $\bd$ as
	$$
	\mathrm{SNR} = 10 \log_{10}  {\|\D \bx^{*}\|^{2}}/{\|\boldsymbol{\varepsilon}\|^{2}}.
	$$
To assess performance, we denote by $\bx$ the solution computed by a given algorithm, and report the computational time in seconds and the reconstruction SNR (RSNR)  defined as
$$
\mathrm{RSNR} = 10 \log_{10} {\|\bx^{*}\|^{2}}/{\|\bx-\bx^*\|^{2}}.
$$
\end{example}
The second example is similar to Example \ref{example: recovery1} but is involved quadratic constraints and more complicated linear inequalities, as outlined below. 
\begin{example}  \label{example: recovery2}
	Let $f_0\left( \bx \right) =\frac{1}{2}\left\| \D\bx-\bd \right\| ^2$ and $\Q_i=\mathbf{P}_i^\top\mathbf{P}_i+0.01\I$ for $i\in[k]$, where $\D\in\mathbb{R}^{d\times n}$,  $\bd\in\mathbb{R}^{d}$,  and $\mathbf{P}_i\in\mathbb{R}^{n\times n}$. Set $\X$ is chosen as $\R^n$, $[-2,2]^n$, or $[0,\infty)^n$.  The $s$ nonzero entries of  ground truth solution $\bx^*\in\mathbb{S}$ are randomly generated from $\X$. Then $\bd=\D\bx^*$ and each entry of $\D$, $\mathbf{P}_i, \bq_i$, and $\A$  is generated from standard normal distribution $\mathcal{N}(0,1)$. To generate $c_i$ and  $\bb$, we randomly select $T_1\subseteq[k]$ with $|T_1|=\lceil k/2\rceil$ and $T_2\subseteq[m]$ with $|T_2|=\lceil m/2 \rceil$, where $\lceil t\rceil$ is the ceiling of $t$. Finally, define
	\begin{align*} 
		c_i =\begin{cases}
			- \frac{1}{2}\langle {\bx^*}, \Q_i{\bx^*} \rangle-\langle \bq_{i}, {\bx^*}\rangle - \zeta_i, &  \text{if}~i\in T_1,\\
			- \frac{1}{2}\langle {\bx^*}, \Q_i{\bx^*} \rangle-\langle \bq_{i}, {\bx^*}\rangle,  & \text{if}~i\notin T_1, \\
		\end{cases} \qquad
		b_i =\begin{cases}
			\langle \ba_{i}, \bx^*\rangle -\xi_i ,  &  \text{if}~i\in T_2,\\
			\langle \ba_{i}, \bx^*\rangle,     &  \text{if}~i\notin T_2,\\
		\end{cases}
	\end{align*}
	where both $\zeta_i$ and $\xi_i$ are uniform random variables from $[0,1]$.	
\end{example}
For Example \ref{example: recovery2}, we report the relative error (${\rm Relerr}:= \| \bx-\bx^* \|/\| \bx^*\|$), the objective function value (Fval), and the computational time in seconds to evaluate the performance of different benchmark algorithms,  where $\bx$ is the solution generated by one algorithm.

\begin{table}[!th]
	\renewcommand{\arraystretch}{1.0}\addtolength{\tabcolsep}{0.5pt}
	\centering 
	\caption{Effects of initial points for Example \ref{example: recovery2} with $\X=\mathbb{R}^n$.}  
	\begin{tabular}{ccccccccc} 
		\toprule 
		\multirow{2}{*}{Initials} & \multicolumn{4}{c}{Relerr} & \multicolumn{4}{c}{Fval}  \\ 
		\cmidrule(lr){2-5} \cmidrule(lr){6-9}  
		&  {min} & {median} & {max} & variance &   {min} & {median} & {max} & variance \\ 
		\midrule  
		$U[0,1]$ 	 &  {2.82e-16} & 4.57e-16 & 5.98e-16 &  {2.99e-33} & 0.00e-00 & 1.78e-15 & 4.44e-15 & 1.25e-30  \\
		
		$\mathcal{N}(0,1)$ 	 &  {3.44e-16} & 5.61e-16 & 7.56e-16 &  {1.07e-32} & 0.00e-00 & 8.88e-16 & 2.66e-15 & 7.26e-31  \\
		
		Weibull &	   {3.25e-16} & 5.27e-16 & 6.47e-16 &  {6.21e-33} & 0.00e-00 & 1.78e-15 & 3.55e-15 & 4.97e-31   \\
		
		$t_{10}$ 	  &  {3.38e-16} & 5.41e-16 & 6.08e-16 &  {4.44e-33} & 0.00e-00 & 1.78e-15 & 3.55e-15 & 1.31e-30   \\
		
		Sparse 	  &  {3.29e-16} & 3.61e-16 & 5.65e-16 &  {3.46e-33} & 0.00e-00 & 1.78e-15 & 3.55e-15 & 1.34e-30   \\ 
		\midrule 
		Overall 	  &  {2.82e-16} & 4.89e-16 & 7.56e-16 &  {9.39e-33} & 0.00e-00 & 1.78e-15 & 4.44e-15 & 1.11e-30   \\ 
		\bottomrule 
	\end{tabular}
	\label{table: recovery-eff-init}  
\end{table}
\subsubsection{Effect of initial points}\label{subsub-com-starting}
We apply {SNSQP} to solve Example \ref{example: recovery2} by setting $\tau=3$,  the maximum number of iterations to be $10^4$, and the tolerance as $\varepsilon=10^{-8}$.  To see the effect of the initial points, we fix $(n,d,k,m,s) = (1000,1000,1,1,50)$ and consider five types of starting points. The first four types of starting points are generated using dense vectors whose entries are generated from four different distributions: the uniform distribution $U[0,1]$, the standard normal distribution $\mathcal{N}(0,1)$, the Weibull distribution $(\lambda = 2, b = 1.5)$, and the $t$-distribution with 10 degrees of freedom. The last type of starting points are generated using sparse vectors constructed as follows: Randomly select $s$ indices to form $\Gamma$ and set $x_i^0=0.1$ if $i\in \Gamma$ and $x_i^0=0$ otherwise.  

For each type of initial points, we run 50 independent trials and report the best, median, and worst results of relative error (Relerr) and objective function values (Fval). We also record these metrics for all 250 trials. As presented in Table \ref{table: recovery-eff-init}, the results show that SNSQP enables exact recovery, as evidenced by the negligible Relerr and Fval values. Furthermore, the variance across all cases is nearly zero, indicating that SNSQP is robust to the choice of initial points for solving Example \ref{example: recovery2}. In general, sparse initialization yields slightly better performance than the other four types, and we therefore adopt it in the subsequent numerical experiments on recovery problems. 
\begin{figure}[!b]
	\centering	
	\begin{subfigure}[b]{0.5\textwidth}
		\centering
		\includegraphics[width=1\textwidth]{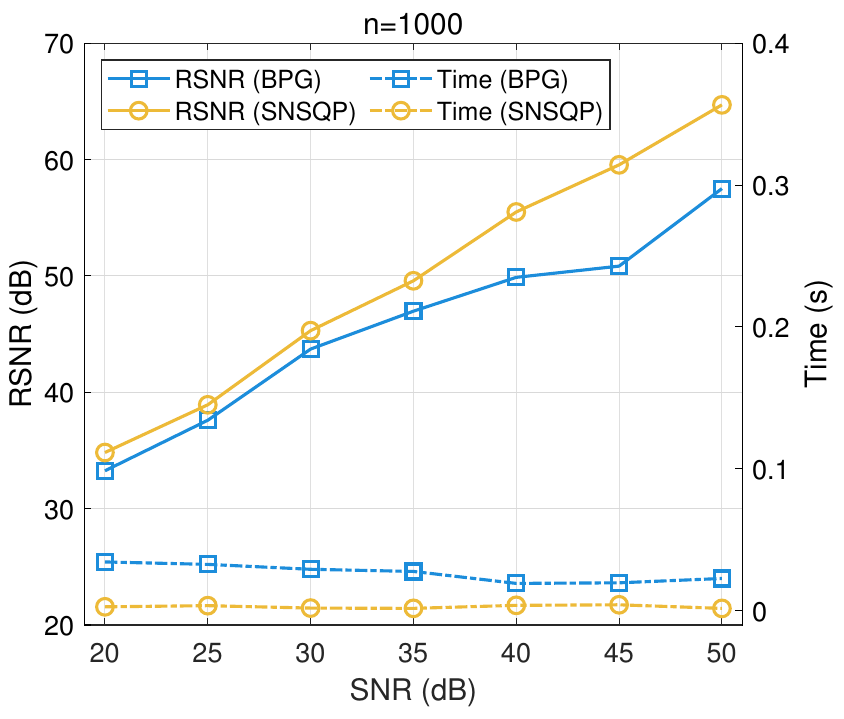}
	\end{subfigure}~~~
	\begin{subfigure}[b]{0.5\textwidth}
		\centering
		\includegraphics[width=1\textwidth]{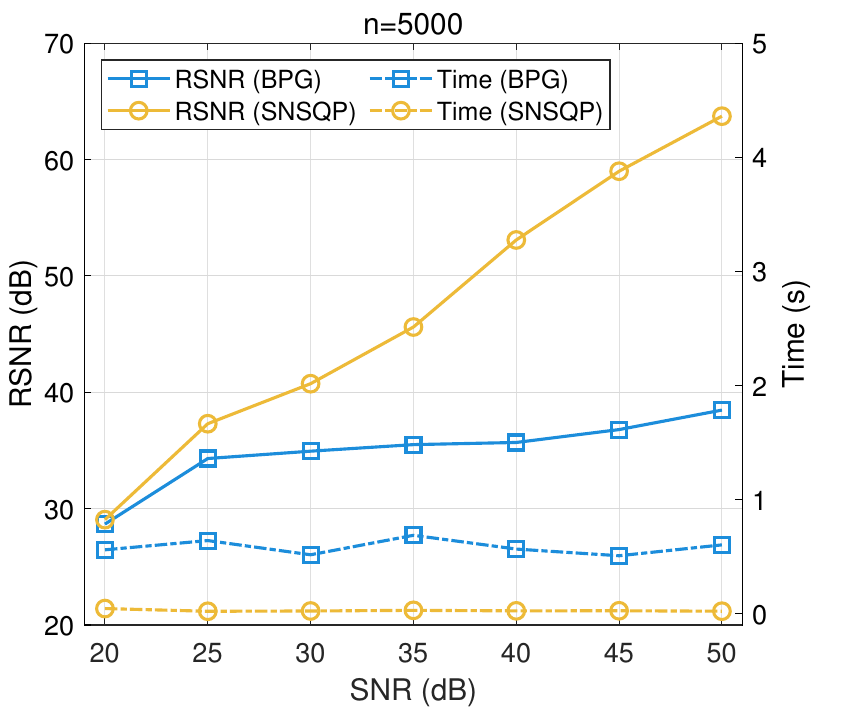}
	\end{subfigure}		
	\caption{Comparison of SNSQP and BPG for Example \ref{example: recovery1}.}
	\label{fig:result-example-5.1}
\end{figure}

\subsubsection{Comparison with a first-order method}
In this part, we compare SNSQP with a proximal gradient algorithm based on the Bregman distance (BPG \cite{pan2025efficient}), which is first-order method, to solve Example \ref{example: recovery1}. Parameters for two algorithms are configured as follows. For SNSQP, we   
set ${\tau=1}$ if ${n\leq1000}$, and ${\tau=0.1}$ otherwise. The maximum number of iterations and the tolerance are set as $10^4$ and $\varepsilon=10^{-8}$. For BPG, we set $L = \max_{i,j} \left| (\D^{\top} \D)_{ij} \right|$, $\lambda=0.005$ if $n=1000$ and $\lambda=0.0005$ if $n=5000$, and $\varepsilon_{1} = \varepsilon_{2} = 10^{-6}$. 

We fix $(d,s)=(0.25n,0.01n)$ but increase SNR over range $\{20,25,\ldots,50\}$, where $n$ is chosen from $\{1000,5000\}$. For each case of $(d,s,n,\text{SNR})$, we run 20 trials and report the median results. As shown in Figure \ref{fig:result-example-5.1}, RSNR generally improves with increasing SNR, and SNSQP achieves higher RSNR in every scenario, indicating higher recovery accuracy. In terms of computational time, SNSQP is faster than BPG in all scenarios.

\subsubsection{Comparison with two solvers for SQCQP relaxations }\label{subsub:com-relax}
We compare {SNSQP} with two QCQP solvers: {CPLEX} and {CVX}, to solve  the SQCQP relaxations by removing the sparsity constraints in SQCQP.  In Example \ref{example: recovery2}, for each $n\in\{1000,2000,\ldots,5000\}$, we set ${d=n+5}$, ${k=m=0.001n}$, ${s=0.01n}$ or $0.05n$ and run 20 independent trials for each $(n,d,k,m,s)$. The median of the 20 trials is presented in Table \ref{table: recovery-QCQP}, where `$-$' indicates cases where the computational time exceeds one hour. The results show that {SNSQP} consistently achieves the lowest Relerr and Fval across all cases, with an accuracy on the order of at least  $10^{-15}$, indicating that it finds the optimal solution. Regarding computation time, {SNSQP} is the fastest. As $n$ increases, the time required by {CPLEX} and {CVX} grows dramatically. For instance, when ${\X = \mathbb{R}^n}$, ${n=5000}$, and ${s=0.01n}$, {CPLEX} exceeds one hour, {CVX} requires more than $1200$ seconds, whereas {SNSQP} consumes $0.127$ seconds, a significant improvement.

\begin{table}[!p] 
	\renewcommand{\arraystretch}{1.085}\addtolength{\tabcolsep}{-2pt}
	\centering 
	\caption{Numerical comparison with QCQP solvers for Example \ref{example: recovery2}.}  
	\begin{tabular}{lcccccccccc} 
		\toprule 
	 &   & \multicolumn{3}{c}{ Relerr} & \multicolumn{3}{c}{Fval} & \multicolumn{3}{c}{Time(s)} \\ 
		\cmidrule(lr){3-5} \cmidrule(lr){6-8} \cmidrule(lr){9-11} 
	$s$	& $n$ & {SNSQP} & {CPLEX} & {CVX} & {SNSQP} & {CPLEX} & {CVX} & {SNSQP} & {CPLEX} & {CVX} \\ 
		\midrule 
		&\multicolumn{10}{c}{$\X=\mathbb{R}^n$} \\ \cline{2-11}
		\multirow{5}{*}{$0.01n$}  
		& 1000 &  {7.97e-16} & 6.57e-02 & 1.97e-02 &  {8.32e-17} & 6.30e-08 & 7.34e-09 &  {0.010} & 1.893 & 4.574 \\
		& 2000 &  {3.07e-16} & 6.86e-02 & 2.61e-02 &  {7.94e-16} & 1.15e-08 & 1.06e-08 &  {0.017} & 36.51 & 52.96 \\
		& 3000 &  {3.75e-16} & 4.15e-02 & 1.57e-02 &  {8.87e-16} & 7.63e-08 & 4.02e-08 &  {0.047} & 442.84 & 200.54 \\
		& 4000 &  {2.25e-16} & 3.22e-02 & 1.08e-02 &  {6.46e-16} & 5.21e-07 & 2.74e-08 &  {0.079} & 1522.3 & 602.20 \\
		& 5000 &  {5.59e-16} & $-$ & 1.41e-02 &  {7.20e-16} & $-$ & 1.99e-08 &  {0.127} & $-$ & 1249.6 \\\midrule

		\multirow{5}{*}{$0.05n$}  
		& 1000 &  {4.31e-16} & 1.76e-02 & 1.46e-02 &  {6.66e-16} & 5.50e-08 & 1.32e-08 &  {0.019} & 2.046 & 4.695 \\
		& 2000 &  {4.45e-16} & 4.48e-02 & 2.44e-02 &  {8.88e-16} & 1.25e-07 & 4.25e-08 &  {0.036} & 35.502& 51.72 \\
		& 3000 &  {5.12e-16} & 1.12e-02 & 2.36e-02 &  {8.02e-16} & 5.21e-07 & 2.88e-08 &  {0.049} & 469.40& 215.55 \\
		& 4000 &  {5.00e-16} & 5.33e-02 & 1.52e-02 &  {6.55e-16} & 7.23e-07 & 2.46e-08 &  {0.198} & 1421.1 & 705.27 \\
		& 5000 &  {6.21e-16} & $-$ & 1.52e-02 &  {7.79e-16} & $-$ & 2.33e-08 &  {0.402} & $-$ & 1201.1 \\\midrule 
		
		&\multicolumn{10}{c}{$\X=[-2,2]^n$} \\ \cline{2-11}    
		\multirow{5}{*}{$0.01n$}   
		& 1000 &  {1.98e-16} & 8.67e-02 & 2.34e-04 &  {8.21e-16} & 6.81e-06 & 3.85e-12 &  {0.012} & 2.490 & 5.495 \\
		& 2000 &  {2.22e-16} & 1.69e-01 & 4.21e-04 &  {7.15e-15} & 5.25e-05 & 1.71e-09 &  {0.020} & 33.51 & 64.48 \\
		& 3000 &  {5.16e-16} & 1.36e-01 & 7.24e-04 &  {8.46e-16} & 4.38e-06 & 8.66e-12 &  {0.031} & 279.17 & 191.43 \\
		& 4000 &  {3.27e-16} & 1.02e-01 & 1.06e-03 &  {8.57e-16} & 7.99e-06 & 1.03e-11 &  {0.079} & 1554.5 & 614.24 \\
		& 5000 &  {4.47e-16} & $-$& 1.11e-03 &  {6.35e-16} & $-$ & 7.54e-12 &  {0.131} & $-$ & 1251.2 \\\midrule  
		
		\multirow{5}{*}{$0.05n$}   
		& 1000 &  {4.29e-16} & 8.03e-02 & 3.04e-04 &  {2.22e-16} & 3.87e-06 & 5.16e-12 &  {0.017} & 3.048 & 6.199 \\
		& 2000 &  {3.56e-15} & 1.43e-01 & 6.57e-04 &  {6.33e-15} & 3.14e-05 & 7.69e-10 &  {0.051} & 34.45 & 66.87 \\
		& 3000 &  {5.33e-16} & 1.13e-01 & 3.90e-03 &  {7.02e-16} & 5.79e-06 & 5.03e-12 &  {0.085} & 302.34 & 225.15 \\
		& 4000 &  {6.25e-16} & 1.21e-01 & 1.53e-03 &  {6.73e-16} & 1.63e-05 & 4.12e-12 &  {0.214} & 1513.4 & 722.78 \\
		& 5000 &  {4.11e-16} & $-$ & 2.01e-03 &  {6.25e-16} & $-$ & 3.86e-12 &  {0.464} & $-$ & 1267.3 \\\midrule  
		
		&\multicolumn{10}{c}{$\X=[0,\infty)^n$} \\ \cline{2-11}
		\multirow{5}{*}{$0.01n$}  
		& 1000 &  {5.27e-16} & 3.80e-02 & 2.00e-04 &  {1.11e-16} & 2.83e-06 & 4.31e-09 &  {0.014} & 2.347 & 8.527 \\
		& 2000 &  {1.31e-16} & 2.23e-02 & 1.69e-04 &  {3.65e-16} & 1.03e-06 & 4.48e-09 &  {0.023} & 34.68 & 54.51 \\
		& 3000 &  {3.34e-16} & 2.11e-02 & 2.06e-04 &  {4.56e-16} & 6.81e-06 & 5.76e-09 &  {0.051} & 217.73 & 258.91 \\
		& 4000 &  {3.92e-16} & 1.12e-02 & 1.86e-04 &  {5.69e-16} & 7.63e-06 & 1.36e-09 &  {0.091} & 1460.0 & 751.15 \\
		& 5000 &  {1.06e-15} & $-$ & 1.79e-03 &  {1.21e-16} & $-$ & 5.90e-09 &  {0.301} & $-$ & 1631.5 \\
		\midrule  
		\multirow{5}{*}{$0.05n$}  
		& 1000 &  {4.48e-16} & 4.40e-02 & 1.81e-04 &  {4.44e-16} & 7.40e-06 & 8.18e-09 &  {0.023} & 3.464 & 8.414 \\
		& 2000 &  {5.53e-16} & 6.70e-02 & 1.49e-02 &  {1.35e-16} & 6.81e-06 & 2.56e-09 &  {0.063} & 32.85 & 55.28 \\
		& 3000 &  {5.79e-16} & 1.19e-02 & 1.59e-02 &  {5.32e-16} & 2.36e-06 & 4.55e-09 &  {0.248} & 239.25 & 275.00 \\
		& 4000 &  {6.75e-16} & 1.07e-02 & 1.38e-04 &  {1.25e-15} & 7.69e-06 & 1.86e-09 &  {1.201} & 1500.2 & 775.68 \\
		& 5000 &  {7.82e-16} & $-$ & 1.82e-04 &  {1.86e-16} & $-$ & 1.33e-09 &  {1.516} & $-$ & 1816.8 \\
		\bottomrule 
	\end{tabular}
	\label{table: recovery-QCQP}  
\end{table}

\subsubsection{Comparison with two solvers for SQCQP reformulations}\label{subsub:com-ref}
We compare {SNSQP} with GUROBI and {CVX} to solve the SQCQP reformulations. The former is adopted to address the following MIP,
\begin{equation}
	\label{eq: MIP of SCQP}
	\underset{\bx,\bw}{\min}~f_0(\bx),~~
	{\rm s.t.}~~   \bx\in \F\cap \X,~ \langle {\bf 1}, \bw\rangle  \leq  s,~  |x_i| \leq  M w_i,~w_i\in  \{ 0,1 \}, ~i\in[n],
\end{equation}
where ${M > 0}$ is a large enough constant, which is set as ${M=10}$ in this example.  While CVX is used to solve the $\ell_1$-norm reformulation of SNSQP. Specifically, we replace the sparsity constraint by $\|\bx\|_1\leq\|\bx^*\|_1$, yielding a convex relaxation of SNSQP, which is then solved by CVX. In Example \ref{example: recovery2}, we focus on $\X=\mathbb{R}^n$ and choose $n\in\{ 100,200,\ldots, 1000\}$. Other  dimensions $(d,k,m,s)$ are set similarly to those in Section \ref{subsub:com-relax}. The median results of 20 independent trials are presented in Table  \ref{table: recovery-MIP}. One can observe that {SNSQP} achieves the lowest Relerr, Fval, and the shortest computational time once again.

\begin{table}[!t] 
	\renewcommand{\arraystretch}{1.0}\addtolength{\tabcolsep}{-3pt}
	\centering 
	\caption{Numerical comparison of SNSQP with GUROBI and CVX for Example \ref{example: recovery2} with $\X=\mathbb{R}^n$.} 
	\begin{tabular}{llccccccccc}  
		\toprule  
			&  &   \multicolumn{3}{c}{Relerr} & \multicolumn{3}{c}{Fval} & \multicolumn{3}{c}{Time(s)} \\
		\cmidrule(lr){3-5} \cmidrule(lr){6-8} \cmidrule(lr){9-11}  
	$s$	& $n$  & {SNSQP} & {GUROBI} & {CVX} & {SNSQP} & {GUROBI} & {CVX} & {SNSQP} & {GUROBI} & {CVX} \\
		\midrule  
		\multirow{5}{*}{$0.01n$}  
		& 100  &  {8.63e-13} & 1.30e-03 & 6.94e-05 &  {2.08e-18} & 8.18e-08 & 8.45e-10 &  {0.002} & 0.617 & 0.529 \\
		& 300  &  {1.76e-14} & 2.73e-04 & 5.55e-05 &  {1.38e-17} & 2.93e-08 & 1.57e-09 &  {0.003} & 3.217 & 1.124 \\
		& 500  &  {6.64e-15} & 3.37e-04 & 4.78e-05 &  {5.55e-17} & 1.92e-08 & 1.64e-09 &  {0.006} & 8.490 & 2.462 \\
		& 700  &  {3.36e-16} & 6.45e-04 & 4.57e-05 &  {5.59e-18} & 1.89e-08 & 2.18e-09 &  {0.009} & 22.66 & 4.721 \\
		& 1000 &  {2.67e-16} & 9.63e-04 & 4.12e-05 &  {1.12e-17} & 2.17e-07 & 2.92e-09 &  {0.010} & 310.6 & 10.65 \\
		\midrule  
		\multirow{5}{*}{$0.05n$}  
		& 100  &  {2.78e-14} & 3.00e-04 & 5.30e-05 &  {5.55e-18} & 2.84e-08 & 2.10e-09 &  {0.003} & 0.886 & 0.470 \\
		& 300  &  {6.25e-14} & 2.86e-04 & 4.60e-05 &  {4.44e-17} & 7.91e-08 & 3.97e-09 &  {0.005} & 24.12 & 1.165 \\
		& 500  &  {2.84e-14} & 1.50e-04 & 4.13e-05 &  {3.33e-17} & 1.05e-08 & 6.32e-09 &  {0.006} & 699.8 & 2.456 \\
		& 700  &  {5.45e-16} & 1.30e-03 & 4.42e-05 &  {3.17e-17} & 1.87e-07 & 7.44e-09 &  {0.012} & 334.4 & 4.631 \\
		& 1000 &  {5.60e-16} & $-$      & 3.79e-05 &  {4.10e-16} & $-$      & 9.59e-09 &  {0.012} & $-$   & 10.53 \\
		\bottomrule  
	\end{tabular}
	\label{table: recovery-MIP}  
\end{table}

\subsubsection{Effect of higher dimensions}
We evaluate the performance of SNSQP in solving Example \ref{example: recovery2} in higher-dimensional settings. A comparison with other solvers is omitted, as Tables \ref{table: recovery-QCQP} and \ref{table: recovery-MIP} indicate that they exhibit significantly higher computational costs for such cases. Table \ref{table: synthetic higher} presents the median results over 10 independent runs, where ${d = n + 5}$ and ${k = m = 5}$.  The results demonstrate that the Relerr and Fval obtained by SNSQP remain stable, consistently ranging in magnitude from $10^{-13}$ to $10^{-15}$, while maintaining efficient computational performance. Specifically, for ${n = 30,000}$, the computational time is 3.251 seconds and 33.59 seconds for $s = \lceil 0.01n \rceil$ and $s = \lceil 0.05n \rceil$.

\begin{table}[!t]
\renewcommand{\arraystretch}{1}\addtolength{\tabcolsep}{9pt}
\centering
\caption{Results by SNSQP for Example \ref{example: recovery2} with $ \X=[-2,2]^n$ in higher dimensions.}
\label{table: synthetic higher}
\begin{tabular}{ccccccr}
\toprule
 & \multicolumn{3}{c}{$s=0.01n$} & \multicolumn{3}{c}{$s=0.05n $} \\
\cmidrule(lr){2-4} \cmidrule(lr){5-7}
$n$  & Relerr & Fval & Time(s) & Relerr & Fval & Time(s) \\
\midrule
10000 & 5.47e-15 & 4.44e-15 & 0.284 & 1.89e-14 & 7.10e-15 & 3.392 \\
15000 & 1.03e-14 & 4.46e-15 & 0.895 & 1.42e-14 & 1.52e-15 & 4.581 \\
20000 & 2.25e-14 & 7.10e-15 & 1.153 & 1.57e-15 & 1.05e-15 & 10.54 \\
25000 & 1.82e-13 & 2.84e-14 & 2.741 & 2.13e-14 & 5.46e-15 & 15.82 \\
30000 & 8.03e-14 & 7.10e-15 & 3.251 & 5.68e-14 & 1.99e-15 & 33.59 \\
\bottomrule
\end{tabular}
\end{table}

\subsection{Sparse canonical correlation analysis}
The sparse canonical correlation analysis (SCCA) problem takes the form of:
\begin{equation}
\label{eq: SCCA1}
        \begin{aligned}
		\max ~& \langle \bw^x, \boldsymbol{\Sigma}^{xy} \bw^y\rangle \\
		\text{s.t.} ~&  \langle \bw^x, \boldsymbol{\Sigma}^{xx} \bw^x\rangle \leq  1,~\|\bw^x\|_0 \leq  s_x,\\ 
                            & \langle \bw^y, \boldsymbol{\Sigma}^{yy} \bw^y\rangle \leq  1,~\|\bw^y\|_0 \leq  s_y, \\           
	\end{aligned}\tag{SCCA}
\end{equation}
where $\boldsymbol{\Sigma}^{xy}\neq0$ is the cross-covariance matrix between two given data matrices ${\bf X}$ and $\Y$, namely $\boldsymbol{\Sigma}^{xy}={\bf X} \Y^\top$, and $\boldsymbol{\Sigma}_{xx}$ and $\boldsymbol{\Sigma}^{yy}$ represent the covariance matrices for ${\bf X}$ and $\Y$, namely $\boldsymbol{\Sigma}^{xx}={\bf X} {\bf X}^\top$ and $\boldsymbol{\Sigma}^{yy}={\bf Y} \Y^\top$. By letting $s = s_x + s_y$ and
$$
\Q_0=\left[ \begin{matrix}
	0&		\boldsymbol{\Sigma}^{xy}\\
	{\boldsymbol{\Sigma}^{xy}}^\top&		0\\
\end{matrix} \right] ,~~ \Q_1=\left[ \begin{matrix}
	\boldsymbol{\Sigma}^{xx}&		0\\
	0&		\boldsymbol{\Sigma}^{yy}\\
\end{matrix} \right],~~\bx=\left[ \begin{array}{c}
	\bw^x\\
	\bw^y\\
\end{array} \right],
$$ 
 problem \eqref{eq: SCCA2} can be relaxed in the form,
\begin{equation}
\label{eq: SCCA2}
        \begin{aligned}
	 \max~  \langle \bx,  \Q_0 \bx \rangle, ~~\text { s.t. }~   \langle \bx, \Q_1 \bx\rangle  \leq  2, ~\|\bx\|_0 \leq  s.        
	\end{aligned}
\end{equation} 
This is a relaxation of problem \eqref{eq: SCCA1} because its feasible region includes the feasible region of \eqref{eq: SCCA1}. However, the following result shows that the optimal solutions to them have a close relationship. Therefore, we employ SNSQP to solve \eqref{eq: SCCA2},  a special case of \eqref{eq: SCQP}.

\begin{proposition} Any optimal solution $(\bx^x;\bx^y)$ of (\ref{eq: SCCA2}) satisfies $\|\bx^x\|_0 > 0$ and $\|\bx^y\|_0 > 0$. 
\end{proposition}

\begin{proof} Since $\boldsymbol{\Sigma}^{xy}\neq0$, there is $\Sigma^{xy}_{i_0j_0}\neq0$. Consider a point $\bx:=(\bu;\bv)$ such that
\begin{align*}
u_i=\begin{cases}
 \frac{1}{\sqrt{\Sigma^{xx}_{i_0i_0}}} & {\rm if}~  i=i_0,~\Sigma^{xx}_{i_0i_0} \neq 0,\\
1 & {\rm if}~  i=i_0,~\Sigma^{xx}_{i_0i_0} = 0,\\
0 & {\rm if}~  i\neq i_0,
\end{cases}\qquad
v_j=\begin{cases}
 \frac{{\rm sign} \left(\Sigma^{xy}_{i_0j_0}\right)}{\sqrt{\Sigma^{yy}_{j_0j_0}}} & {\rm if}~  j=j_0,~\Sigma^{yy}_{j_0j_0} \neq 0,\\
1 & {\rm if}~  j=j_0,~\Sigma^{yy}_{j_0j_0} = 0,\\
0 & {\rm if}~  j\neq j_0,
\end{cases}
\end{align*}
where ${\rm sign} (t)$ is the sign of $t$. Therefore, $\bx=(\bu;\bv)$ is feasible to problem (\ref{eq: SCCA2}), and meanwhile it satisfies that $\langle \bx,  \Q_0 \bx \rangle>0$. Note that any point $\bx'=(\bu';\bv')$ satisfying $\|\bu'\|_0 = 0$ or $\|\bv'\|_0 = 0$ leads to $\langle \bx',  \Q_0 \bx' \rangle=0$. Hence, $\bx'$ is not the optimal solution, showing the desired result.
\end{proof}
\begin{remark} In Theorem \ref{thm_convergen rate}, the local convergence rate relies on two regularity conditions: the restricted LICQ and the condition (\ref{eq: strong SOC}) hold at the stationary point. However, ondition (\ref{eq: strong SOC}) may not hold for problem (\ref{eq: SCCA2}). To see this, let $\bx^*=(\bx^*_x;\bx^*_y)$ be a P-stationary point and $\mu^*$ be the corresponding Lagrangian multiplier. If (\ref{eq: strong SOC}) holds at $\bx^*$, then  
	 \begin{equation}
	 	\label{eq: SCCA-remark}
	 	\langle \bd, \H^*  \bd \rangle > 0, \forall 0 \ne \bd\in \mathbb{Q}\left( \Y^*;T \right)=\{\bd=(\bd_x;\bd_y):(\boldsymbol{\Sigma}^{xx}\bd_x;\boldsymbol{\Sigma}^{yy}\bd_y)=0\}.
	 \end{equation}
 It is readily observed that any vector $\bd\in\mathbb{Q}\left( \Y^*;T \right)$ satisfies ${\bf X}^\top\bd_x=0, {\bf Y}^\top\bd_y=0$, leading  to
	$$
	\begin{aligned}
		\langle \bd, \H^*  \bd \rangle = & \left\langle \begin{bmatrix} \bd_x \\ \bd_y \end{bmatrix}, 
		\begin{bmatrix} \mu^* \boldsymbol{\Sigma}^{xx} & -\boldsymbol{\Sigma}^{xy} \\ -{\boldsymbol{\Sigma}^{xy}}^\top & \mu^*\boldsymbol{\Sigma}^{yy} \end{bmatrix} 
		\begin{bmatrix} \bd_x \\ \bd_y \end{bmatrix} \right\rangle \\
		& = \mu^*\langle\bd_x, \boldsymbol{\Sigma}^{xx} \bd_x \rangle + \mu^*\langle\bd_y, \boldsymbol{\Sigma}^{yy} \bd_y \rangle - 2\langle\bd_x, \boldsymbol{\Sigma}^{xy} \bd_y \rangle \\
		& = -2\langle\bd_x, \boldsymbol{\Sigma}^{xy} \bd_y \rangle    = -2 \langle {\bf X}^\top\bd_x, {\bf Y}^\top\bd_y \rangle    = 0.
	\end{aligned}
	$$
This contradicts (\ref{eq: SCCA-remark}). However, considering that condition (\ref{eq: strong SOC}) is only a sufficient condition for convergence, we still employ SNSQP to solve this problem.
\end{remark}

\subsubsection{Benchmark methods} Many algorithms have been proposed to address SCCA problems. The majority of these methods are based on relaxations of the $\ell_0$ norm, while a subset employs greedy strategies. We select three representative relaxation-based approaches: {SGEM} \cite{sriperumbudur2011majorization}, {SCCA} \cite{chu2013sparse}, and {SCCAPD} \cite{hardoon2011sparse},  in addition to two greedy algorithms: {SpanCCA} \cite{asteris2016simple} and {SWCCA} \cite{min2018sparse}.  It is important to note that the greedy methods are applicable only in scenarios where the covariance matrices are identity matrices. Consequently, in the subsequent numerical experiments, we specifically set $\boldsymbol{\Sigma}^{xx}={\bf I}$ and $\boldsymbol{\Sigma}^{yy}={\bf I}$ when evaluating {SpanCCA} and {SWCCA}. In addition, we also use GUROBI to solve problem \eqref{eq: SCCA2}, that is, by transforming it into a MIP using the same approach described in Section \ref{subsub:com-ref}. 

The hyperparameters for all algorithms are configured as follows.  
For {SNSQP}, initial point $\bx^0$ is generated using MATLAB built-in function {\tt canoncorr}, while $\tau$ is selected via a grid search over $\{0.001, 0.002, \ldots, 0.01\}$. 
For {SGEM},  the same initial point as in {SNSQP} is used, with parameters set as $\epsilon=10^{-6}$ and $\tau  = \max\{0, -\lambda_{\min}(\Q_0)\}$, where $\lambda_{\min}(\Q_0)$ represents the smallest eigenvalue of $\Q_0$.  
For {SCCA},  maximum number $\ell_{\max}$ of iterations is fixed at 5000, and  $\delta$ is selected from a predefined grid $\{0.1, 0.2, \dots, 1\}$.  
For {SCCAPD}, we use its default settings. 
For {SpanCCA}, we set $\ell_{\max}=5000$ and $r=5$.   
For {SWCCA},  the initial point is the same as in {SNSQP}, and the sparsity level of $\bw$ is set to be $0.6N$. It terminates when either $\ell>5000$ or  $\|\bw^{k+1}-\bw^k \|^2 < 10^{-6}$ is satisfied. For GUROBI, we set the maximum running time to one hour. To evaluate the performance, we report the correlation, sparsity levels $\rho_x$ and $\rho_y$ of $\bw^x$ and $\bw^y$,  violation $\text{VOC}_x$ and $\text{VOC}_y$ of the unit variance constraint, and the computational time, where
\begin{align*}
&\text{Correlation}:=\frac{\langle\bw^x, \Sigma^{xy} \bw^y\rangle}{\sqrt{\langle\bw^x, \Sigma^{xx} \bw^x\rangle\langle\bw^y, \Sigma^{yy} \bw^y\rangle} },\\
&\rho_x:=\frac{n_x-\|\bw^x\|_0}{n_x},\qquad\text{VOC}_x:=\| \langle \bw^x, \boldsymbol{\Sigma}^{xx} \bw^x\rangle- 1\|,\\
&\rho_y:=\frac{n_y-\|\bw^y\|_0}{n_y},\qquad \text{VOC}_y:=\| \langle \bw^y, \boldsymbol{\Sigma}^{yy} \bw^y\rangle- 1\|.
\end{align*}
\subsubsection{Testing examples} 
\begin{example}[\textbf{Synthetic data \cite{chu2013sparse}}]
    \label{example: SCCA1}
    Let $\mathbf{X}\in\mathbb{R}^{n_x \times N}$ and $\mathbf{Y}\in\mathbb{R}^{n_y \times N}$ be generated by 
    \begin{equation}
        \mathbf{X}=\left(({\bf 1};-{\bf 1};0)+\boldsymbol{\epsilon}\right) {\bf{u}}^{\top}, \quad \mathbf{Y}=\left((0;{\bf 1};-{\bf 1})+\boldsymbol{\varepsilon}\right) {\bf{u}}^{\top},
    \end{equation}
where ${\bf 1}\in \mathbb{R}^{n_x/8}$, $\boldsymbol{\epsilon} \in \mathbb{R}^{n_x}$ and $\boldsymbol{\varepsilon} \in \mathbb{R}^{n_y}$ are two noise vectors with  $\epsilon_i\sim \mathcal{N}(0,0.1^2)$ and $\varepsilon_i\sim \mathcal{N}(0,0.1^2)$, and ${\bu  \in \mathbb{R}^{N}}$ is a random vector with $u_i\sim \mathcal{N}(0,1)$. 

\end{example}

\begin{example}[\textbf{Real data}]
    \label{example: SCCA2}  Four real datasets are selected to generate ${\bf{X}}$ and ${\bf{Y}}$. They are {{\tt SRBCT}}, {{\tt lymphoma}}, {{\tt breast cancer}}, and {{\tt glioma}}\footnote{Available at \href{www.ncbi.nlm.nih.gov/geo/query/acc.cgi?acc=GSE2223}{www.ncbi.nlm.nih.gov/geo/query/acc.cgi?acc=GSE2223}.} \cite{bredel2005functional, bredel2005high} with details summarized in in Table \ref{table: SCCA real dataset}.  The sample-wise normalization is conducted on them so that each sample has mean zero and variance one. In addition, {{\tt SRBCT}} and {{\tt lymphoma}} only contain one set of samples, so we manually divide the variables into two equal parts to obtain ${\bf{X}}$ and ${\bf{Y}}$.
\end{example}

\begin{table}[!th]
\renewcommand{\arraystretch}{1.0}\addtolength{\tabcolsep}{5.5pt}
\caption{Descriptions of four real datasets.}
\label{table: SCCA real dataset}
\centering
\begin{tabular}{llcccc}
    \hline
    Datasets & Source & N & $n$ & $n_x$ & $n_y$ \\ \hline
    {\tt SRBCT}      & R package: plsgenomics \cite{plsgenomics} & 82 & 2308 & 1154 & 1154 \\  
    {\tt lymphoma}   & R package: KODAMA \cite{KODAMA} & 61 & 4026 & 2013 & 2013 \\ 
    {\tt breast cancer} & R package: PMA \cite{PMA} & 89 & 21821 & 2149 & 19672 \\ 
    {\tt glioma}     & NCBI: Gene Expression Omnibus & 55 & 62127 & 22962 & 39165 \\ \hline
\end{tabular}
\end{table}

\subsubsection{Numerical comparison} From the construction of $\mathbf{X}$ and $\mathbf{Y}$ in Example \ref{example: SCCA1}, it is evident that the first $n_x/4$ variables of $\mathbf{X}$ exhibit correlation with the last $n_x/4$ variables of $\mathbf{Y}$. An effective sparse CCA algorithm should be capable of computing weight vectors $\bw_x$ and $\bw_y$ that accurately identify these correlated variables. Specifically, the nonzero elements of $\bw_x$ should be restricted to the first $n_x/4$ components, while those of $\bw_y$ should be confined to the last $n_x/4$ components.   As illustrated in Figure \ref{fig: SCCA synthetic data}, the {SNSQP} algorithm successfully achieves this objective.

 \begin{figure}[!t]
    \centering
    \begin{subfigure}[b]{0.49\textwidth}
        \centering
        \includegraphics[width=\textwidth]{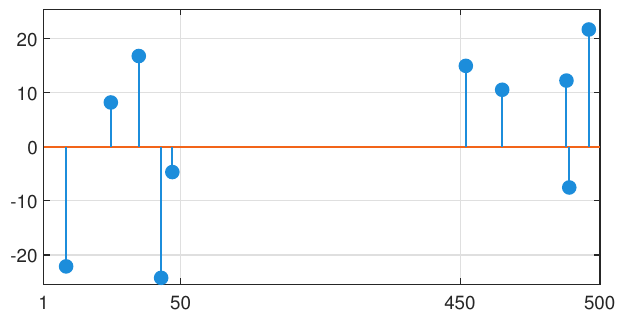}  
        \caption{$n_x=200, n_y=300$}
    \end{subfigure}
    \hfill
    \begin{subfigure}[b]{0.49\textwidth}
        \centering
        \includegraphics[width=\textwidth]{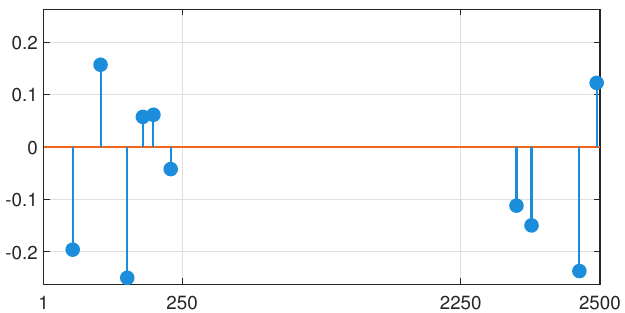}  
        \caption{$n_x=1000, n_y=1500$}
    \end{subfigure}
    
    \caption{Solutions obtained by {SNSQP} for Example \ref{example: SCCA1} with $s=10$.}
    \label{fig: SCCA synthetic data}
\end{figure}

\begin{table}[!t]
\renewcommand{\arraystretch}{.95}\addtolength{\tabcolsep}{3.0pt}
	\centering
	\caption{Comparison with relaxation methods for Example \ref{example: SCCA1}. } 
	\begin{tabular}{llccccrr}
\toprule
     Algs.       &   Para.     & Correlation  & $\rho_x$ & $\rho_y$ & $\text{VOC}_x$ & $\text{VOC}_y$ &Time(s) \\ \hline 
          \multicolumn{8}{c}{$n_x=200$, $n_y=300$} \\ \midrule
 {SNSQP} & $s=5$          & 1.0000      & 99.0\%        & 99.0\%        & 6.43e-15    & 5.99e-15    & 0.006   \\ 
                           & $s=10$         & 1.0000      & 99.0\%        & 97.3\%     & 1.34e-10    & 1.78e-10    & 0.009   \\  
 {SGEM}  & $\rho=0.01$    & 0.9980      & 99.5\%      & 99.7\%      & 1.94e-02    & 1.94e-02    & 18.51  \\ 
                          & $\rho=0.005$   & 1.0000      & 76.0\%      & 84.3\%      & 1.69e-04    & 1.69e-04    & 6.388   \\  
 {SCCA}  & $\mu=5$        & 1.0000      & 99.0\%        & 99.0\%        & 2.35e-05    & 8.64e-04    & 0.120   \\ 
                          & $\mu=10$       & 1.0000      & 99.5\%      & 99.3\%      & 4.65e-05    & 9.17e-04    & 0.207   \\  
 SCCAPD         &               & 1.0000      & 75.0\%      & 61.6\%      & 6.43e-16    & 5.99e-16    & 0.024   \\ 
{GUROBI}         &               & 1.0000      & 99.5\%      & 97.0\%      & 2.90e-03    & 2.90e-03    & 3600   \\ \midrule
\multicolumn{8}{c}{$n_x=1000$, $n_y=1500$} \\ \midrule
 {SNSQP} & $s=5$          & 1.0000      & 99.8\%      & 99.8\%      & 1.61e-15    & 1.62e-15    & 0.005   \\ 
                           & $s=10$         & 1.0000      & 99.7\%      & 99.5\%      & 1.64e-10    & 1.64e-10    & 0.008   \\  
 {SGEM}  & $\rho=0.01$    & 0.9990      & 96.5\%      & 97.5\%      & 2.86e-05    & 2.86e-05    & 166.5 \\ 
                          & $\rho=0.005$   & 1.0000      & 95.6\%      & 96.8\%      & 6.71e-05    & 6.71e-05    & 138.0 \\  
 {SCCA}  & $\mu=5$        & 1.0000      & 99.5\%      & 99.2\%      & 1.01e-10    & 9.90e-04    & 0.419   \\ 
                          & $\mu=10$       & 1.0000      & 99.8\%      & 99.3\%      & 5.09e-13    & 9.90e-04    & 0.700   \\  
SCCAPD         &               & 1.0000      & 92.9\%      & 64.8\%      & 2.22e-16    & 5.99e-16    & 0.035   \\
{GUROBI}         &               & 1.0000      & 99.8\%      & 99.6\%      & 7.92e-03    & 7.92e-03    & 3600   \\
\bottomrule
\end{tabular}
	\label{table: SCCA_synthetic}
\end{table}

 A comprehensive numerical comparison of {SNSQP} with three relaxation algorithms for solving Examples \ref{example: SCCA1} and   \ref{example: SCCA2} is presented in Table \ref{table: SCCA_synthetic} and Table \ref{table: SCCA_real1}. One can observe that {SNSQP} outperforms the others in terms of overall performance, achieving the highest correlation and sparsity levels while consuming the shortest time. Due to the significant computational cost of GUROBI as the problem dimension increases, we exclude it for solving Example \ref{example: SCCA2}, a higher-dimensional case.
 \begin{table}[!t]
\renewcommand{\arraystretch}{.925}\addtolength{\tabcolsep}{2.5pt}
	\centering
	\caption{Comparison with relaxation methods for Example \ref{example: SCCA2}.} \label{table: SCCA_real1}
	\begin{tabular}{llcccccc}
\toprule
 Algs.       &   Para.     & Correlation  & $\rho_x$ & $\rho_y$ & $\text{VOC}_x$ & $\text{VOC}_y$ & Time(s) \\ \hline 
          \multicolumn{8}{c}{{\tt SRBCT}} \\ \midrule
 {SNSQP} & $s=40$          & 0.9968      & 97.83\%      & 98.70\%      & 1.26e-14    & 1.26e-14    & 0.014   \\ 
                          & $s=80$          & 1.0000      & 96.62\%      & 96.45\%      & 4.68e-11    & 4.68e-11    & 0.035   \\  
 {SGEM}  & $\rho=0.01$     & 0.6317      & 1.822\%       & 1.215\%       & 3.10e-03    & 3.10e-03    & 1587\\ 
                          & $\rho=0.005$    & 0.5879      & 0.435\%       & 0.782\%       & 2.61e-03    & 2.61e-03    & 1177\\  
 {SCCA}  & $\mu=5$         & 1.0000      & 86.40\%      & 86.74\%      & 8.11e-04    & 3.51e-04    & 0.842   \\ 
                          & $\mu=10$        & 0.9887      & 91.16\%      & 90.81\%      & 7.23e-02    & 6.45e-03    & 0.826   \\ 
 SCCAPD         &                & 0.9316      & 98.09\%      & 66.20\%      & 3.32e-16          & 2.22e-16    & 0.063   \\ \midrule
\multicolumn{8}{c}{{\tt lymphoma}} \\ \midrule
 {SNSQP} & $s=20$          & 0.9904      & 99.35\%      & 99.65\%      & 5.57e-11    & 5.57e-11    & 0.014   \\ 
                          & $s=50$          & 0.9999      & 98.56\%      & 98.96\%      & 6.94e-11    & 6.94e-11    & 0.034   \\  
 {SCCA}  & $\mu=5$         & 0.9014      & 98.61\%      & 98.56\%      & 1.82e-02    & 2.80e-02    & 2.759   \\ 
                          & $\mu=10$        & 0.4016      & 99.65\%      & 99.60\%      & 1.32e-02    & 2.80e-02    & 3.865   \\  
SCCAPD         &                & 0.9210      & 99.06\%      & 56.38\%      & 4.44e-16    & 4.44e-16    & 0.162   \\\midrule  
\multicolumn{8}{c}{{\tt breast cancer}} \\ \midrule
{SNSQP} & $s=40$          & 0.9941      & 99.58\%      & 99.84\%      & 9.28e-09    & 9.28e-09    & 0.110   \\ 
                        & $s=80$          & 1.0000      & 99.44\%      & 99.76\%      & 3.87e-09    & 3.87e-09    & 0.160  \\  
 {SCCA}  & $\mu=0.5$       & 0.9835      & 95.86\%      & 99.72\%      & 1.34e-03    & 1.55e-02    & 36.27  \\ 
                          & $\mu=1$         & 0.7723      & 96.32\%      & 99.94\%      & 8.04e-03    & 3.95e-01    & 36.15  \\  
SCCAPD         &                & 0.9193      & 99.67\%      & 99.74\%      & 3.33e-16    & 5.55e-16    & 0.164   \\ \midrule
\multicolumn{8}{c}{{\tt glioma}} \\ \midrule
 {SNSQP} & $s=50$          & 1.0000      & 99.92\%      & 99.92\%      & 1.07e-11    & 1.07e-11    & 0.476   \\ 
                          & $s=100$         & 1.0000      & 99.80\%      & 99.86\%      & 3.31e-08    & 3.31e-08    & 0.341   \\  
 {SCCA}  & $\mu=0.2$       & 0.9955      & 99.31\%      & 99.56\%      & 1.48e-02    & 9.83e-03    & 355.4 \\ 
                          & $\mu=0.5$       & 0.8825      & 99.53\%      & 99.78\%      & 9.29e-02    & 2.64e-01    & 353.3 \\  
SCCAPD         &                & 0.9223      & 99.98\%      & 99.84\%      & 2.22e-16    & 1.11e-16    & 3.114   \\ 
\bottomrule
\end{tabular}
\end{table}
 
 We further compare {SNSQP} with two greedy methods, {SpanCCA} and {SWCCA}, to examine the effect of the sparsity level $s$. As shown in Figure \ref{fig: SCCA real data}, the correlation computed by all three methods exhibits an increasing trend as $s$ grows.  {SNSQP} consistently achieves the highest correlation level.  It is noteworthy that {SWCCA} is very fast for this example, as it solves the problem with identity covariance matrices, i.e., $\boldsymbol{\Sigma}{xx}={\bf I}$ and $\boldsymbol{\Sigma}{yy}={\bf I}$, that enable closed-form solutions to subproblems. However, due to this simplification, {SWCCA} fails to yield desirable correlation.
 
 \begin{figure}[!h]
    \centering
        \includegraphics[width=0.46\textwidth]{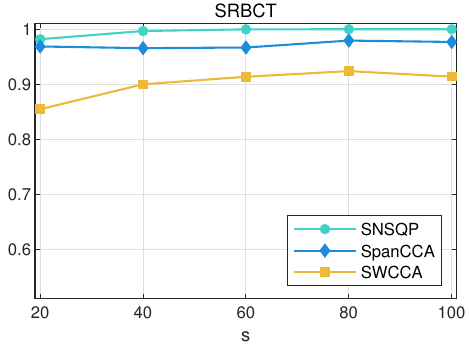}~~~ \qquad  
        \includegraphics[width=0.46\textwidth]{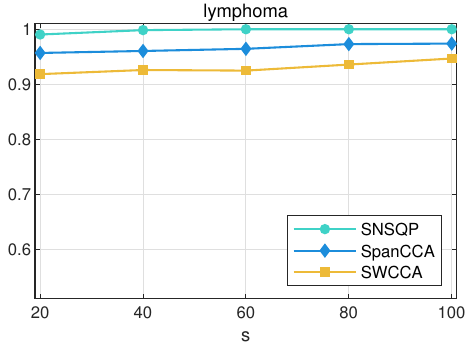} \\[1ex] 
        \includegraphics[width=0.46\textwidth]{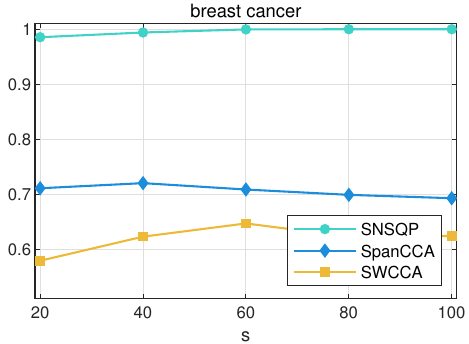}~~~ \qquad  
        \includegraphics[width=0.46\textwidth]{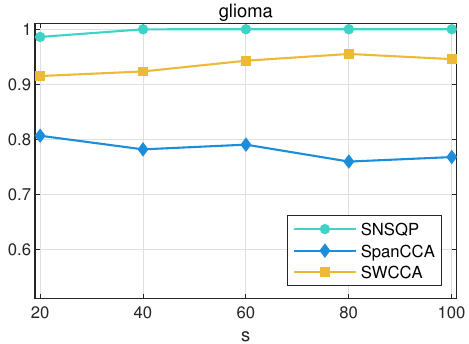} 
    \caption{Effect of sparsity level $s$ for Example \ref{example: SCCA1}.}
    \label{fig: SCCA real data}
\end{figure}

\subsection{Sparse portfolio selection problem}
The sparse portfolio selection (SPS) problem proposed in \cite{cui2013convex} is formulated as follows,
\begin{equation}
	\label{eq: SPS}
	\begin{aligned} 
		\min ~& \langle \bx, (\Q+\Q_1) \bx \rangle\\
		\text { s.t. } ~& \langle \bx, \Q_1 \bx  \rangle \leq  \sigma_0,~\langle \bx, \bm{\ba_1} \rangle \geq  r_0, ~ \langle \bx, \bm{1}\rangle  =1, \\
		&   \|\bx\|_0 \leq  s ,~\bx \in\{0\} \cup\left[\ba, \bb\right],
	\end{aligned} \tag{SPS}
\end{equation}
where $\Q \in \mathbb{R}^{n\times n}$ is a symmetric positive semi-definite matrix, $\D \in \mathbb{R}^{n\times n}$ is a non-negative diagonal matrix, $\ba_1 \in \mathbb{R}^n$ is the return vector, $\sigma_0$ is the prescribed nonsystematic risk level and $r_0$ is the prescribed weekly return level, and $\ba,\bb \in \mathbb{R}^n$ are the lower and upper bounds of $\bx$. In the objection function,  first part $\langle \bx,  \Q \bx \rangle$ is called `systematic risk' and second part $\langle \bx, \Q_1 \bx \rangle$ is called `nonsystematic/specific risk'. In the sequel, we set set  $(\sigma_0,r_0)=(0.001,0.002)$, $\ba=\bm{0}$, and $\bb=(0.3,0.3,\ldots,0.3)^\top$. We note that the formulation in \eqref{eq: SPS} includes an equality constraint. Following the approach in LNA \cite{zhao2022}, we handle this constraint by directly incorporating it into the stationary equations for Newton's method iterations, rather than using the NCP function.

\subsubsection{Benchmark methods}
We compare {SNSQP} with {GUROBI} and {SALM} \cite{bai2016splitting}. {GUROBI} is used to solve a reformulation of (\ref{eq: SPS}), where the sparsity constraint is replaced by the ones in \eqref{eq: MIP of SCQP}. Its maximal runtime is set to one hour. For {SALM}, we configure the step size as $\omega = 0.3$ for Example \ref{example: SPS1} and $\omega = 1$ for Example \ref{example: SPS2}. The initial Lagrange multiplier and penalty parameter are set to $\blambda^0 = 0$  and   $\rho = 1$. For {SNSQP}, we set $\tau = 1$. Both algorithms terminate when either $\ell>1000$ or the tolerance condition $|\cdot| \leq 10^{-6}$ is met. Since {{GUROBI}} enables high-quality approximation of the global solution, we use the relative error defined by ${\rm Relerr}=|f-f_{\rm MIP}|/f_{\rm MIP}$ to measure the accuracy of {SNSQP} and {SALM}, where $f$ is their objective function values and $f_{\rm MIP}$ is the one obtained by  GUROBI. 

\subsubsection{Testing examples} 
\begin{example}[{\bf Portfolio datasets}] The dataset are drawn from the Standard and Poor's 500 (SP500) and Russell 2000 indices. Specifically, we utilize weekly return data for 468 stocks from the SP500 and 873 stocks from the Russell 2000 over the period from 2015 to 2020. The matrices $\Q$ and $\Q_1$ are derived from factor models constructed for these datasets. In addition, we randomly select $n$ stocks from the entire stock pool.
	\label{example: SPS1}
\end{example}

\begin{example}[\textbf{Synthetic data}]
\label{example: SPS2} Slightly higher-dimensional datasets are generated as follows. Let $\Q_0 = \D^\top\D$ with $\D \in \mathbb{R}^{n/4 \times n}$. Each entry of $\D$ and each diagonal entry of $\Q_1$ are independently sampled from a uniform distribution over $[0,0.01]$, while each entry of $\ba_1$ is drawn from a normal distribution $\mathcal{N}(0,0.5^2)$. We vary $n$ over set $\{1000, 1500, \dots, 3000\}$.
\end{example}

\begin{table}[!t]
	\renewcommand{\arraystretch}{0.95}\addtolength{\tabcolsep}{0.0pt}
	\centering 
	\caption{{Effects of initial points for Example \ref{example: SPS1} using the SP500 dataset with $(n,s)=(100,10)$.}}  
	\begin{tabular}{ccccccccccc} 
		\toprule 
		&\multirow{2}{*} & \multicolumn{4}{c}{ Relerr} & \multicolumn{4}{c}{Fval}  \\ 
		\cmidrule(lr){3-6} \cmidrule(lr){7-10}  
		Initials	 &Algs.&  {min} & {median} & {max} & variance &   {min} & {median} & {max} & variance \\ 
		\midrule  
		\multirow{2}{*}{$U[0,1]$}& SNSQP &  {0.118} & 0.141 & 0.292 & 0.003 & 2.21e-04 & 2.26e-04 & 2.55e-04 & 1.10e-10  \\
		&SALM &  {0.981} & 1.983 & 2.679 & 0.242 & 4.00e-04 & 7.00e-04 & 8.73e-04 & 1.36e-08  \\
		\midrule 
		\multirow{2}{*}{$\mathcal{N}(0,1)$}&SNSQP &  {0.080} & 0.111 & 0.203 & 0.002 & 2.13e-04 & 2.20e-04 & 2.38e-04 & 6.08e-11  \\
		&SALM &  {0.461} & 1.276 & 4.351 & 1.833 & 3.47e-04 & 6.00e-04 & 1.27e-03 & 1.03e-07  \\
		\midrule  
		\multirow{2}{*}{Weibull}&SNSQP &  {0.055} & 0.132 & 0.205 & 0.003 & 2.08e-04 & 2.23e-04 & 2.38e-04 & 1.04e-10  \\
		&SALM &  {0.401} & 1.341 & 2.404 & 0.581 & 3.32e-04 & 5.56e-04 & 8.08e-04 & 3.27e-08  \\
		\midrule 
		\multirow{2}{*}{$t_{10}$}&SNSQP &  {0.078} & 0.124 & 0.260 & 0.002 & 2.13e-04 & 2.22e-04 & 2.49e-04 & 1.08e-10  \\
		&SALM &  {0.357} & 1.276 & 2.786 & 0.521 & 3.27e-04 & 6.00e-04 & 8.99e-04 & 2.94e-08  \\
		\midrule 
		\multirow{2}{*}{CVX}&SNSQP &  {0.012} & 0.012 & 0.012 & 0.000 & 1.98e-04 & 1.98e-04 & 1.98e-04 & 0.00e-00  \\
		&SALM &  {0.053} & 0.053 & 0.053 & 0.000 & 2.50e-04 & 2.50e-04 & 2.50e-04 &  0.00e-00  \\
		\midrule 
		\multirow{2}{*}{Overall}&SNSQP &  {0.012} & 0.118 & 0.292 & 0.005 & 1.98e-04 & 2.21e-04 & 2.55e-04 & 2.00e-10  \\
		&SALM &  {0.053} & 1.250 & 4.351 & 1.042 & 2.50e-04 & 6.00e-04 & 1.27e-03 & 5.87e-08  \\
		\bottomrule 
	\end{tabular}
	\label{table: SPS-eff-init}   
\end{table}
\subsubsection{Numerical comparison}
\textit{a) Effect of initial points.} To see this, 
 	we consider 5 types of initializations again for both SNSQP and SALM. The first four types of starting points are generated the same as those in Section \ref{subsub-com-starting}. The last type also uses sparse vectors as initial points. Specifically, we first apply CVX to solve \eqref{example: SPS1} without the sparsity constraint to obtain a solution, and then derive a sparse initial point by retaining its $s$ largest (in absolute value) entries while setting the remaining ones to zero. Table \ref{table: SPS-eff-init} reports the best, median, and worst results over 50 independent trials for each type of initial points, as well as the results over all 250 trials. In all cases, SNSQP attains lower Relerr and Fval values than SALM. Moreover, the substantially smaller variances achieved by SNSQP indicate that it is more robust to initial points than SALM. Since the initialization obtained via CVX delivers the best overall performance, we adopt it for all subsequent numerical experiments.

\begin{table}[!t]
\renewcommand{\arraystretch}{0.95}\addtolength{\tabcolsep}{1pt}
	\centering
	\caption{Comparison with two methods for Example \ref{example: SPS1}} 
	\begin{tabular}{ccccccccccr}
\toprule 
 &  & \multicolumn{2}{c}{Relerr} & \multicolumn{3}{c}{Fval} & \multicolumn{3}{c}{Time(s)} \\
\cmidrule(lr){3-4} \cmidrule(lr){5-7} \cmidrule(lr){8-10}  
$n$	& $s$ & {SNSQP} & {SALM} & {SNSQP} & {SALM} & {GUROBI} & {SNSQP} & {SALM} & {GUROBI} \\
\midrule
\multicolumn{10}{c}{{\tt SP500}} \\\midrule
 {50}                   & 5  &  {0.032} & 0.048   & 2.88e-04 & 2.90e-04 & 2.79e-04     &  {0.011} & 0.069 & 0.336       \\
& 10 &  {0.004} & 0.008   & 2.49e-04 & 2.51e-04 & 2.48e-04     &  {0.030} & 0.190 & 0.455      \\  
 {100} & 5  &  {0.085} & 0.114    & 2.89e-04 & 2.97e-04 & 2.66e-04    &  {0.013} & 0.527 & 1.092       \\  
                     & 10 &  {0.013} & 0.024    & 2.45e-04 & 2.48e-04 & 2.42e-04    &  {0.051} & 0.996 & 2.526       \\ 
                     
 {200} & 5  &  {0.127} & 0.135    & 2.52e-04 & 2.54e-04 & 2.24e-04    &  {0.072} & 1.156 & 5.279       \\ 
                     & 10 &  {0.003} & 0.013     & 1.96e-04 & 1.98e-04 & 1.95e-04   &  {0.170} & 1.789 & 12.42      \\ 
                     
 {300} & 5  &  {0.044} & 0.086    & 2.38e-04 & 2.48e-04 & 2.28e-04    &  {0.180} & 2.880 & 29.60       \\ 
                     & 10 &  {0.060} & 0.091    & 1.90e-04 & 1.95e-04 & 1.79e-04    &  {0.199} & 3.724 & 126.88     \\ 
                     
 {400} & 5  &  {0.063} & 0.075     & 2.52e-04 & 2.55e-04 & 2.37e-04   &  {0.070} & 2.696 & 116.16     \\
                     & 10 &  {0.037} & 0.092     & 1.86e-04 & 1.96e-04 & 1.79e-04   &  {0.091} & 5.069 & 1379.6   \\  \midrule
                     \multicolumn{10}{c}{{\tt Russell2000}} \\
        \midrule
 {50}                   & 5  &  {0.016} & 0.027   & 5.22e-04 & 5.28e-04 & 5.16e-04     &  {0.020} & 0.081 & 0.233       \\
& 10 &  {0.008} & 0.150   & 4.13e-04 & 4.22e-04 & 4.09e-04     &  {0.014} & 0.043 & 0.320      \\ 

 {200} & 5  &  {0.016} & 0.039    & 2.70e-04 & 2.77e-04 & 2.66e-04    &  {0.061} & 1.421 & 11.66       \\  
                     & 10 &  {0.002} & 0.010    & 2.71e-04 & 2.73e-04 & 2.70e-04    &  {0.035} & 0.836 & 7.501       \\ 
                     
 {400} & 5  &  {0.006} & 0.035    & 2.09e-04 & 2.16e-04 & 2.08e-04    &  {0.031} & 4.023 & 804.93       \\ 
                     & 10 &  {0.019} & 0.034     & 2.08e-04 & 2.12e-04 & 2.04e-04   &  {0.088} & 5.688 & 830.44      \\ 
                     
 {600} & 5  &  {0.002} & 0.016    & 2.08e-04 & 2.10e-04 & 2.07e-04    &  {0.038} & 8.602 & 510.25       \\ 
                     & 10 &  {0.012} & 0.045    & 2.43e-04 & 2.50e-04 & 2.40e-04    &  {0.041} & 9.542 & 3593.1     \\ 
                     
{800} & 5  &  {0.000} & 0.014     & 2.04e-04 & 2.07e-04 & 2.04e-04   &  {0.052} & 15.82 & 2951.6     \\
                     & 10 &  {0.007} & 0.038     & 1.61e-04 & 1.66e-04 & 1.59e-04   &  {0.061} & 12.08 & 3600.0    \\
                     
                     \bottomrule 
\end{tabular}
	\label{table: SPS1}
\end{table}

\begin{table}[!t]
\renewcommand{\arraystretch}{0.95}\addtolength{\tabcolsep}{0.5pt}
\centering
\caption{Comparison for Example \ref{example: SPS2} in higher dimensions. }
\label{table: SPS2}
\begin{tabular}{ccccc ccc ccrr}
 \hline
 & \multicolumn{5}{c}{$s = 5$} &&  \multicolumn{5}{c}{$s = 10$} \\\cline{2-6} \cline{8-12} 
  & \multicolumn{2}{c}{Fval}  && \multicolumn{2}{c}{Time(s)}&&\multicolumn{2}{c}{Fval}  && \multicolumn{2}{c}{Time(s)} \\
  \cline{2-3}\cline{5-6}\cline{8-9}\cline{11-12}
$n$ & {SNSQP}& {SALM} && {SNSQP}& {SALM} && {SNSQP}& {SALM} && {SNSQP}& {SALM}\\\hline
1000  &  {5.83e-03}& 5.84e-03 &&  {0.012}& 13.18 &&  {5.69e-03}& 5.71e-03      &&  {0.029}& 8.280 \\
1500 &  {8.43e-03}& 8.46e-03      &&  {0.042}& 46.31 &&  {8.47e-03}& 8.49e-03      &&  {0.047}& 29.92 \\
2000 &  {1.16e-02}& 1.17e-02      &&  {0.030}& 104.6 &&  {1.14e-02}& 1.15e-02      &&  {0.032}& 91.58 \\
2500 &  {1.50e-02}& 1.51e-02      &&  {0.032}& 219.6 &&  {1.46e-02}& 1.48e-02      &&  {0.035}& 180.0 \\
3000 &  {1.80e-02}& 1.82e-02      &&  {0.039}& 443.4 &&  {1.75e-02}& 1.76e-02      &&  {0.042}& 406.8 \\
\hline
\end{tabular}
\end{table}

\textit{b) Performance of solving SPS problems using low dimensional datasets.} We compare {SNSQP} with {SALM} and GUROBI for solving Example \ref{example: SPS1}.	As shown in Table \ref{table: SPS1}, {SNSQP} consistently achieves lower Relerr and Fval than {SALM} across all cases for both datasets, indicating superior solution quality. Additionally, {SNSQP} demonstrates a significant speed advantage over its counterparts in every scenario, particularly in high-dimensional cases. For instance, on the Russell2000 dataset with $(n,s)=(800,5)$, the computational time of {SNSQP}, {SALM}, and {GUROBI} is 0.052, 15.82, and 2951.6 seconds, respectively. 

\textit{c)  Performance of solving SPS problems using higher-dimensional datasets.} We examine the numerical comparison for slightly higher-dimensional scenarios based on Example \ref{example: SPS2}. Since {{GUROBI}} requires a significantly longer time to solve the problem, it is excluded from the comparison. The results in Table \ref{table: SPS2} show that {SNSQP} not only achieves lower Fval but also runs much faster than {SALM} across all test instances.

\section{Conclusion} 
It is known that SQCQP is a computationally challenging problem, particularly in large-scale or high-dimensional settings. In contrast to existing methods, which primarily focus on solving its mixed-integer programming reformulations or relaxations, we introduce a novel paradigm by designing an efficient semismooth Newton-type algorithm, SNSQP, that directly tackles SQCQP. The key innovation of our approach lies in the formulation of a newly introduced system of stationary equations, which characterizes the optimality conditions of the original problem. The algorithm is classified as a second-order method and thus exhibits a locally quadratic convergence rate while maintaining relatively low computational complexity due to the sparse structure of the solution. Extensive numerical experiments demonstrate that the algorithm consistently produces high-accuracy solutions with fast computational speed. However, given the challenges in establishing global convergence, the method remains local. Considering its strong numerical performance and extensive applications, ensuring its global convergence is worthy of further research.

\section*{Acknowledgments}
This work is supported by the National Key R\&D Program of China (2023YFA1011100) and the National Natural Science Foundation of China (No. 12271022).

\bibliographystyle{abbrv}
\bibliography{references}

\begin{thebibliography}{10}

\bibitem{ahn2017difference}
M.~Ahn, J.-S. Pang, and J.~Xin.
\newblock Difference-of-convex learning: directional stationarity, optimality,
  and sparsity.
\newblock {\em SIAM Journal on Optimization}, 27(3):1637--1665, 2017.

\bibitem{asteris2016simple}
M.~Asteris, A.~Kyrillidis, O.~Koyejo, and R.~Poldrack.
\newblock A simple and provable algorithm for sparse diagonal cca.
\newblock In {\em International Conference on Machine Learning}, pages
  1148--1157. PMLR, 2016.

\bibitem{GraSP}
S.~Bahmani, B.~Raj, and P.~T. Boufounos.
\newblock Greedy sparsity-constrained optimization.
\newblock {\em The Journal of Machine Learning Research}, 14(1):807--841, 2013.

\bibitem{bai2016splitting}
Y.~Bai, R.~Liang, and Z.~Yang.
\newblock Splitting augmented lagrangian method for optimization problems with
  a cardinality constraint and semicontinuous variables.
\newblock {\em Optimization Methods and Software}, 31(5):1089--1109, 2016.

\bibitem{bauschke2014restricted}
H.~H. Bauschke, D.~R. Luke, H.~M. Phan, and X.~Wang.
\newblock Restricted normal cones and sparsity optimization with affine
  constraints.
\newblock {\em Foundations of Computational Mathematics}, 14:63--83, 2014.

\bibitem{beck2013sparsity}
A.~Beck and Y.~C. Eldar.
\newblock Sparsity constrained nonlinear optimization: Optimality conditions
  and algorithms.
\newblock {\em SIAM Journal on Optimization}, 23(3):1480--1509, 2013.

\bibitem{beck2016minimization}
A.~Beck and N.~Hallak.
\newblock On the minimization over sparse symmetric sets: projections,
  optimality conditions, and algorithms.
\newblock {\em Mathematics of Operations Research}, 41(1):196--223, 2016.

\bibitem{beck2016sparse}
A.~Beck and Y.~Vaisbourd.
\newblock The sparse principal component analysis problem: Optimality
  conditions and algorithms.
\newblock {\em Journal of Optimization Theory and Applications}, 170:119--143,
  2016.

\bibitem{bertsimas2009algorithm}
D.~Bertsimas and R.~Shioda.
\newblock Algorithm for cardinality-constrained quadratic optimization.
\newblock {\em Computational Optimization and Applications}, 43(1):1--22, 2009.

\bibitem{CGIHT}
J.~D. Blanchard, J.~Tanner, and K.~Wei.
\newblock Cgiht: conjugate gradient iterative hard thresholding for compressed
  sensing and matrix completion.
\newblock {\em Information and Inference: A Journal of the IMA}, 4(4):289--327,
  2015.

\bibitem{GP}
T.~Blumensath and M.~E. Davies.
\newblock Gradient pursuits.
\newblock {\em IEEE Transactions on Signal Processing}, 56(6):2370--2382, 2008.

\bibitem{IHT}
T.~Blumensath and M.~E. Davies.
\newblock Iterative hard thresholding for compressed sensing.
\newblock {\em Applied and computational harmonic analysis}, 27(3):265--274,
  2009.

\bibitem{plsgenomics}
A.-L. Boulesteix, G.~Durif, S.~Lambert-Lacroix, J.~Peyre, and K.~Strimmer.
\newblock {\em plsgenomics: PLS Analyses for Genomics}, 2024.
\newblock R package version 1.5-3.

\bibitem{bredel2005functional}
M.~Bredel, C.~Bredel, D.~Juric, G.~R. Harsh, H.~Vogel, L.~D. Recht, and B.~I.
  Sikic.
\newblock Functional network analysis reveals extended gliomagenesis pathway
  maps and three novel myc-interacting genes in human gliomas.
\newblock {\em Cancer research}, 65(19):8679--8689, 2005.

\bibitem{bredel2005high}
M.~Bredel, C.~Bredel, D.~Juric, G.~R. Harsh, H.~Vogel, L.~D. Recht, and B.~I.
  Sikic.
\newblock High-resolution genome-wide mapping of genetic alterations in human
  glial brain tumors.
\newblock {\em Cancer research}, 65(10):4088--4096, 2005.

\bibitem{burdakov2016mathematical}
O.~P. Burdakov, C.~Kanzow, and A.~Schwartz.
\newblock Mathematical programs with cardinality constraints: reformulation by
  complementarity-type conditions and a regularization method.
\newblock {\em SIAM Journal on Optimization}, 26(1):397--425, 2016.

\bibitem{KODAMA}
S.~Cacciatore and L.~Tenori.
\newblock {\em KODAMA: Knowledge Discovery by Accuracy Maximization}, 2023.
\newblock R package version 2.4.

\bibitem{vcervinka2016constraint}
M.~{\v{C}}ervinka, C.~Kanzow, and A.~Schwartz.
\newblock Constraint qualifications and optimality conditions for optimization
  problems with cardinality constraints.
\newblock {\em Mathematical Programming}, 160:353--377, 2016.

\bibitem{cesarone2013new}
F.~Cesarone, A.~Scozzari, and F.~Tardella.
\newblock A new method for mean-variance portfolio optimization with
  cardinality constraints.
\newblock {\em Annals of Operations Research}, 205:213--234, 2013.

\bibitem{chang2000heuristics}
T.-J. Chang, N.~Meade, J.~E. Beasley, and Y.~M. Sharaiha.
\newblock Heuristics for cardinality constrained portfolio optimisation.
\newblock {\em Computers \& Operations Research}, 27(13):1271--1302, 2000.

\bibitem{chartrand2007exact}
R.~Chartrand.
\newblock Exact reconstruction of sparse signals via nonconvex minimization.
\newblock {\em IEEE Signal Processing Letters}, 14(10):707--710, 2007.

\bibitem{chu2013sparse}
D.~Chu, L.-Z. Liao, M.~K. Ng, and X.~Zhang.
\newblock Sparse canonical correlation analysis: New formulation and algorithm.
\newblock {\em IEEE Transactions on Pattern Analysis and Machine Intelligence},
  35(12):3050--3065, 2013.

\bibitem{clarke1990optimization}
F.~H. Clarke.
\newblock {\em Optimization and nonsmooth analysis}.
\newblock SIAM, 1990.

\bibitem{cui2013convex}
X.~Cui, X.~Zheng, S.~Zhu, and X.~Sun.
\newblock Convex relaxations and miqcqp reformulations for a class of
  cardinality-constrained portfolio selection problems.
\newblock {\em Journal of Global Optimization}, 56(4):1409--1423, 2013.

\bibitem{de1996semismooth}
T.~De~Luca, F.~Facchinei, and C.~Kanzow.
\newblock A semismooth equation approach to the solution of nonlinear
  complementarity problems.
\newblock {\em Mathematical Programming}, 75:407--439, 1996.

\bibitem{di2012concave}
D.~Di~Lorenzo, G.~Liuzzi, F.~Rinaldi, F.~Schoen, and M.~Sciandrone.
\newblock A concave optimization-based approach for sparse portfolio selection.
\newblock {\em Optimization Methods and Software}, 27(6):983--1000, 2012.

\bibitem{dong2019structural}
H.~Dong, M.~Ahn, and J.-S. Pang.
\newblock Structural properties of affine sparsity constraints.
\newblock {\em Mathematical Programming}, 176:95--135, 2019.

\bibitem{facchinei2003finite}
F.~Facchinei and J.-S. Pang.
\newblock {\em Finite-dimensional variational inequalities and complementarity
  problems}.
\newblock Springer, 2003.

\bibitem{facchinei1997new}
F.~Facchinei and J.~Soares.
\newblock A new merit function for nonlinear complementarity problems and a
  related algorithm.
\newblock {\em SIAM Journal on Optimization}, 7(1):225--247, 1997.

\bibitem{fan2001variable}
J.~Fan and R.~Li.
\newblock Variable selection via nonconcave penalized likelihood and its oracle
  properties.
\newblock {\em Journal of the American Statistical Association},
  96(456):1348--1360, 2001.

\bibitem{feng2013complementarity}
M.~Feng, J.~E. Mitchell, J.-S. Pang, X.~Shen, and A.~W{\"a}chter.
\newblock Complementarity formulations of l0-norm optimization problems.
\newblock {\em Industrial Engineering and Management Sciences. Technical
  Report. Northwestern University, Evanston, IL, USA}, 5, 2013.

\bibitem{fischer1992special}
A.~Fischer.
\newblock A special newton-type optimization method.
\newblock {\em Optimization}, 24(3-4):269--284, 1992.

\bibitem{hamza2019hybrid}
S.~A. Hamza and M.~G. Amin.
\newblock Hybrid sparse array beamforming design for general rank signal
  models.
\newblock {\em IEEE Transactions on Signal Processing}, 67(24):6215--6226,
  2019.

\bibitem{hardoon2011sparse}
D.~R. Hardoon and J.~Shawe-Taylor.
\newblock Sparse canonical correlation analysis.
\newblock {\em Machine Learning}, 83:331--353, 2011.

\bibitem{hesse2014alternating}
R.~Hesse, D.~R. Luke, and P.~Neumann.
\newblock Alternating projections and douglas-rachford for sparse affine
  feasibility.
\newblock {\em IEEE Transactions on Signal Processing}, 62(18):4868--4881,
  2014.

\bibitem{huang2023sparse}
H.~Huang, H.~C. So, and A.~M. Zoubir.
\newblock Sparse array beamformer design via admm.
\newblock {\em IEEE Transactions on Signal Processing}, 2023.

\bibitem{kanzow2021augmented}
C.~Kanzow, A.~B. Raharja, and A.~Schwartz.
\newblock An augmented lagrangian method for cardinality-constrained
  optimization problems.
\newblock {\em Journal of Optimization Theory and Applications},
  189(3):793--813, 2021.

\bibitem{kanzow2022sparse}
C.~Kanzow, A.~Schwartz, and F.~Wei{\ss}.
\newblock The sparse(st) optimization problem: reformulations, optimality,
  stationarity, and numerical results.
\newblock {\em Computational Optimization and Applications}, 90:77--112, 2025.

\bibitem{kyrillidis2013sparse}
A.~Kyrillidis, S.~Becker, V.~Cevher, and C.~Koch.
\newblock Sparse projections onto the simplex.
\newblock In {\em International Conference on Machine Learning}, pages
  235--243. PMLR, 2013.

\bibitem{li2020matrix}
X.~Li, N.~Xiu, and S.~Zhou.
\newblock Matrix optimization over low-rank spectral sets: stationary points
  and local and global minimizers.
\newblock {\em Journal of Optimization Theory and Applications}, 184:895--930,
  2020.

\bibitem{lindenbaum2021l0}
O.~Lindenbaum, M.~Salhov, A.~Averbuch, and Y.~Kluger.
\newblock L0-sparse canonical correlation analysis.
\newblock In {\em International Conference on Learning Representations}, 2021.

\bibitem{lu2015optimization}
Z.~Lu.
\newblock Optimization over sparse symmetric sets via a nonmonotone projected
  gradient method.
\newblock {\em arXiv preprint arXiv:1509.08581}, 2015.

\bibitem{lu2013sparse}
Z.~Lu and Y.~Zhang.
\newblock Sparse approximation via penalty decomposition methods.
\newblock {\em SIAM Journal on Optimization}, 23(4):2448--2478, 2013.

\bibitem{min2018sparse}
W.~Min, J.~Liu, and S.~Zhang.
\newblock Sparse weighted canonical correlation analysis.
\newblock {\em Chinese Journal of Electronics}, 27(3):459--466, 2018.

\bibitem{CoSaMP}
D.~Needell and J.~A. Tropp.
\newblock Cosamp: Iterative signal recovery from incomplete and inaccurate
  samples.
\newblock {\em Applied and Computational Harmonic Analysis}, 26(3):301--321,
  2009.

\bibitem{pan2025efficient}
J.~Pan and M.~Yan.
\newblock Efficient sparse probability measures recovery via bregman gradient.
\newblock {\em Journal of Scientific Computing}, 102(3):66, 2025.

\bibitem{pan2017optimality}
L.~Pan, N.~Xiu, and J.~Fan.
\newblock Optimality conditions for sparse nonlinear programming.
\newblock {\em Science China Mathematics}, 60:759--776, 2017.

\bibitem{pan2017convergent}
L.~Pan, S.~Zhou, N.~Xiu, and H.-D. Qi.
\newblock A convergent iterative hard thresholding for nonnegative sparsity
  optimization.
\newblock {\em Pacific Journal of Optimization}, 13(2):325--353, 2017.

\bibitem{pan2015solutions}
L.-L. Pan, N.-H. Xiu, and S.-L. Zhou.
\newblock On solutions of sparsity constrained optimization.
\newblock {\em Journal of the Operations Research Society of China},
  3:421--439, 2015.

\bibitem{qi1997semismooth}
L.~Qi and H.~Jiang.
\newblock Semismooth karush-kuhn-tucker equations and convergence analysis of
  newton and quasi-newton methods for solving these equations.
\newblock {\em Mathematics of Operations Research}, 22(2):301--325, 1997.

\bibitem{rockafellar2009variational}
R.~T. Rockafellar and R.~J.-B. Wets.
\newblock {\em Variational analysis}, volume 317.
\newblock Springer Science \& Business Media, 2009.

\bibitem{sriperumbudur2011majorization}
B.~K. Sriperumbudur, D.~A. Torres, and G.~R. Lanckriet.
\newblock A majorization-minimization approach to the sparse generalized
  eigenvalue problem.
\newblock {\em Machine Learning}, 85:3--39, 2011.

\bibitem{steffensen2024relaxation}
S.~Steffensen.
\newblock Relaxation approaches for nonlinear sparse optimization problems.
\newblock {\em Optimization}, 73(10):3237--3258, 2024.

\bibitem{steffensen2010new}
S.~Steffensen and M.~Ulbrich.
\newblock A new relaxation scheme for mathematical programs with equilibrium
  constraints.
\newblock {\em SIAM Journal on Optimization}, 20(5):2504--2539, 2010.

\bibitem{streichert2004evolutionary}
F.~Streichert, H.~Ulmer, and A.~Zell.
\newblock Evolutionary algorithms and the cardinality constrained portfolio
  optimization problem.
\newblock In {\em Operations research proceedings 2003: Selected papers of the
  international conference on operations research (OR 2003) Heidelberg,
  September 3--5, 2003}, pages 253--260. Springer, 2004.

\bibitem{sun1999ncp}
D.~Sun and L.~Qi.
\newblock On ncp-functions.
\newblock {\em Computational Optimization and Applications}, 13(1):201--220,
  1999.

\bibitem{sun2023gradient}
J.~Sun, L.~Kong, and S.~Zhou.
\newblock Gradient projection newton algorithm for sparse collaborative
  learning using synthetic and real datasets of applications.
\newblock {\em Journal of Computational and Applied Mathematics}, 422:114872,
  2023.

\bibitem{OMP}
J.~A. Tropp and A.~C. Gilbert.
\newblock Signal recovery from random measurements via orthogonal matching
  pursuit.
\newblock {\em IEEE Transactions on information theory}, 53(12):4655--4666,
  2007.

\bibitem{wachsmuth2013licq}
G.~Wachsmuth.
\newblock On {LICQ} and the uniqueness of lagrange multipliers.
\newblock {\em Operations Research Letters}, 41(1):78--80, 2013.

\bibitem{wang2021extended}
R.~Wang, N.~Xiu, and S.~Zhou.
\newblock An extended newton-type algorithm for $\ell_2$-regularized sparse
  logistic regression and its efficiency for classifying large-scale datasets.
\newblock {\em Journal of Computational and Applied Mathematics}, 397:113656,
  2021.

\bibitem{PMA}
D.~Witten, R.~Tibshirani, S.~Gross, and B.~Narasimhan.
\newblock {\em PMA: Penalized Multivariate Analysis}, 2024.
\newblock R package version 1.2-3.

\bibitem{xiao2022geometric}
G.~Xiao and Z.-J. Bai.
\newblock A geometric proximal gradient method for sparse least squares
  regression with probabilistic simplex constraint.
\newblock {\em Journal of Scientific Computing}, 92(1):22, 2022.

\bibitem{GraHTP}
X.-T. Yuan, P.~Li, and T.~Zhang.
\newblock Gradient hard thresholding pursuit.
\newblock {\em Journal of Machine Learning Research}, 18(166):1--43, 2018.

\bibitem{NTGP}
X.-T. Yuan and Q.~Liu.
\newblock Newton greedy pursuit: A quadratic approximation method for
  sparsity-constrained optimization.
\newblock In {\em Proceedings of the IEEE conference on computer vision and
  pattern recognition}, pages 4122--4129, 2014.

\bibitem{zhao2022}
C.~Zhao, N.~Xiu, H.~Qi, and Z.~Luo.
\newblock A {L}agrange--{N}ewton algorithm for sparse nonlinear programming.
\newblock {\em Mathematical Programming}, 195(1-2):903--928, 2022.

\bibitem{GPNP}
S.~Zhou.
\newblock Gradient projection newton pursuit for sparsity constrained
  optimization.
\newblock {\em Applied and Computational Harmonic Analysis}, 61:75--100, 2022.

\bibitem{NHTP}
S.~Zhou, N.~Xiu, and H.-D. Qi.
\newblock Global and quadratic convergence of newton hard-thresholding pursuit.
\newblock {\em Journal of Machine Learning Research}, 22(12):1--45, 2021.

\end{thebibliography}


\end{document}